\documentclass[11pt]{amsart} 
\calclayout

\usepackage{setspace} 
\usepackage{graphicx}
\usepackage{microtype}

\usepackage{pgfplots}
\pgfplotsset{compat=1.15}
\usepackage{tikz}
\usepackage{mathrsfs}
\usetikzlibrary{arrows}

\usepackage[english]{babel}
\usepackage{csquotes}
\usepackage{amssymb}
\usepackage{amsmath}
\usepackage{amsthm}
\usepackage{float,comment}
\usepackage{wrapfig}
\usepackage{stmaryrd}
\usepackage{slashed}
\usepackage{tikz-cd}
\tikzcdset{scale cd/.style={every label/.append style={scale=#1},
    cells={nodes={scale=#1}}}}
\usetikzlibrary{positioning}
\usetikzlibrary{calc,intersections}
\usepackage{hyperref}
\usepackage{todonotes}
\usepackage[all,cmtip]{xy}

\usepackage[justification=centering]{caption}

\newtheorem{thm}{Theorem}[section]
\newtheorem{lem}[thm]{Lemma}
\newtheorem{prop}[thm]{Proposition}
\newtheorem{coro}[thm]{Corollary}

\theoremstyle{remark}
\newtheorem{rema}[thm]{Remark}
\newtheorem{exa}[thm]{Example}
\newtheorem{defi}[thm]{Definition}

\DeclareFontFamily{U} {MnSymbolA}{}
\DeclareFontShape{U}{MnSymbolA}{m}{n}{
  <-6> MnSymbolA5
  <6-7> MnSymbolA6
  <7-8> MnSymbolA7
  <8-9> MnSymbolA8
  <9-10> MnSymbolA9
  <10-12> MnSymbolA10
  <12-> MnSymbolA12}{}
\DeclareFontShape{U}{MnSymbolA}{b}{n}{
  <-6> MnSymbolA-Bold5
  <6-7> MnSymbolA-Bold6
  <7-8> MnSymbolA-Bold7
  <8-9> MnSymbolA-Bold8
  <9-10> MnSymbolA-Bold9
  <10-12> MnSymbolA-Bold10
  <12-> MnSymbolA-Bold12}{}

\DeclareSymbolFont{MnSyA} {U} {MnSymbolA}{m}{n}
\DeclareMathSymbol{\lcirclearrowright}{\mathrel}{MnSyA}{252}
\DeclareMathSymbol{\rcirclearrowleft}{\mathrel}{MnSyA}{250}

\DeclareMathOperator\ad{ad}
\DeclareMathOperator\Ad{Ad}

\DeclareMathOperator{\Lie}{Lie}

\DeclareMathOperator{\im}{im}
\DeclareMathOperator{\pr}{pr}
\DeclareMathOperator{\id}{Id}

\DeclareMathOperator{\al}{\mathit{A}}
\DeclareMathOperator{\gra}{Graph}

\newcommand{\cD}{\mathcal{D}}
\newcommand{\cE}{\mathcal{E}}

\newcommand{\cG}{\mathcal{G}}
\newcommand{\cH}{\mathcal{H}}

\newcommand{\cT}{\mathcal{T}}
\newcommand{\cV}{\mathcal{V}}

\newcommand{\bR}{\mathbb{R}}
\newcommand{\bT}{\mathbb{T}}

\newcommand{\fd}{\mathfrak{d}}
\newcommand{\g}{\mathfrak{g}}
\newcommand{\h}{\mathfrak{h}}
\newcommand{\fc}{\mathfrak{c}}
\newcommand{\fk}{\mathfrak{k}}
\newcommand{\fl}{\mathfrak{l}}
\newcommand{\fX}{\mathfrak{X}}

\newcommand{\gs}{\mathtt{s}}
\newcommand{\gt}{\mathtt{t}}
\newcommand{\gu}{\mathtt{\epsilon}}
\newcommand{\gi}{\mathtt{i}}
\newcommand{\gm}{\mathtt{m}}
\newcommand{\ga}{\mathtt{a}}
\newcommand{\gb}{\mathtt{b}}

\newcommand{\gr}{\mathtt{r}}
\newcommand{\gl}{\mathtt{l}}

\newcommand{\bl}{\mathbf{L}}
\newcommand{\bc}{\mathbf{C}}

\newcommand{\gro}[2]{\mathcal{ #1}\rightrightarrows #2 }
\newcommand{\CAg}[3]{ \bT^{#2}_{#3}{#1}}

\newcommand{\rank} 
{\operatornamewithlimits{rank}}

\title{Shifted lagrangian structures in Poisson geometry}

\author{Daniel \'{A}lvarez}

\address{Instituto de Matem\'atica Pura e Aplicada,
Estrada Dona Castorina 110, Rio de Janeiro, 22460-320, Brazil}

\email{uerbum@impa.br}

\author{Henrique Bursztyn}

\address{Instituto de Matem\'atica Pura e Aplicada,
Estrada Dona Castorina 110, Rio de Janeiro, 22460-320, Brazil }

\email{henrique@impa.br}

\author{Miquel Cueca}

\address{Departement of Mathematics, KU Leuven. Celestijnenlaan 200B, Leuven (Heverlee), B-3001,
Belgium}

\email{miquel.cuecaten@kuleuven.be}

\date{} % Delete this line to display the current date

\begin{document}

\begin{abstract}   
This paper develops new aspects of the interplay between shifted symplectic geometry and classical Poisson geometry, focusing on lagrangian morphisms into 2-shifted symplectic groups. We establish a Lie-type correspondence between such morphisms and Dirac structures in transitive Courant algebroids given by the product of an exact Courant algebroid and a quadratic Lie algebra. As a key application, we identify the global objects integrating quasi-Poisson manifolds, which we call {\em multiplicative $D$-valued moment maps}; this extends the integration of Poisson manifolds to symplectic groupoids and the lifting of Poisson actions to multiplicative hamiltonian actions.
We devise systematic constructions of quasi-symplectic groupoids via fibred products of 2-shifted lagrangians, extending classical reduction procedures. This places known constructions, such as the integrations of Poisson homogeneous spaces and Poisson quotients, into a broader, conceptual framework, while yielding  new examples.
\end{abstract}

\maketitle

\tableofcontents 

\section{Introduction}

Shifted symplectic geometry \cite{ptvv} is a far-reaching extension of symplectic geometry to the context of higher and derived geometry with applications to TQFT \cite{aksz,cal:AKSZ}. In this paper, we use techniques from shifted symplectic geometry to develop effective methods for the study of classical Poisson geometry and related structures.

A fundamental aspect of Poisson geometry is its close relation with Lie algebroids and Lie groupoids (see e.g. \cite{cosgro}). Important manifestations of this connection are the Lie-theoretic correspondences, via differentiation and integration, between Poisson manifolds and symplectic groupoids \cite{macxu2}, and, more generally, between Dirac structures (in exact Courant algebroids) and quasi-symplectic groupoids (a.k.a. twisted presymplectic groupoids) \cite{burint}. These results provide a first gateway to shifted symplectic geometry.

Specifically, there are well-established connections between Poisson geometry and {\em 1-shifted} symplectic geometry. For instance, 1-shifted symplectic groupoids are known to be the same as quasi-symplectic groupoids \cite{cal:lag, Lesdiablerets}, and several versions of hamiltonian reduction can be interpreted, and naturally generalized, in terms of 1-shifted lagrangian fibred products \cite{cal:lag,cro:red,saf:qua}; see also \cite{bon:shif, pym:symalg,saf:poi}.

This paper presents further connections in the 2-shifted setting. While 2-shifted symplectic structures on Lie groups have appeared in many contexts, e.g. \cite{polwie,shuphd,weisymmod}, this work focuses on the yet unexplored  {\em lagrangian structures} on groupoid morphisms to 2-shifted symplectic groups. Concrete applications to Poisson geometry include an extension of  
the main result of \cite{burint} -- on the integration of Dirac structures in exact Courant algebroids -- to Dirac structures in more general transitive Courant algebroids as well as 
systematic constructions of (quasi-)symplectic groupoids, extending classical reduction procedures in \cite{ferigl,mommap, mikwei} and constructions in \cite{intpoihom}. These results also yield new tools for the study of quasi-Poisson manifolds  \cite{alekos,alemeikos}, particularly those arising from representation varieties.
We detail our results below.

\subsection*{Main results and structure of the paper}

Section \S \ref{sec:general} briefly reviews background material on shifted symplectic groupoids \cite{Lesdiablerets} (see also \cite{cueca}). Building on the general theory of \cite{cal:lag,ptvv,saf:qua}, we develop an adapted, workable definition of lagrangian structures on groupoid morphisms (Definition \ref{def:gen-lag}) and study their fibred products. 
We illustrate the framework in the 0- and 1-shifted settings, including the infinitesimal picture and associated Lie theory, together with key examples related to Poisson geometry (see Propositions \ref{prop:0lag} and \ref{prop:1lagint}). We also highlight the description of 1-shifted lagrangian structures in terms of Courant morphisms and morphisms of Manin pairs \cite{buriglsev} (recalled in Appendix~\ref{app:coualg}), see Proposition \ref{prop:1lag}, which provides a direct connection  with moment maps and hamiltonian spaces in the context of \cite{burcra,burint,xumor}.

The main contributions of the paper begin in Section $\S$ \ref{sec:2shift}, which focuses on lagrangian morphisms in the 2-shifted setting. We consider Lie groups $D$ equipped with a 2-form $\Omega \in \Omega^2(D^2)$ and a 3-form $\Theta \in \Omega^3(D)$ forming a closed (normalized) element in the Bott–Shulman–Stasheff complex of $D$ \cite{shuphd}. Such an element determines, at the Lie-algebra level, an ad-invariant symmetric bilinear form on $\fd$ (see \S \ref{subsec:2sg}, following \cite{Lesdiablerets}). We say that $\Omega+\Theta$ defines a {\em 2-shifted symplectic form} on $D$ when the corresponding bilinear form is nondegenerate, thus inducing a quadratic structure $\langle \cdot,\cdot \rangle$ on $\fd$.

Given a Lie groupoid $\cH \rightrightarrows N$, an {\em isotropic structure} on a Lie groupoid morphism $\Phi:\mathcal{H}\to D$ is a primitive for the pullback $\Phi^*(\Omega+\Theta)$ in the Bott–Shulman–Stasheff complex of $\mathcal{H}$. (As commonly remarked in the literature, ``isotropicity'' is not a property of the morphism, but rather an extra structure, given by a specific choice of primitive). Our first main result is Theorem \ref{thm:intiso}, which establishes a Lie correspondence between such isotropic structures and infinitesimal data on the induced Lie algebroid morphism $\varphi:A_{\mathcal{H}}\to \mathfrak{d}$. The proof uses the van Est theorem \cite{weivanest} and its functoriality  (recalled in Appendix~\ref{app:wei}). This correspondence extends the main results in \cite{burcab,burint} on the infinitesimal description and associated Lie theory of multiplicative 2-forms.

A {\em lagrangian structure} on $\Phi: \cH \to D$ is a 2-shifted isotropic structure satisfying an additional ``nondegeneracy'' condition (a homotopical analogue of the classical definition of a lagrangian submanifold). We give a complete characterization of such lagrangian morphisms in Proposition~\ref{prop:lag}, making clear how they generalize quasi-symplectic groupoids, with the failure of their usual properties controlled by the pullback $\Phi^*(\Omega+\Theta)$; it is convenient to think of these objects as ``$\Phi$-relative quasi-symplectic groupoids''.  Their corresponding infinitesimal data now take the form of Dirac structures in Courant algebroids of the type $\mathbb{T}_\eta M \times \mathfrak{d}$ (here $\bT_\eta M$ denotes the Courant algebroid $TM\oplus T^*M$ with bracket twisted by a closed 3-form $\eta$), and a refinement of Theorem \ref{thm:intiso} leads to Corollary \ref{cor:int2lag}, establishing a Lie-type correspondence
$$
\{ \text{Dirac structures in } \mathbb{T_\eta}M \times \mathfrak{d}\} \rightleftharpoons \{\text{Lagrangian morphisms } \Phi: \cH \to D \};
$$
this recovers the integration theory of Dirac structures in exact Courant algebroids from \cite{burint} when $D$ is a point. It also provides an integration procedure for a particular class of infinitesimal lagrangian structures considered in \cite{pym:symalg}.

We further study lagrangian fibred products for such 2-shifted lagrangian morphisms, showing how they give rise to quasi-symplectic groupoids (or, infinitesimally, to Dirac structures in exact Courant algebroids). As an application, we consider a natural lagrangian morphism associated with the groupoid of arrows of a 2-shifted symplectic group and use it to obtain an integration of the Gauss–Dirac structure \cite{purspi} on a complex reductive Lie group as a lagrangian fibred product (Example~\ref{ex:gauss-dirac}). Many other examples are presented in $\S$ \ref{sec:qpoisson}.

Section $\S$ \ref{sec:lag-to-poi} presents a construction linking  2-shifted lagrangian morphisms to quasi-Poisson groupoids \cite{ilx:univer}. Interpreting quasi-Poisson structures on groupoids as 1-shifted Poisson structures \cite{bon:shif,saf:poi}, this should be understood as a concrete instance of the general principle that $n$-shifted lagrangian structures induce $(n-1)$-shifted Poisson structures \cite{saf:der}. A key ingredient underlying the construction is an alternative characterization of 2-shifted lagrangian morphisms in terms of {\em morphisms of CA-groupoids} (Prop.~\ref{prop:coumor}). Using this perspective, we show that, for any 2-shifted lagrangian morphism $\Phi: \cH \to D$, each choice of lagrangian complement to the associated Dirac structure $\mathbf{L}\subset \bT_\eta N \times \fd$ induces a quasi-Poisson groupoid structure on $\cH$. It follows that $\Phi: \cH \to D$ encodes a quasi-Poisson structure well defined up to twists  (Prop. \ref{prop:qpManin} and Thm.~\ref{thm:intmp}). When $D$ is a point, this construction recovers the “inversion” of quasi-symplectic groupoids to quasi-Poisson groupoids in \cite[$\S$ 4]{buriglsev}; it also provides alternative viewpoints on the integration of Lie quasi-bialgebras to quasi-Poisson groups \cite{kos:quan} and on the construction of dynamical Poisson groupoids \cite{eti:rmat}.

Section $\S$ \ref{sec:qpoisson} concerns {\em quasi-Poisson manifolds} \cite{alekos,alemeikos}, i.e. manifolds carrying suitable actions of Lie quasi-bialgebras arising in Poisson Lie theory and in the study of Poisson structures on moduli spaces of flat bundles. Their connection with the present work is based on the fact, shown in \cite{quadir}, that quasi-Poisson manifolds admit a description as special Dirac structures in $\bT M\times \fd$  (see Def.~\ref{def:qpoiL}), where $\fd$ is the Drinfeld double of the Lie quasi-bialgebra, and hence can be viewed as infinitesimal counterparts of 2-shifted lagrangian morphisms. In Theorem~\ref{thm:intqpoi}, we characterize the 2-shifted lagrangian morphisms integrating quasi-Poisson manifolds, showing that they are certain group-valued moment maps. Following \cite{manpairev}, we term these objects {\em multiplicative $D$-valued moment maps}, since they are multiplicative analogues of the moment maps studied there. This yields a Lie-type correspondence
$$
\{ \text{Quasi-Poisson manifolds}\} \rightleftharpoons \{\text{Multiplicative $D$-valued moment maps}\}.
$$
When $D$ is trivial, one recovers the integration of Poisson manifolds to symplectic groupoids. We summarize these correspondences schematically in Fig.~\ref{fig:table of integration}.

The emergence of ``moment maps'' as integrations of quasi-Poisson manifolds broadly extends the fact that the cotangent lift of any action is hamiltonian  (or, more generally, that Poisson actions integrate to hamiltonian actions on symplectic groupoids \cite{mommap}).
As with ordinary moment maps, one can consider {\em reduction} of multiplicative $D$-valued moment maps. As in the classical setting, this proceeds in two steps: restriction, based on fibred products of 2-shifted lagrangians, followed by a quotient operation (Prop.~\ref{prop:integL} and Thm.~\ref{thm:reducSG}). This framework provides a systematic method for constructing quasi-symplectic groupoids. Examples include the explicit integrations of affine Dirac structures on Poisson-Lie groups and symplectic groupoids of Poisson homogeneous spaces found in \cite{intpoihom}, obtained by reducing a multiplicative $D$-valued moment map associated with the action of a Poisson Lie group on itself by multiplication. We also recover, from a new perspective, the integration of (complete) Poisson actions to multiplicative hamiltonian actions (building on \cite{lupoiact}), as well as the integrations of Poisson quotients studied in \cite{ferigl,mommap}. 
Integrating the quasi-Poisson manifolds studied in \cite{alemeikos} leads to an interpretation of the construction of symplectic groupoids via moduli spaces in \cite{quahammer} as 2-shifted lagrangian fibred products. 
We give examples of explicit integrations of such quasi-Poisson structures on Lie groups and partial flag varieties (Prop.~\ref{pro:integration of G,0} and Example~\ref{ex:homsp}), which serve as building blocks for integrating (quasi-)Poisson structures on decorated moduli spaces of flat connections \cite{quisur}, a direction developed further in \cite{part2}.

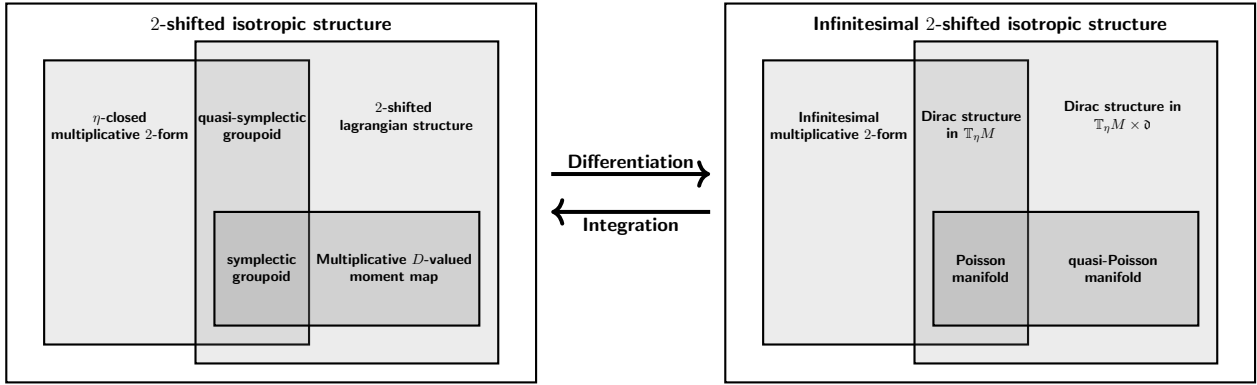
\begin{figure}
    \centering
    \begin{tikzpicture}[scale=0.5, transform shape, font=\sffamily]

      % ==========================================
      % GLOBAL COUNTERPARTS (LEFT SIDE)
      % ==========================================
      \begin{scope}
        % Outer Box: Universal Set 
        \draw[thick] (0,0) rectangle (14, 10);
        \node[anchor=north] at (7, 9.8) {\Large \textbf{$2$-shifted isotropic structure}};

        % Boxes
        % Multiplicative 2-form box
        \filldraw[thick, fill=gray, fill opacity=0.15] (1, 1) rectangle (8, 8.5);
        % 2-lagrangian box (made taller and deeper: y goes from 0.5 to 9.0)
        \filldraw[thick, fill=gray, fill opacity=0.15] (5, 0.5) rectangle (13, 9);
        % Intersection/Inner box
        \filldraw[thick, fill=gray, fill opacity=0.25] (5.5, 1.5) rectangle (12.5, 4.5);

        % Text Labels
        \node[align=center] at (3, 6.75) {\textbf{$\eta$-closed} \\ \textbf{multiplicative $2$-form}};
        \node[align=center] at (10.5, 7) {\textbf{$2$-shifted}\\ \textbf{ lagrangian structure}};
        \node[align=center] at (6.5, 6.75) {\textbf{quasi-symplectic} \\ \textbf{groupoid}};
        \node[align=center] at (10.25, 3) {\textbf{Multiplicative $D$-valued} \\ \textbf{moment map}};
        \node[align=center] at (6.75, 3) {\textbf{symplectic} \\ \textbf{groupoid}};
      \end{scope}

      % ==========================================
      % CONNECTING ARROWS
      % ==========================================
      % Shifted Lie functor arrow slightly up to y=5.5 for balance
      \draw[->, line width=1.5pt] (14.4, 5.5) -- (18.6, 5.5) node[midway, above] {\Large \textbf{Differentiation}};
      
      % Added Integration arrow beneath it at y=4.5
      \draw[<-, line width=1.5pt] (14.4, 4.5) -- (18.6, 4.5) node[midway, below] {\Large \textbf{Integration}};

      % ==========================================
      % INFINITESIMAL STRUCTURES (RIGHT SIDE)
      % ==========================================
      % Shifted right by 19 units to sit next to the Global box
      \begin{scope}[shift={(19,0)}]
        % Outer Box: Universal Set
        \draw[thick] (0,0) rectangle (14, 10);
        \node[anchor=north] at (7, 9.8) {\Large \textbf{Infinitesimal $2$-shifted isotropic structure}};

        % Boxes
        % Infinitesimal multiplicative 2-form box
        \filldraw[thick, fill=gray, fill opacity=0.15] (1, 1) rectangle (8, 8.5);
        % Dirac structure box (made taller and deeper: y goes from 0.5 to 9.0)
        \filldraw[thick, fill=gray, fill opacity=0.15] (5, 0.5) rectangle (13, 9);
        % Intersection/Inner box
        \filldraw[thick, fill=gray, fill opacity=0.25] (5.5, 1.5) rectangle (12.5, 4.5);

        % Text Labels
        \node[align=center] at (3, 6.75) {\textbf{Infinitesimal} \\ \textbf{multiplicative $2$-form}};
        \node[align=center] at (10.5, 7) {\textbf{Dirac structure in} \\ \textbf{$\bT_\eta M\times \fd$}};
        \node[align=center] at (6.5, 6.75) {\textbf{Dirac structure} \\ \textbf{in $\bT_\eta M$}};
        \node[align=center] at (10.25, 3) {\textbf{quasi-Poisson} \\ \textbf{manifold}};
        \node[align=center] at (6.75, 3) {\textbf{Poisson} \\ \textbf{manifold}};
      \end{scope}

\end{tikzpicture}
\caption{Lie theory for 2-shifted isotropic structures}\label{fig:table of integration}
\end{figure}

Several directions extend the scope of this paper. Our work naturally generalizes to lagrangian morphisms into more general 2-shifted symplectic {\em groupoids} (not just groups), as well as lagrangian structures on generalized (rather than strict) morphisms; infinitesimally, these relate to Dirac structures in heterotic Courant algebroids \cite{barhek}. Our results also admit extensions to Lie $n$-groupoids. 
Concerning multiplicative $D$-valued moment maps, their properties may be better understood through the codiagonal functor from bisimplicial manifolds to simplicial manifolds \cite{raj:from}. 
In another direction, following \cite{cueca}, 2-shifted symplectic groups admit a Morita-equivalent infinite-dimensional model (parallel to the Morita equivalence between quasi-hamiltonian spaces and hamiltonian loop-group spaces in \cite{alemeimal,moxu}); the infinite-dimensional counterparts of the lagrangian morphisms considered here deserve further study. We plan to pursue these directions in future work.

The paper is developed in the framework of real, $C^\infty$-manifolds.

\medskip

\noindent{\bf Notation.} 
For a (real) vector space $V$, we denote by $V_M=V\times M \to M$ the trivial vector bundle over a manifold $M$.  For a map of vector bundles $\varphi: A\to B$ covering $\phi$, we keep the same notation for the induced map $\varphi: A\to \phi^*B$. 

We use the notation $\theta^l$ and $\theta^r$ for left/right Maurer Cartan forms on Lie groups.

Throughout the paper, we occasionally abuse notation by denoting pullbacks using the defining expressions of maps, rather than the maps themselves. For example, for the projection $f: M\times N \to N$, $f(x,y)=y$, we may write $f^*\omega|_{(x,y)}$ as $y^*\omega$. 

\medskip

\noindent{\bf Acknowledgements.}
We thank D. Calaque, E. Getzler, M. Gualtieri, J.-H. Lu, E. Meinrenken, M. Mol, and C. Zhu for helpful discussions. D.A. and M.C. are grateful to IMPA for its hospitality at various stages of this project. H.B. acknowledges financial support from CNPq and Faperj. M.C. was partially supported by FWO grants 1249325N and G014726N.

\section{General framework: shifted symplectic groupoids and lagrangian morphisms}\label{sec:general}
In this section we provide the general framework extending symplectic manifolds and their lagrangian submanifolds to the realm of Lie groupoids, following \cite{Lesdiablerets,ptvv}.

Given a Lie groupoid $\gro{G}{M}$, we will denote its structure maps by $\gs,\gt: \cG \to M$ (source and target maps), $\gm: \cG_{(2)} \to \cG$ (multiplication map), $\gu: M \to \cG$ (unit map), and $\gi: \cG \to \cG$ (inversion map); we may label these maps by the groupoid if there is any risk of confusion. Here $\cG_{(2)}$ denotes the submanifold of pairs $(g,h)$ in $\cG\times \cG$ satisfying $\gs(g)=\gt(h)$. We will generally identify $M$ with its image in $\cG$ under the unit map.

The {\em Lie algebroid of $\cG$} is defined by $\al_\cG=\ker T\gs |_{M}$ with anchor map $\mathtt{a}=T\gt:\al_\cG\to TM$ and bracket induced by right-invariant vector fields. 
For a morphism $\Phi$ between Lie groupoids $\gro{G}{M}$ and $\gro{H}{N}$, we denote by $\phi=\Phi|_M:M\to N$ its restriction to units and by $\varphi:=\mathrm{Lie}(\Phi): A_\cG\to A_\cH$ the corresponding map of Lie algebroids.

To define shifted symplectic structures on a Lie groupoid, we start by recalling shifted differential forms.

\subsection{Differential forms on Lie groupoids}\label{subsec:diffforms}
Given a Lie groupoid $\gro{G}{M}$, we denote the manifold
of all of its $n$-tuples of composable elements by $\cG_{(n)}$  (with $\cG_{(0)}=M$). The projections, multiplication and unit maps define a simplicial manifold $(\cG_{(\bullet)},d_i,s_j) $, called the {\em nerve of $\gro{G}{M}$} (descriptions of the face and degeneracy maps can be found e.g. in \cite{behcohsta,cueca}). 
The spaces of differential forms $\Omega^b(\cG_{(a)}) $ form a double complex $(\Omega^\bullet(\cG_{(\bullet)}), \delta, d)$, known as the {\em Bott-Shulman-Stasheff complex \cite{botshu}}, 
where  
 \[ \delta=\sum_{i=0}^{a+1} (-1)^id_i^*:\Omega^b(\cG_{(a)})\to \Omega^b(\cG_{(a+1)}), \quad d:\Omega^b(\cG_{(a)}) \to \Omega^{b+1}(\cG_{(a)}), 
 \] 
are the simplicial  and de Rham differentials. The associated total complex $\text{Tot}_n(\Omega^\bullet(\cG_{(\bullet)}))=\oplus_{a+b=n}\Omega^b(\cG_{(a)})$ is equipped with a differential $\partial$ given by $\delta+(-1)^ad$ when restricted to each $\Omega^b(\cG_{(a)})$.
The forms that pullback to zero along all the degeneracy maps, 
$$
\widehat{\Omega}^\bullet(\cG_{(\bullet)})=\{\omega\in \Omega^\bullet(\cG_{(\bullet)})\ | \ s_i^*\omega=0 \ \forall i\},
$$ 
form a subcomplex of the double complex $(\Omega^\bullet(\cG_{(\bullet)}), \delta, d)$, referred to as the {\em normalized } subcomplex.

An {\em $m$-shifted $n$-form} on a Lie groupoid $\gro{G}{M}$ is a collection of forms 
\begin{equation*}
    \omega_\bullet=\sum_{i=0}^m \omega_i, \quad \text{with}\quad \omega_i\in \widehat{\Omega}^{m+n-i}(\cG_{(i)}).
\end{equation*}
Notice that $\partial(\omega_\bullet)$ is an $(m+1)$-shifted $n$-form. 
We say that $\omega_\bullet$ is {\em closed} if $\partial(\omega_\bullet)=0$.

\begin{exa}\label{ex:0-shif2-form}
A closed $0$-shifted $2$-form on a Lie groupoid $\gro{G}{M}$ is given by $\varpi\in\Omega^2(M)$ (note that forms on $M$ are automatically normalized) 
such that
\begin{equation}\label{eq:0closed}
d\varpi=0\quad \text{and}\quad \delta\varpi=\gs^*\varpi-\gt^*\varpi=0.\vspace{-6mm}
\end{equation}
\hfill $\diamond$
\end{exa}

\begin{exa}\label{ex:1-shif2-form}
A closed $1$-shifted $2$-form on a Lie groupoid $\gro{G}{M}$ consists of  
$$
\omega\in {\Omega}^2(\cG)\quad \text{and}\quad \eta\in\Omega^3(M)
$$ 
satisfying  
$$
\delta\omega=\pr_1^*\omega+\pr_2^*\omega-\gm^*\omega=0, \quad d\omega=\delta \eta=\gs^*\eta-\gt^*\eta\quad \text{and}\quad d\eta=0.
$$
The condition $\delta \omega=0$ says that $\omega$ is a {\em multiplicative} $2$-form on $\cG$, and it ensures that  $\omega$ is automatically normalized. 
\hfill $\diamond$
\end{exa}

\begin{exa}\label{ex:2-shif2-form}
A closed $2$-shifted $2$-form on a Lie groupoid $\gro{G}{M}$ consists of 
$$
\omega_2\in\widehat{\Omega}^2(\cG_{(2)}), \quad \omega_1\in\widehat{\Omega}^3(\cG)\quad \text{and}\quad \omega_0\in\Omega^4(M)
$$ 
satisfying the following conditions: 
$$
\delta\omega_2=0, \quad \delta \omega_1=\pr_1^*\omega_1+\pr_2^*\omega_1-\gm^*\omega_1=-d\omega_2,\quad d\omega_1=\delta \omega_0=\gs^*\omega_0-\gt^*\omega_0\quad\text{and}\quad d\omega_0=0.\vspace{-6mm}
$$
\hfill $\diamond$
\end{exa}

At the infinitesimal level, the normalized subcomplex $\widehat{\Omega}^\bullet(\cG_{(\bullet)})$ is encoded in the  {\em Weil algebra} of $A_\cG$ by means of the van Est map   \cite{weivanest, raj:sup}, which is a  map of double complexes
\begin{equation}\label{eq:VE}
\mathrm{VE}:(\widehat{\Omega}^\bullet(\cG_{(\bullet)}), \delta, d)\to (W^{\bullet,\bullet}(A_\cG), d^h, d^v),
\end{equation}
as recalled in Appendix \ref{app:wei}.

\subsection{Shifted symplectic forms}\label{subsec:shiftedsymp}
A symplectic structure on a manifold $M$ is a closed $2$-form $\omega\in\Omega^2(M)$ that is nondegenerate, in the sense that the contraction map 
\begin{equation}\label{eq:omegaflat}
\omega^\flat:TM\to T^*M, \quad X\mapsto \omega(X, \cdot),
\end{equation}
is an isomorphism. Following \cite{Lesdiablerets}, to adapt this definition to the realm of Lie groupoids $\gro{G}{M}$,  we replace $\omega$ by a closed shifted $2$-form, the tangent and cotangent bundles by the tangent and cotangent complexes of $\cG$,
and the map $\omega^\flat$ by an analogous map between these complexes, required to be a quasi-isomorphism.

For a Lie groupoid $\gro{G}{M}$, its \emph{tangent complex} is the 2-term complex of vector bundles
$$
\cT_\cG:= (\al_\cG\xrightarrow{\mathtt{a}}TM),
$$ 
with $\al_\cG$ in degree $1$ and $TM$ in degree $0$ (and zero elsewhere), and its   \emph{cotangent complex} is
$$
\cT^*_\cG[-1]:=(T^*M\xrightarrow{\mathtt{a}^*}\al_\cG^*),
$$ 
with  $T^*M$ in degree $1$ and  $\al_\cG^*$ in degree $0$ (and zero elsewhere).

We note that passing from a Lie groupoid to its tangent complex is functorial; for a groupoid morphism $\Phi:\cG\to \cH$, we denote by $\cT\Phi:\cT_\cG\to \cT_\cH$ the induced chain map. We will occasionally regard $\cT\Phi$ as a chain map from $\cT_\cG$ to $\phi^*\cT_\cH$, with no change in notation.

\begin{rema}\label{rem:tancot} The tangent and cotangent bundles of a Lie groupoid $\gro{G}{M}$ are VB-groupoids (see e.g. \cite[\S 11.2]{macgen}), known as the \emph{tangent}
and \emph{cotangent prolongations} of $\cG$, 
$$
T\gro{G}{TM},  \qquad\;\; T^*\gro{G}{A^*_{\cG}}.
$$ 
The tangent and cotangent complexes arise from these VB-groupoids through a groupoid version of the Dold-Kan correspondence. These complexes coincide with the so-called {\em core complexes} of $T\cG$ and $T^*\cG$, respectively, and, as such, 
they carry the additional structure of representations up to homotopy \cite{ariascraruth}; see \cite[$\S$ 3.3]{raj:vb-group} and \cite[$\S$ 2.1]{mat:morvb}.
\hfill $\diamond$

\end{rema}

Regarding the analogue of \eqref{eq:omegaflat}, for a Lie groupoid $\gro{G}{M}$, any $\omega_m\in \widehat{\Omega}^2(\cG_{(m)})$ gives rise to a degree-preserving map 
$$
\Lambda_{\omega_m}:\cT_\cG\to \cT^*_\cG[-m],
$$ 
where here we view $\cT_\cG$ and $\cT^*_\cG$ just as graded vector bundles, as follows: 
for $l\in\{0,1\}$, $x\in M$, and $v\in (\cT_\cG)_{l} |_x$, $w\in (\cT_\cG)_{m-l}|_x$, we define\footnote{In the formula we are considering the natural inclusions $(\cT_\cG)_{k}|_x\subseteq T_x \cG_{(k)}$; since in our case $\cT_\cG$ is concentrated in degrees $0$ and $1$, the nontrivial inclusions are just $TM\subseteq T\cG_{(0)}=TM$ and $A=\ker T\gs|_M \subseteq T\cG|_M$.}

\begin{equation}\label{nondeg-pairing}
    \langle\Lambda_{\omega_m}( v), w \rangle := (-1)^{\frac{m(m-1)}{2}} \sum_{\varsigma \in Sh(l,m-l)} \mathrm{sign}(\varsigma)\ \omega_m\big(T(s_{\varsigma(m-1)}\dots s_{\varsigma(l)} )v, T(s_{\varsigma(l-1)} \dots s_{\varsigma(0)})w\big)
\end{equation}
where $Sh(l, m-l)$ is the set of $(l, m-l)$-shuffles. (See Remark \ref{rem:m2sign} for a justification of the sign factor in the previous formula.)

When $\delta \omega_m=0$, one may verify  (see \cite[Lem.~ E.1]{cueca}) that  $\Lambda_{\omega_m}$ is compatible with differentials, i.e. it is a chain map.

The next remark explains how $\Lambda_{\omega_m}$ relates to $\omega_m^\flat: T \cG_{(m)} \to T^*\cG_{(m)}$.

\begin{rema} 
As a morphism of graded vector bundles, we may view $\Lambda_{\omega_m}$ as 
a map $(\cT_\cG \otimes \cT_\cG)_m \rightarrow \mathbb{R}_M $. For a simplicial manifold $\cG=(\cG_{(\bullet)},d_i,s_j)$ over $M$,  consider the simplicial vector bundle $T_M\cG := (T \cG_{(n)}|_M,Td_i|_M,Ts_j|_M)$. When $\cG$ is a (higher) Lie groupoid, its tangent complex is obtained from $T_M\cG$ through the so-called ``normalized complex functor''  $N(\cdot)$ (see e.g. \cite[\S 2.1]{cueca} for its definition and further references),
$$
T_M\cG \mapsto N(T_M\cG) = \cT_\cG,
$$
which gives one direction of the Dold-Kan correspondence. 
The Eilenberg-Zilber Theorem provides a canonical quasi-isomorphism
    \[ 
    \text{EZ}: N(T_M\cG)\otimes N(T_M\cG)  \rightarrow N(T_M\cG \otimes T_M\cG), 
    \]
    explicitly described e.g. in \cite[Def.~29.7 (ii)]{simobj}. Noticing that $N(T_M\cG \otimes T_M\cG)_n \subset T \cG_{(n)}|_M \otimes T\cG_{(n)}|_M$ for all $n$, one shows that $\Lambda_{\omega_m}$ is given (modulo a sign convention) by the following composition:
    \[ \begin{tikzcd}
(\mathcal{T}_\mathcal{G}\otimes \mathcal{T}_\mathcal{G} )_m\arrow[d, hook] \arrow[rrrr, "\Lambda_{\omega_m}"] &  &                                                                   &  & \mathbb{R}_M                                                    \\
\mathcal{T}_\mathcal{G}\otimes \mathcal{T}_\mathcal{G}=N(T_M\cG)\otimes N(T_M\cG) \arrow[rr, "\text{EZ}"']                                                            &  & N(T_M\cG\otimes T_M\cG) \arrow[rr, "\text{projection}"'] &  & {T}\mathcal{G}_{(m)}|_M\otimes {T}\mathcal{G}_{(m)}|_M, \arrow[u, "\omega_m"']
\end{tikzcd}\]
see also \cite[\S 3]{ronchizhu}. \hfill $\diamond$
\end{rema}

\begin{defi}\label{def:shiftedsymp}
A closed $m$-shifted $2$-form $\omega_\bullet$ on the Lie groupoid  $\gro{G}{M}$ is called \emph{nondegenerate} if the chain map $\Lambda_{\omega_m}:\cT_\cG\to \cT^*_\cG[-m]$ is a quasi-isomorphism. An \emph{$m$-shifted symplectic structure} on the Lie groupoid $\gro{G}{M}$ is a closed and non-degenerate $m$-shifted $2$-form. 
\end{defi}

Since $\cT_\cG$ is concentrated in degrees $0$ and $1$, it has only two homology groups. For both to vanish, the Lie algebroid $A_\cG$ must be isomorphic to $TM$, and in this case any closed $m$-shifted $2$-form on $\gro{G}{M}$ is automatically non-degenerate. On the other hand, the non-triviality of the homology groups forces the $m$-shifted symplectic structures to have $m\in\{0,1,2\}$: 
\begin{itemize}
    \item If only the zero homology group is nontrivial, then the symplectic form must be 0-shifted.
    \item If only the first homology group is nontrivial, then the symplectic form must be 2-shifted.
    \item If both homology groups are nontrival, then the symplectic form must be 1-shifted.
\end{itemize}

We will briefly discuss the cases $m=0, 1$ in $\S$ \ref{subsec:0shift} and $\S$ \ref{subsec:1shift}. The case $m=2$ is the focus of $\S$ \ref{sec:2shift}.

\begin{rema}\label{rem:m2sign}
The map $\Lambda_{\omega_m}$ in \eqref{nondeg-pairing} is related to the van Est map \eqref{eq:VE}. Following the notation in Appendix \ref{app:wei}, for an
 $m$-shifted $2$-form $\omega_\bullet$ on $\gro{G}{M}$, consider $\mathrm{VE}(\omega_m)=(c_0,\cdots, c_m)\in W^{m,2}(A_\cG)$. Using the formula for the van Est map, we see that $\Lambda_{\omega_m}$  can be identified with the highest degree component $c_m$; more specifically,
\begin{equation*}
    \left\{
    \begin{array}{lll}
        \text{For } m=0,& \langle\Lambda_{\omega_0}( X), Y \rangle=c_0(|)(X,Y)& \forall \, X, Y\in TM.   \\
         \text{For } m=1,& \langle\Lambda_{\omega_1}( X), u \rangle =c_1(| u)(X)
         & \forall\, X\in TM, \ u\in A.   \\
         \text{For } m=2,& \langle\Lambda_{\omega_2}( u), v \rangle=c_2(|u,v)& \forall\, u,v\in A.   
    \end{array}\right.%\vspace{-6mm}
\end{equation*} 
The sign factor in \eqref{nondeg-pairing} was introduced to ensure the identity for $m=2$ (note that it has no effect for $m=0$ or $m=1$).
\hfill $\diamond$ 
\end{rema}

\subsection{Shifted lagrangian structures on morphisms} \label{subsec:shiftedlag}
We now recall the concept of shifted lagrangian structures, mainly following  \cite{cal:lag, saf:qua}.

Given a symplectic manifold $(M,\omega)$, a submanifold $i:N\to M$ is lagrangian if
\begin{itemize}
    \item $N$ is isotropic, i.e., $i^*\omega=0$; equivalently,
    \begin{equation}\label{eq:isotr}
(Ti)^*\circ \omega^\flat \circ Ti=0;
    \end{equation} 
    \item The composite 
    \begin{equation}\label{eq:normal}
    \nu_N\xrightarrow{T^*i\ \circ\ \omega^\flat} T^*N
    \end{equation}
    is an isomorphism, where $\nu_N=TM|_N/TN$ is the normal bundle. 
\end{itemize}

In the more general context of Lie groupoids, the role of lagrangian submanifolds is played by morphisms into shifted symplectic groupoids with the first condition above  replaced by the vanishing of the pullback of the shifted symplectic form in cohomology together with a {\em choice} of primitive; given this choice, the isomorphism \eqref{eq:normal} is replaced by a suitable quasi-isomorphism of complexes. So, rather than being an intrinsic property of a morphism, the lagrangian condition requires additional data attached to the morphism, referred to as a ``lagrangian structure''.

Let $(\gro{G}{M}, \omega_\bullet)$ be an $m$-shifted symplectic groupoid, and let $\gro{H}{N}$ be a Lie groupoid.

\begin{defi}\label{def:iso}
An \emph{isotropic structure} on a morphism $\Phi:\cH \to \cG$ is an $(m-1)$-shifted $2$-form $\sigma_\bullet$ such that $\Phi^*\omega_\bullet=\partial\sigma_\bullet$. 
\end{defi}

We denote a groupoid morphism $\Phi:\cH \rightarrow \cG$ equipped with an isotropic structure $\sigma_\bullet$ by $(\cH, \Phi, \sigma_\bullet)$, or simply by $(\Phi, \sigma_\bullet)$ if there is no risk of confusion, and refer to it as an {\em isotropic morphism}.

\begin{rema} Isotropic structures admit an interpretation analogous to \eqref{eq:isotr}.
Any morphism $\Phi:\cH \rightarrow \cG$ gives rise to a chain map $\cT_\cH[1]\to \cT^*_\cH[1-m]$ 
given by the composition
$$
\cT_\cH[1]\xrightarrow{\cT\Phi} \phi^*\cT_\cG[1]\xrightarrow{\Lambda_{\omega_m}} \phi^*\cT^*_\cG[1-m]\xrightarrow{(\cT\Phi)^*}\cT^*_\cH[1-m].
$$ 
Then an isotropic structure $\sigma_\bullet$ on $\Phi:\cH \rightarrow\cG$ defines  a homotopy from  the previous map $\cT_\cH[1]\to \cT^*_\cH[1-m]$ to the zero chain map via the maps $(\Lambda_{\sigma_{m-1}})_l: (\cT_\cH)_l\to (\cT^*_\cH)_{l+ 1-m}$, for $l=0,1$.  
(This is a consequence of the condition $\delta\sigma_{m-1}=\Phi^*\omega_m$ and can be directly verified for $m=0,1,2$; the general fact can be proven along the lines of \cite[Lemma E.1]{cueca}). 
\hfill $\diamond$
\end{rema}

In order to adapt the lagrangian condition \eqref{eq:normal} to the realm of Lie groupoids, we need a replacement for the normal bundle of a submanifold.
As a motivation, recall that a chain map is a quasi-isomorphism if and only if all the homology groups of its mapping cone are zero, i.e. if and only if its mapping cone is exact. 
We hence consider the \emph{normal complex} $\cV_\Phi$ of a groupoid morphism $\Phi:\cH\to \cG$ as the mapping cone of the induced map $\cT \Phi$ of tangent complexes: 
\begin{equation}\label{eq:normalcx}
\cV_\Phi:=\left(\cT_\cH\oplus\phi^*\cT_\cG[1], \begin{pmatrix}\mathtt{a}_\cH& 0\\ \cT\Phi&-\mathtt{a}_\cG\end{pmatrix}\right)=\left(\al_\cH\xrightarrow{(\mathtt{a}_\cH, \Lie(\Phi))}TN\oplus \phi^*\al_\cG\xrightarrow{\quad T\phi-\mathtt{a}_\cG}\phi^*TM\right),
\end{equation}
where we view $\mathtt{a}_\cH$ and $\mathtt{a}_\cG$ as the differentials in the corresponding tangent complexes. This definition recovers the classical concept of the normal bundle of a submanifold $i:N\to M$ by regarding $N$ and $M$ as unit groupoids and passing to the homology of the normal complex determined by $i$.

The following result, proven by direct inspection, gives the analogue of the map in \eqref{eq:normal}.

\begin{prop}\label{prop:iso} Let $\omega_\bullet$ be an $m$-shifted symplectic form on $\gro{G}{M}$, and  let $(\cH, \Phi, \sigma_\bullet)$ be an isotropic morphism into $(\cG,\omega_\bullet)$.
Then the following is a chain map 
\begin{equation*}
    \begin{array}{rccl}
        \Lambda_{\Phi,\sigma}:&\cV_\Phi&\to& \cT^*_\cH[1-m]  \\
         &v\oplus w&\mapsto& \Lambda_{\sigma_{m-1}}(v)+ (-1)^{\frac{m(m-1)}{2}}(\cT\Phi)^*\circ \Lambda_{\omega_m} (w).
    \end{array}
\end{equation*}
\end{prop}

\begin{rema}  To make contact with the description in \cite[$\S$ 1.2]{saf:qua} of isotropic structures as commutativity data, note that the outer square in the diagram below commutes up to the homotopy $\Lambda_{\sigma_{m-1}}$, 
\scriptsize
\[\begin{tikzcd}
 &  & \mathcal{T}_\mathcal{H}[1] \arrow[lld, "\mathcal{T}\Phi"'] \arrow[rrd]          &  &                                        \\
\phi^*\mathcal{T}_\mathcal{G}[1] \arrow[rrd,"U \mapsto 0\oplus U"'] \arrow[rrddd, "
"', bend right] &  &        &  & 0 \arrow[lld] \arrow[llddd, bend left] \\
   &  & \mathcal{V}_\Phi \arrow[dd, "\Lambda_{\Phi,\sigma}" description, dashed] &  &           \\
   &  &                                                              &  &                                        \\                    &  & {\mathcal{T}^*_\mathcal{H}[1-m]}                             &  &                                       
\end{tikzcd}\] \normalsize
while the chain map described in the previous proposition makes commutative the triangles in the diagram above.  
\hfill $\diamond$

\end{rema}

We are now ready to introduce the main object of study in this work.

\begin{defi}[\cite{cal:lag, ptvv, saf:qua}]\label{def:gen-lag}
An isotropic morphism $(\cH, \Phi, \sigma_\bullet)$ into $(\gro{G}{M}, \omega_\bullet)$ is called \emph{lagrangian} if the chain map 
$\Lambda_{\Phi,\sigma}:\cV_\Phi\to \cT^*_\cH[1-m]$ is a quasi-isomorphism. In this case, we say that $\sigma_\bullet$ is a {\em lagrangian structure} on the morphism $\Phi$, and that $(\cH, \Phi, \sigma_\bullet)$ is a {\em lagrangian morphism}.  \hfill $\diamond$ 
\end{defi}

A useful criterion to verify that $\Lambda_{\Phi,\sigma}$ is a quasi-isomorphism 
is checking the exactness of its mapping cone, given explicitly by the following complex: 
    \begin{equation}\label{eq:exseq-lag}
        \left(\cT_\cH\oplus\phi^*\cT_\cG[1]\oplus\cT^*_\cH[2-m], \begin{pmatrix}\mathtt{a}_\cH&0&0\\ \cT\Phi &-\mathtt{a}_\cG &0\\ \Lambda_{\sigma_{m-1}}&(-1)^{\frac{m(m-1)}{2}}(\cT\Phi)^*\circ \Lambda_{\omega_m}&(-1)^{m+1}\mathtt{a}^*_\cH\end{pmatrix}\right).
    \end{equation}

We now list some general examples of lagrangian structures.
More concrete characterizations and examples for 0, 1 and 2 shifts will be discussed in $\S$ \ref{subsec:0shift}, $\S$ \ref{subsec:1shift} and $\S$ \ref{subsec:2lag} below. 

The first general observation is that lagrangian morphisms into $m$-shifted symplectic groupoids are generalizations of $(m-1)$-shifted symplectic groupoids,  see \cite[Ex. 2.3]{cal:lag}. 

\begin{exa}[Lagrangian morphisms into a point]\label{ex:2lagpoint}
Recall that a point regarded as the trivial groupoid with the zero 2-form is shifted symplectic for any shift. Then a lagrangian morphism into a point, viewed as $m$-shifted symplectic, is the same as a Lie groupoid $\cH$ equipped with an $(m-1)$-shifted 2-form $\sigma_\bullet$ such that $\partial\sigma_\bullet=0 $ and $\Lambda_{\Phi,\sigma}=\Lambda_{\sigma_{m-1}}:\cT_\cH\to \cT^*_\cH[1-m]$ is a quasi-isomorphism (note that in this case $ \cV_\Phi=\cT_\cH$), i.e., $(\cH,\sigma_\bullet)$ is an $(m-1)$-shifted symplectic groupoid. 
\hfill $\diamond$
\end{exa}

\begin{exa}[Diagonal]\label{ex:diag}
Let $(\gro{G}{M}, \omega_\bullet)$ be an $m$-shifted symplectic groupoid and denote by $\overline{\cG}$ the $m$-shifted symplectic groupoid defined by $-\omega_\bullet$. Then the trivial form $\sigma_\bullet=0$ is a lagrangian structure on the diagonal morphism $\Delta:\cG\to \overline{\cG}\times \cG$. Indeed the isotropic condition follows from $\Delta^*(\text{pr}_1^*\omega_\bullet-\text{pr}_2^*\omega_\bullet)=0$. To check that $(\Delta,0)$ is lagrangian, consider the normal complex
$\cV_\Delta=\cT_\cG\oplus \cT_\cG[1]\oplus \cT_\cG[1]$ as in \eqref{eq:normalcx}; in this case the map
$\cV_\Delta\to \cT_\cG[1]$, $(u,v,w)\mapsto v-w$ is a quasi-isomorphism. Then the map $\Lambda_{\Delta,0}: \cV_\Delta \to \cT^*_\cG[1-m]$, which is given by
$(u,v,w) \mapsto \Lambda_{\omega_m}(v)-\Lambda_{\omega_m}(w)$, is a quasi-isomorphism since it is the composition of two quasi-isomorphisms:
\begin{equation*}
    \xymatrix{\cV_\Delta \ar[d]_{}\ar[rr]^{\Lambda_{\Delta,0}}&& \cT^*_\cG[1-m]\\ \cT_\cG[1]\ar[urr]_{\Lambda_{\omega_m}}&}
\end{equation*}

\hfill $\diamond$\end{exa}

\begin{exa}[Groupoid of arrows]\label{ex:arro-grou}
The {\em groupoid of arrows} of a Lie groupoid $\cG \rightrightarrows M$, denoted by $\cG^I\rightrightarrows \cG$, is the action groupoid for the action of $\cG\times \cG$ on $\cG$ with respect to left and right multiplications.
Explicitly,  $\cG^I=\{(g,\gamma, h)\in\cG\times\cG\times\cG \ | \ \gs(\gamma)=\gs(h) \text{ and } \gt(\gamma)=\gs(g) \},$ $$
\gs^I(g,\gamma,h)=\gamma,\quad \gt^I(g,\gamma,h)=g\gamma h^{-1} \text{ and } (g_2,g_1\gamma h_1^{-1}, h_2)\cdot(g_1,\gamma ,h_1)=(g_2g_1, \gamma, h_2h_1).
$$ 
The map $\widetilde{\Delta}: \cG \to \cG^I$, $g\mapsto (g,\gu({\gs(g)}),g)$, is a weak equivalence (in the sense of \cite[$\S$ 5.4]{moeint}, alternatively called a Morita map \cite{mat:morvb}), see \cite[\S I.9]{simp:goer}, that fits into the following diagram:
$$
\xymatrix{ \cG\ar^{\Delta}[dr]\ar_{{\widetilde{\Delta}}}[d]\\ \cG^I\ar[r]_{ev}&\cG\times \cG \ ,} 
$$
where $ev(g,\gamma,h)=(g,h)$.

When $\cG$ is equipped with an $m$-shifted symplectic form $\omega_\bullet$, it is expected from Morita invariance that the trivial lagrangian structure on $\Delta$ (Example~\ref{ex:diag}) transfers to a lagrangian structure on $ev:\cG^I\to \overline{\cG}\times \cG$. Indeed, there is an explicit choice of lagrangian structure $\sigma_\bullet$ on $ev$:
For $m=1$, $\sigma_0=\omega_1$, and 
for $m=2$, 
$$
\sigma_0=\omega_1\quad \text{and}\quad {\sigma_1}_{(g,\gamma,h)}= (g\gamma h^{-1},h)^*{\omega_2}-(g,\gamma)^*{\omega_2}.
$$ 
The origin of this formula lies in the identification $\cG^I=Maps(\Delta^1, \cG)$, where $Maps$ denotes the internal hom in the category of simplicial sets; $\sigma_\bullet$ is the transgression of $\omega_\bullet$ using the Eilenberg-Zilber map, see e.g. \cite[Prop 2.10]{alg:hat} and \cite[\S I.9]{simp:goer}. \hfill$\diamond$
\end{exa}

\subsection{Fibred products of lagrangian morphisms} \label{subsec:lagprod}
Lagrangian morphisms can be used to relate shifted symplectic groupoids by ``lagrangian correspondences'',  leading to a vast generalization of  Weinstein's ``symplectic category'' \cite{cal:lag}. 

Recall that, given Lie groupoid morphisms $\Phi_i: \cG_i\to \cH$, $i=1,2$, their {\em fibred product} $\cG_1\times_\cH \cG_2$ has a natural groupoid structure over $M_1\times_N M_2$ with respect to componentwise multiplication. This is a Lie groupoid e.g. when $\Phi_1$ and $\Phi_2$ are transverse maps (see \cite[App.~A]{burcabhoy} for a more general discussion).

The following result is a differential geometric version of \cite[Thm. 4.4]{cal:lag}. 

\begin{thm}\label{thm:comp-lag} 
Let $(\cG_i,\omega_\bullet^i)$ be $m$-shifted symplectic groupoids, $i=1,2,3$, and consider lagrangian morphisms $(\cH_1,\Phi,\sigma^1_\bullet)$ into $\overline{\cG}_1\times \cG_2$ and  $(\cH_2,\Psi,\sigma^2_\bullet)$ into $\overline{\cG}_2\times \cG_3$. 
 
If $\Phi_2: \cH_1\to \cG_2$ and $\Psi_2: \cH_2\to \cG_2$ are transverse maps, then
$$
(\mathcal{H}_1 \times_{\cG_2} \mathcal{H}_2, (\Phi_1,\Psi_3), \pr_1^*\sigma^1_\bullet+\pr_2^*\sigma^2_\bullet)
$$
is a lagrangian morphism into $\overline{\cG}_1\times \cG_3$. 
\end{thm}

For $m=1$, this results can be found  in \cite[Thm 4.1]{max:coi}; for $m=2$, see Thm. \ref{thm:2lagprod} below.  The construction in the previous theorem is depicted in the following diagram:
\scriptsize\[ \begin{tikzcd}
              &                                     & \mathcal{H}_1\times_{\mathcal{G}_2}\mathcal{H}_2 \arrow[ld] \arrow[rd] &                                     &               \\
              & \mathcal{H}_1 \arrow[ld] \arrow[rd] &                                                                        & \mathcal{H}_2 \arrow[ld] \arrow[rd] &               \\
\mathcal{G}_1 &                                     & \mathcal{G}_2                                                          &                                     & \mathcal{G}_3
.\end{tikzcd}\]\normalsize

Setting  $\cG_1$ and $\cG_3$ as points in the previous theorem, we conclude from Example~\ref{ex:2lagpoint} that the (transverse) fibred product of lagrangian morphisms into a given $m$-shifted symplectic groupoid is an $(m-1)$-shifted symplectic groupoid. As before, this is a differential geometric adaptation of a general result for derived stacks \cite[Thm~2.9]{ptvv}.

\begin{coro}\label{cor:lagprod} 
Let $( \cH_i, \Phi_i, \sigma^i_\bullet)$, for $i=1,2$, be lagrangian morphisms into the $m$-shifted symplectic groupoid $(\cG, \omega_\bullet)$ such that $\Phi_1$ and $\Phi_2$ are transverse maps. Then $(\cH_1 \times_\mathcal{G} \cH_2,\sigma_\bullet=\pr_1^*\sigma^1_\bullet-\pr_2^*\sigma^2_\bullet)$ is an $(m-1)$-shifted symplectic groupoid. 
\end{coro}

\begin{rema}[On transversality]\label{rem:fibpro-trans} 
Fibered products of Lie groupoids may exist (as Lie groupoids) without the transversality assumption (see e.g. \cite[App. A]{burcabhoy}); in Theorem~\ref{thm:comp-lag}, however, transversality is essential to ensure the lagrangian property of the fibered product (or, in the case of Cor.~\ref{cor:lagprod}, that 
$\sigma_\bullet$ is nondegenerate). As an illustration,
let $\cG$ be a shifted symplectic groupoid and consider the lagrangian morphisms in $\cG\times\overline{\cG}$  given by the diagonal and the groupoid of arrows (see Examples~\ref{ex:diag} and \ref{ex:arro-grou}):
    $$
    \cG\xrightarrow{\Delta}\cG\times\overline{\cG}\xleftarrow{ev} \cG^I.
    $$
    Their fibred product is  the action groupoid  of $\cG$ acting on the isotropies $\mathrm{Iso}(\cG)=\cup_{x\in M}\cG_x$ by conjugation, known as the  ``inertia groupoid''. This is not a Lie groupoid in general, but it may be even when the morphisms are not transverse.      In such a case,   the inertia groupoid may not carry the expected shifted symplectic structure. For example, any cotangent bundle is a 1-shifted symplectic groupoid $T^*M\rightrightarrows M$ (see Example~\ref{ex:sympg}); in this case the action of $T^*M$ on $\mathrm{Iso}(T^*M) = T^*M$ is trivial, so the
    inertia groupoid is a Lie groupoid, but the fact that its anchor is not injective prevents it from being $0$-shifted symplectic (as explained in $\S$ \ref{subsec:0shift}).
    \hfill $\diamond$
\end{rema}

As we will recall in $\S$ \ref{subsec:1shift}, Corollary~\ref{cor:lagprod}  is a far-reaching generalization of symplectic reduction, see e.g. \cite[\S 2.2]{cal:lag}.
For the sake of completeness we will give its proof in \S \ref{sec:2fibpro} for $m=2$, which is the case of interest in this paper.

Given Lie groupoid morphisms $\Phi_i: \cH_i\to \cG$, $i=1,2$, recall that their {\em homotopy (or weak) fibred product} (see e.g.  \cite[$\S$ 5.3]{moeint}) agrees with the  fibred product of $\Phi_1\times \Phi_2: \cH_1\times \cH_2\to  \cG\times \cG$ and the projection $ev: \cG^I\to \cG\times \cG$, $ev(g,\gamma,h)=(g,h)$, where $\cG^I$ is the groupoid of arrows (Example~\ref{ex:arro-grou}).

\begin{coro} Let $( \cH_i, \Phi_i, \sigma^i_\bullet)$, for $i=1,2$, be lagrangian morphisms into the $m$-shifted symplectic groupoid $(\cG, \omega_\bullet)$. If the maps $\Phi_1\times \Phi_2$ and $ev$ are transverse,  then the homotopy fibred product of $\Phi_1$ and $\Phi_2$ is an $(m-1)$-shifted symplectic groupoid.
\end{coro}

This result is a direct consequence of Cor.~\ref{cor:lagprod}
since $\Phi_1\times\Phi_2$ and $ev$ are lagrangian morphisms, see Example~\ref{ex:arro-grou}.

\subsection{The $0$-shifted case: transversally symplectic foliations} \label{subsec:0shift}

 Let $\gro{G}{M}$ be a Lie groupoid equipped with a closed $0$-shifted 2-form $\varpi$, i.e., 
 $\varpi\in \Omega^2(M)$ such that 
 $$ 
 d\varpi=0, \qquad \gs^*\varpi=\gt^*\varpi,
 $$ 
 see Example \ref{ex:0-shif2-form}. 
 The second condition implies that $\mathtt{a}(A_\cG)\subseteq \ker(\varpi)$, ensuring that the map $\Lambda_\varpi:\cT_\cG\to \cT^*_\cG$ in \eqref{nondeg-pairing}, which in this case is given by 
\begin{equation}\label{eq:0-non-deg}
     \begin{array}{c}
          \xymatrix{  \al_\cG \ar[r]^-{\mathtt{a} } \ar[d]^-{0}& TM \ar[d]^-{\varpi^\flat}\ar[r]&0 \ar[d]^0\\ 
     0\ar[r]&T^*M \ar[r]_-{-\mathtt{a}^* } & \al_\cG^*,} 
     \end{array}
\end{equation}
 is a chain map.  Following Def.~\ref{def:shiftedsymp}, $\varpi$ is nondegenerate (i.e., symplectic) if and only if $\Lambda_\varpi$ is a quasi-isomorphism, which is equivalent to requiring that
the anchor map $\mathtt{a}$ is injective and that $\varpi^\flat$ induces an isomorphism
$$
(TM/\mathtt{a}(A_\cG)) \stackrel{\sim}{\longrightarrow} \ker(\mathtt{a}^*)=(TM/\mathtt{a}(A_\cG))^*.
$$

A Lie groupoid $\cG\rightrightarrows M$ with injective anchor map is called a {\em foliation groupoid}. In this case $\cG$ carries  a natural representation on the vector bundle $TM/\mathtt{a}(A_\cG)$, called the {\em normal representation}, and a 2-form $\varpi\in \Omega^2(M)$ satisfies $\gs^*\varpi=\gt^*\varpi$ if and only if $\mathtt{a}(A_\cG)\subseteq \ker(\varpi)$ and the induced element in $\Gamma(\wedge^2((TM/\mathtt{a}(A_\cG))^*))$ is $\cG$-invariant. Such a 2-form $\varpi$ is referred to as {\em basic}, and we say that it is {\em transversally nondegenerate} if $\mathtt{a}(A_\cG) =  \ker(\varpi)$, or equivalently if the induced map $(TM/\mathtt{a}(A_\cG)) \to (TM/\mathtt{a}(A_\cG))^*$ is an isomorphism.

Therefore a 0-shifted symplectic groupoid consists in a foliation groupoid $\gro{G}{M}$ equipped with a closed 2-form $\varpi$ on $M$ that is basic and transversally nondegenerate. In particular, if the orbit space of $\cG$ is smooth, it inherits a usual symplectic form. 
These objects are studied in \cite{stasym, Ler:ham} and model $0$-shifted symplectic stacks.

Let $(\gro{G}{M}, \varpi)$ be a $0$-shifted symplectic groupoid. Since Lie groupoids do not have $(-1)$-shifted forms, an isotropic structure on $\Phi:\cH\to \cG$ has to be zero, so the isotropic condition becomes $\phi^*\varpi=0$. According to \eqref{eq:exseq-lag}, the lagrangian condition is equivalent to  
\begin{equation}\label{eq:0-lag}
    0\to\al_\cH\xrightarrow{(\mathtt{a}_\cH,\Lie(\Phi))}TN\oplus \phi^*\al_\cG\xrightarrow{T\phi-\mathtt{a}_\cG}\phi^*TM\xrightarrow{T^*\phi\circ \varpi^\flat}T^*N\xrightarrow{-\mathtt{a}^*_\cH}\al^*_\cH\to0
\end{equation}
being an exact sequence. It follows that the anchor $\mathtt{a}_\cH: A_\cH\to TH$ must be injective.  On the other hand, if $\mathtt{a}_\cH$ is injective, the exactness of \eqref{eq:0-lag} translates into the exactness of
$$
0\to \frac{TN}{\mathtt{a}_\cH(\al_\cH)} \to \phi^*\left (\frac{TM}{\mathtt{a}_\cG(\al_\cG)}\right )\to \left (\frac{TN}{\mathtt{a}_\cH(\al_\cH)}\right )^*\to 0,
$$
where the first map is induced by $T\phi$ and the second by $T^*\phi\circ \varpi^\flat$.
In conclusion, we have 

\begin{prop}\label{prop:0lag}
Let $(\gro{G}{M}, \varpi)$ be a $0$-shifted symplectic groupoid. Then $\Phi:\cH\to \cG$ 
is isotropic if and only if $\phi^*\varpi=0$, and it is  lagrangian  if and only if, additionally, $\cH$ has injective anchor, the map $TN/\mathtt{a}_\cH(A_\cH) \to \phi^*(TM/\mathtt{a}_{\cG}(A_\cG))$ is injective and $\rank(TN/\mathtt{a}_\cH(A_\cH))=\frac{1}{2}\rank(TM/\mathtt{a}_{\cG}(A_\cG))$. 
\end{prop}

In particular, if the orbit spaces of $N$ and $M$ are smooth, $\phi$ induces a usual lagrangian immersion.

\subsection{The $1$-shifted case: quasi-symplectic groupoids}\label{subsec:1shift}

 It has been observed \cite{cal:lag, Lesdiablerets, saf:poi} that 1-shifted symplectic groupoids coincide with {\em quasi-symplectic groupoids} \cite{moxu}, alternatively called {\em twisted presymplectic groupoids} in \cite{burint}. Lagrangian morphisms in this setting include several variants of hamiltonian spaces, and their fibred products (as in $\S$~\ref{subsec:lagprod}) encompass several generalizations of symplectic reduction; see e.g.  \cite{cro:red,balmax}.

Quasi-symplectic groupoids arise as integrations of Dirac structures in exact Courant algebroids
 \cite{burint}. 
We will briefly recall, for later use, these objects and describe the global and infinitesimal notions of lagrangian morphisms in this context.

\subsubsection*{\underline{$1$-shifted symplectic groupoids and Dirac structures}}\label{subsubsec:quasisym}

A quasi-symplectic groupoid \cite{burint,moxu} is a Lie groupoid $\cG\rightrightarrows M$ equipped with forms $\omega\in \Omega^2(\cG)$ and $\eta\in \Omega^3(M)$ such that
\begin{equation}\label{eq:1shiftedclosed}
\gm^*\omega=\pr_1^*\omega+\pr_2^*\omega, \quad d\omega=\gs^*\eta-\gt^*\eta\quad \text{and}\quad d\eta=0,
\end{equation}
and satisfying 
\begin{equation}\label{eq:1shiftnondeg}
\ker(T\gs)_x\cap\ker(\omega)_x\cap\ker(T\gt)_x=0 \quad \forall x\in M, \quad\text{and}\quad \dim \cG=2\dim M.
\end{equation}

Conditions \eqref{eq:1shiftedclosed} say that $\omega$ and $\eta$ define a closed $1$-shifted 2-form (cf. Example \ref{ex:1-shif2-form}). In this case the 
chain map $\Lambda_\omega:\cT_\cG\to \cT^*_\cG[-1]$ defined in \eqref{nondeg-pairing} is given by
\begin{equation}\label{eq:1-non-deg}
\begin{array}{c}
      \xymatrix{  \al_\cG \ar[r]^-{\mathtt{a} } \ar[d]^-{\lambda_\omega}& TM \ar[d]^-{-\lambda_\omega^*} \\ 
     T^*M \ar[r]^{\mathtt{a}^*} & \al_\cG^*,} 
\end{array}
\end{equation}
where 
$
\lambda_\omega:\al_\cG \rightarrow T^*M$ is the map
\begin{equation}\label{eq:lambdaomeg}
\lambda_\omega(u)=(i_u\omega)|_{TM}
\end{equation}
(here we use the decomposition $T\cG|_M=TM\oplus A_\cG$). The fact that $\Lambda_\omega$ is a quasi-isomorphism amounts to the conditions 
\begin{equation}\label{eq:1shiftIM}
\ker(\mathtt{a})\cap\ker(\lambda_\omega)=0\quad\text{and}\quad \rank A_\cG=\dim M,
\end{equation}
which are equivalent to \eqref{eq:1shiftnondeg}.

\begin{rema}\label{rem:qsympmorita}
 It is shown in \cite[\S 5.2]{mat:morvb} (following Remark~\ref{rem:tancot}) that the nondegeneracy condition on closed 1-shifted 2-forms (i.e., that the chain map \eqref{eq:1-non-deg} is a quasi-isomorphism) admits a global description, namely that 
$\omega^\flat:T\cG\to T^*\cG$ is a weak equivalence of  groupoids; see also Remark \ref{rem:VB} below. 
\hfill $\diamond$
\end{rema}

We recall the infinitesimal description of
closed 1-shifted 2-forms in terms of Lie algebroids from \cite{burint,burcab}. For a Lie algebroid $A\to M$, a {\em $\eta$-closed IM 2-form} is a pair $(\lambda, \eta)$, where
$\lambda: A\to T^*M$ is a vector-bundle map (covering the identity), $\eta\in \Omega^3(M)$, such that 
\begin{equation}\label{eq:IM2}
i_{\mathtt{a}(u)}(\lambda(v))=-i_{\mathtt{a}(v)}(\lambda(u)), \quad \lambda([u,v])=\pounds_{\mathtt{a}(u)} \lambda(v) - i_{\mathtt{a}(v)}d(\lambda(u)) +  i_{\mathtt{a}(v)}i_{\mathtt{a}(u)}\eta \quad \mbox{and} \quad  d\eta=0. 
\end{equation}

By \cite[Thm. 2.5]{burint}, any closed 1-shifted 2-form $\omega+\eta$ on a Lie groupoid $\cG\rightrightarrows M$ gives rise to a $\eta$-closed IM 2-form on $A_\cG$ via
$$
\omega+\eta \mapsto (\lambda_\omega, \eta),
$$
and each $\eta$-closed IM 2-form on $A_\cG$ arises from a unique closed 1-shifted 2-form in this way when $\cG$ is source simply connected. 

Through this correspondence, one obtains the infinitesimal description of 1-shifted symplectic forms as $\eta$-closed IM 2-forms for which  \eqref{eq:1shiftIM} holds. Such IM 2-forms can be geometrically expressed as {\em twisted Dirac structures} \cite{sevwei}; indeed, for an $\eta$-closed IM 2-form $(\lambda, \eta)$ on a Lie algebroid $A$ satisfying 
\begin{equation}\label{eq:IMinf1}
\ker(\mathtt{a})\cap\ker(\lambda)=0, \quad\text{ and }\quad \rank A=\dim M
\end{equation}
(cf. \eqref{eq:1shiftIM}), the map 
\begin{equation*}
(\mathtt{a},\lambda):A \hookrightarrow TM\oplus T^*M
\end{equation*}
embeds $A$ as a Dirac structure $\mathbf{L}$ in the $\eta$-twisted Courant algebroid 
$\bT_\eta M$ (see Example~\ref{ex:twistedCourant}), in such a way that $A$ and $\mathbf{L}$ are identified as Lie algebroids; upon this identification, $\lambda$ becomes the natural projection $\mathbf{L}\to T^*M$. 

We conclude that the infinitesimal counterparts of 1-shifted symplectic groupoids are twisted Dirac structures  \cite[$\S$ 2.4]{burint}: any 1-shifted symplectic form $\omega + \eta$ 
on a Lie groupoid $\cG \rightrightarrows M$ gives rise to a $\eta$-twisted Dirac structure on $M$ by $\mathbf{L}=(\mathtt{a},\lambda_\omega)(A_\cG)$. On the other hand, if a Dirac structure $\mathbf{L}$ in $\bT_\eta M$  is integrable as a Lie algebroid then  its source-simply-connected integration carries a unique 1-shifted symplectic structure $\omega+\eta$ such that $\lambda_\omega$ is the projection $\mathbf{L}\to T^*M$.

\begin{exa}[Symplectic groupoids]\label{ex:sympg}
An important class of 1-shifted symplectic groupoids is given by usual symplectic groupoids \cite{cosgro,weisygr}, i.e., Lie groupoids $\gro{G}{M}$ equipped with an ordinary symplectic form $\omega\in \Omega^2(\cG)$ that is multiplicative (and $\eta=0$). At the infinitesimal level, these correspond to Poisson manifolds; see e.g. \cite[Ch.~14]{crfemabook} for an exposition with examples. 
\hfill $\diamond$
\end{exa}

\begin{exa}[AMM groupoid]\label{exa:amm}
Given a Lie group $K$ with $\mathrm{Ad}$-invariant quadratic Lie algebra  $(\fk,\langle\cdot,\cdot\rangle)$, the so-called {\em AMM groupoid} \cite{moxu} is  defined by the 
action groupoid $K\ltimes K\rightrightarrows K$ with respect to the action of $K$ on itself by conjugation, i.e. the target is given by $\gt(k,\gamma)=k\gamma k^{-1}$, with 1-shifted symplectic form given by 
\begin{align}  \label{eq:ammgpd}
    \omega_{\scriptscriptstyle{AMM}}=\frac{1}{2}\Big(\langle \pr_1^*\theta^l , \pr_2^*\theta^r \rangle- (k\gamma k^{-1},k)^*\langle \pr_1^*\theta^l , \pr_2^*\theta^r \rangle\Big),\qquad  %|_{(k,\gamma)}= \frac{1}{2}\big(\langle \mathrm{Ad}_\gamma \pr_1^*\theta^l,\pr_1^*\theta^l\rangle+\langle\pr_1^*\theta^l,\pr_2^*(\theta^l+\theta^r)\rangle\big),\quad 
    \eta= - \frac{1}{12}\langle\theta^l,[\theta^l,\theta^l]\rangle,
\end{align} 
%$\frac{1}{2}\big(\langle \mathrm{Ad}_\gamma \pr_1^*\theta^l,\pr_1^*\theta^l\rangle+\langle\pr_1^*\theta^l,\pr_2^*(\theta^l+\theta^r)\rangle\big)$
where $\theta^l, \theta^r \in \Omega^1(K,\fk)$ denote the left/right Maurer-Cartan 1-forms on $K$. At the infinitesimal level, this quasi-symplectic groupoid corresponds to the 
 {\em Cartan-Dirac structure} $\mathbf{L}_{\scriptscriptstyle{CD}}$  on $K$ \cite{sevwei} (see \cite[$\S$ 7]{burint}), 
$$ 
\mathbf{L}_{\scriptscriptstyle{CD}}|_\gamma = \left \{\big(u^r-u^l, \frac{1}{2}(\langle \theta^l, u\rangle+ \langle \theta^r, u\rangle)\big)|_\gamma, \,  \, u\in \fk  \right \},
$$
i.e., the $\eta$-twisted Dirac structure on $K$ defined as the image of the diagonal $\fk\hookrightarrow\fk\oplus {\fk}$ under the Courant algebroid isomorphism \eqref{eq:carcousplitting} of Example \ref{ex:actioncartan}.
\hfill $\diamond$
\end{exa}

\subsubsection*{\underline{Global and infinitesimal $1$-shifted lagrangian structures}}
 Let $(\gro{G}{M}, \omega+\eta)$ be a $1$-shifted symplectic groupoid. Consider a Lie groupoid $\gro{H}{N}$ and a morphism $\Phi:\cH\to \cG$. By Def.~\ref{def:iso}, an isotropic structure on $\Phi$ is a 2-form $\sigma\in\Omega^2(N)$ satisfying
 \begin{equation}\label{eq:1isotr}
  \delta\sigma=\gs^*\sigma-\gt^*\sigma  = \Phi^*\omega   \quad\text{ and }\quad d\sigma = \phi^*\eta.
 \end{equation}

The chain map $\Lambda_{\Phi,\sigma}$ from Prop.~ \ref{prop:iso} is given by 
    \begin{equation}\label{eq:1-lag}
    \begin{array}{c}
\xymatrix{
\al_\cH\ar[rr]^-{(
             \mathtt{a}_\cH, \Lie(\Phi) )}\ar[d]&&TN\oplus \phi^*\al_\cG\ar[d]^-{
             \sigma^\flat + T^*\phi\circ \lambda_\omega }\ar[rr]^{
             T\phi-\mathtt{a}_\cG}&& \phi^*TM\ar[d]^{-\Lie(\Phi)^*\circ\lambda_\omega^*}
             \\
0\ar[rr]&&T^*N\ar[rr]_{-\mathtt{a}^*_\cH}&&\al_\cH^*.}
    \end{array}
\end{equation}
The  equality of IM 2-forms 
$\lambda_{\delta \sigma}= \lambda_{\phi^*\omega}$ (resulting from the equality of multiplicative 2-forms $\Phi^*\omega=\delta\sigma$) reads 
\begin{equation}\label{eq:inf-iso}
    -\sigma^\flat\circ\mathtt{a}_\cH=T^*\phi\circ\lambda_\omega\circ\Lie(\Phi),
\end{equation} 
which is the identity ensuring that $\Lambda_{\Phi,\sigma}$ is a morphism of chain complexes.

Following Def.~\ref{def:gen-lag}, an isotropic morphism $(\cH, \Phi,\sigma)$ is lagrangian when $\Lambda_{\Phi,\sigma}$ is a quasi-isomorphism, or, equivalently, by \eqref{eq:exseq-lag},
when the following sequence of vector bundles over $N$ is exact:
        \begin{equation} \label{eq:exseq1lag}
         0\to\al_\cH\xrightarrow{(
             \mathtt{a}_\cH, \Lie(\Phi))} TN\oplus\phi^*\al_\cG\xrightarrow{\scriptsize\begin{pmatrix} T\phi & -\mathtt{a}_\cG \\ \sigma^\flat & T^*\phi\circ\lambda_\omega \end{pmatrix}}\phi^*TM\oplus T^*N\xrightarrow{-\Lie(\Phi)^*\circ\lambda_\omega^*+\mathtt{a}_\cH^*}\al_\cH^*\to 0.
       \end{equation}

We will now give a characterization of this lagrangian condition in terms of Dirac structures and Courant morphisms, see $\S$\ref{subsec:cmorph} for definitions and notation.

Following  Example~\ref{ex:dirmaps}, the map $\phi = \Phi|_N: N\to M$ gives rise to 
a Dirac structure $R_\phi$ in $\mathbb{T}_{\phi^*\eta}N\times \overline{\mathbb{T}_\eta M}$ supported on $\mathrm{Graph}(\phi)$ given by 
$$
R_\phi:= \{((Y,(T\phi)^*(\alpha)),(T\phi(Y),\alpha))\,|\, \, Y\in T_xN, \, \alpha \in T_{\phi(x)}^*M, \, x\in N\},
$$
while the condition $d\sigma =\phi^*\eta$ implies that 
\begin{equation}\label{eq:Rphibeta}
R_{\phi,\sigma} := \{((Y,\beta),(X,\alpha))\,|\, ((Y,\beta - i_Y\sigma),(X,\alpha)) \in R_\phi \}
\end{equation}
is a Dirac structure in $\mathbb{T}N\times \overline{\mathbb{T}_\eta M}$ supported on $\mathrm{Graph}(\phi)$.
The Dirac structures $R_{\phi}$ and $R_{\phi,\sigma}$ are examples of {\em Courant morphisms} (in the sense of \cite[$\S$ 2.2]{buriglsev}).

Let $\mathbf{L} \subseteq \mathbb{T}_\eta M$ be the Dirac structure corresponding to $(\cG, \omega+\eta)$, given by the  image of $(\mathtt{a}_\cG, \lambda_\omega): A_\cG\to \mathbb{T}M$.
The isotropic morphism $(\cH,\sigma,\Phi)$ gives rise to a Lie algebroid morphism 
$$
 A_\cH \to TN\times \mathbf{L}
$$
given by the composition
\begin{equation}\label{eq:isomap}
\al_\cH\xrightarrow{(
             \mathtt{a}_\cH, \Lie(\Phi))} TN\times \al_\cG \xrightarrow[\sim]{(\mathrm{id},(\mathtt{a}_\cG,\lambda_\omega))} TN\times \mathbf{L}.
\end{equation}
A key observation now is that condition \eqref{eq:inf-iso} is equivalent to the fact that the image of the map \eqref{eq:isomap} is contained in 
$$
R_{\phi,\sigma}\cap (TN\times \mathbf{L}).
$$
Note that whenever this intersection has constant rank, it is a Lie subalgebroid of $TN\times \mathbf{L}$ with base $\mathrm{Graph}(\phi)$.

\begin{prop}\label{prop:1lag}
An isotropic morphism $(\cH, \Phi,\sigma)$  into $(\gro{G}{M}, \omega+\eta)$ is lagrangian  if and only if the Lie-algebroid morphism $ A_\cH \to TN\times \mathbf{L}$ in \eqref{eq:isomap} is an isomorphism onto $R_{\phi,\sigma}\cap (TN\times \mathbf{L}).$
\end{prop}

\begin{proof}
The condition that \eqref{eq:isomap} is an isomorphism onto $R_{\phi,\sigma}\cap (TN\times \mathbf{L})$ is equivalent to the fact that  the map $(\mathtt{a}_\cH,\Lie(\Phi))$ is injective with image given by
$$
\im(\mathtt{a}_\cH, \Lie(\Phi)) = \{ (Y, u)\in TN\oplus \phi^*\al_\cG\ |\ T\phi(Y)=\mathtt{a}_\cG(u),\ \sigma^\flat(Y)=-T^*\phi\circ\lambda_\omega(u)\},
$$
which are the conditions giving the exactness of \eqref{eq:exseq1lag} in the first two nontrivial positions. We must show that this ensures the exactness of the whole sequence.

By dualizing the last map, we see that the exactness of \eqref{eq:exseq1lag} at $A^*_\cH$  follows from the injectivity of $(\mathtt{a}_\cG, \lambda_\omega)$  together with the injectivity of $(\mathtt{a}_\cH,\Lie(\Phi))$, while the exactness at $\phi^*TM\oplus T^*N$ is equivalent to 
$$
\im(\mathtt{a}_\cH, -\lambda_\omega\circ\Lie(\Phi)) = \{ (Y, \alpha)\in TN\oplus \phi^*T^*M\ |\ \sigma^\flat(Y)=T^*\phi(\alpha),\ \mathtt{a}^*_\cG(\alpha)=\lambda_\omega^*\circ T\phi(Y)\ \}.
$$
But we know that $\im(\mathtt{a}_\cH, \lambda_\omega\circ\Lie(\Phi))$ is contained in the right-hand side above, and the opposite inclusion holds again by the injectivity of $(\mathtt{a}_\cG, \lambda_\omega)$. 
\end{proof}

The previous discussion leads to a purely infinitesimal description of $1$-shifted lagrangian morphisms (cf. \cite[\S 7.2]{pym:symalg}). 

Let $A\to M$ and $B\to N$ be Lie algebroids,  and let $\varphi: B\to A$ be a Lie algebroid morphism covering $\phi: N\to M$. 

The infinitesimal analogue of a 1-shifted symplectic structure on $A$
is an $\eta$-closed IM 2-form $(\lambda,\eta)$ satisfying \eqref{eq:IMinf1}, or equivalently, such that $(\mathtt{a}_A,\lambda): A\to TM\oplus T^*M$ is an isomorphism onto a $\eta$-twisted Dirac structure $\mathbf{L}$. 

\begin{defi}\label{def:inf1lag}
An {\em isotropic structure} on $\varphi: B \to A$ is a 2-form $\sigma \in \Omega^2(N)$ such that
$$
d\sigma=\phi^*\eta, \quad \mbox{ and } \quad -\sigma^\flat\circ \mathtt{a}_B = T^*\phi\circ \lambda \circ \varphi,
$$
the second identity being equivalent to the condition that the image of
\begin{equation}\label{eq:infisot1}
(\mathtt{a}_B, (\mathtt{a}_A\circ \varphi, \lambda\circ \varphi)): B\to TN\times \mathbf{L}
\end{equation}
be contained in $R_{\phi,\sigma}\cap (TN\times \mathbf{L})$. We say that $\sigma$ is a {\em lagrangian structure} on $\varphi$ if, in addition, the map \eqref{eq:infisot1} is an isomorphism onto $R_{\phi,\sigma}\cap (TN\times \mathbf{L})$.
\end{defi}

\begin{prop}\label{pro:1lagcorresp}
Let $(\gro{G}{M}, \omega+\eta)$ be a $1$-shifted symplectic groupoid. If $\sigma$ is an isotropic (resp. lagrangian) structure on a morphism $\Phi: \cH \to \cG$, then it is an isotropic (resp. lagrangian) structure on $\Lie(\Phi)$, and the converse holds provided $\cH$ is source connected.
\end{prop}

\begin{proof}
Comparing the definitions of global and infinitesimal isotropic (resp. lagrangian) structures, the proposition follows from the fact that, if $\cH$ is source connected, the equality of multiplicative 2-forms $\Phi^*\omega=\delta \sigma$ is equivalent to the equality $\lambda_{\Phi^*\omega}=\lambda_{\delta \sigma}$ of IM 2-forms \cite{burint,burcab}, which is the identity $-\sigma^\flat\circ \mathtt{a}_\cH = T^*\phi\circ \lambda_\omega \circ \Lie(\Phi)$.
\end{proof}

It will be convenient to have a slight variant of Def.~\ref{def:inf1lag}:

\begin{defi}\label{def:IMLag}
An {\em infinitesimal lagrangian morphism} into a Dirac structure $\mathbf{L}\subset \mathbb{T}_\eta M$ is a triple $(N,\phi, \sigma)$, where $\phi: N\to M$ and $\sigma\in \Omega^2(N)$ are such that $d\sigma=\phi^* \eta$ and $R_{\phi,\sigma}\cap (TN\times \mathbf{L})$ has constant rank.
\end{defi}

To compare the definitions, note that
if $(N,\phi, \sigma)$ is an infinitesimal lagrangian morphism into $\mathbf{L}$, then 
the projection $R_{\phi,\sigma}\cap (TN\times \mathbf{L}) \to \mathbf{L}$ is a morphism of Lie algebroids and $\sigma$ is a lagrangian structure (as in Def.~\ref{def:inf1lag});
on the other hand, a lagrangian structure $\sigma$ on a morphism $\varphi: B\to A$  can be viewed as an infinitesimal lagrangian morphism into $\mathbf{L}=(\mathtt{a}_A,\lambda)(A)$ once one identifies $B$ and $R_{\phi,\sigma}\cap (TN\times \mathbf{L})$ as Lie algebroids via \eqref{eq:infisot1}.

\begin{exa}[Strong Dirac maps as infinitesimal lagrangian morphisms]\label{ex:strongdiracmap}
Given a $\eta$-twisted Dirac structure $\mathbf{L}$ on $M$, 
a special class of infinitesimal lagrangian morphisms $(N, \phi, \sigma)$ into $\mathbf{L}$ 
is given by those with the property that $R_{\phi,\sigma}$ is a morphism of Manin pairs from $(\mathbb{T}N, TN)$ to $(\mathbb{T}_\eta M, \mathbf{L})$ (in the sense of \cite[$\S$~2.3]{buriglsev}), see $\S$~\ref{subsec:cmorph}; in this case, the constant-rank condition in Def.~\ref{def:IMLag} follows from the stronger condition that the natural projection 
$$
R_{\phi,\sigma}\cap (TN\times \mathbf{L})\to \mathbf{L}
$$ 
is a fiberwise isomorphism. As recalled in Example~\ref{ex:dirmaps}, an infinitesimal lagrangian morphism $(N, \phi, \sigma)$ is of this type if and only if $\phi$  is a {\em strong Dirac map} from $(N,-\sigma)$ to $(M,\mathbf{L})$. When $\mathbf{L}$ is a Poisson structure on $M$, these are just Poisson maps from symplectic manifolds into $M$.
 \hfill $\diamond$
\end{exa}

As expected, we have a Lie-type correspondence between infinitesimal and global lagrangian morphisms, as a direct consequence of Prop.~\ref{pro:1lagcorresp} and the integration of Lie-algebroid morphisms (see e.g. \cite[Prop.~6.8]{moeint}).

\begin{prop}\label{prop:1lagint}
Let $(\gro{G}{M}, \omega+\eta)$ be a $1$-shifted symplectic groupoid with corresponding Dirac structure ${\bf L} \subset \bT_\eta M$.
\begin{itemize}
\item[(a)] If $(\cH, \Phi,\sigma)$ is a lagrangian morphism into $(\cG, \omega+\eta)$ then  $(N,\phi,\sigma)$ is an infinitesimal lagrangian into $\mathbf{L}$. 
\item[(b)] If $(N,\phi,\sigma)$ is an infinitesimal lagrangian into $\mathbf{L}$ and $\cH$ is a source-simply-connected Lie groupoid integrating $R_{\phi,\sigma}\cap (TN\times \mathbf{L})$, then $(\cH,\Phi,\sigma)$ is a lagrangian morphism into $(\cG, \omega+\eta)$, where $\Phi: \cH\to \cG$ is the integration of the Lie algebroid morphism $R_{\phi,\sigma}\cap (TN\times \mathbf{L})\to \mathbf{L}$ given by the natural projection.
\end{itemize}
\end{prop}

\subsubsection*{\underline{Examples of 1-shifted lagrangians, moment maps and reduction}}\label{subsubsec:hamiltonian}

We will briefly recall (see \cite{cal:lag, saf:poi, cro:red, balmax}) how 1-shifted lagrangian morphisms extend general notions of moment maps and reduction in Poisson and Dirac geometry.

We distinguish two main classes of lagrangian morphisms, referred to as lagrangian morphisms {\em of stabilizer type} (following the terminology in \cite{cro:red}) and {\em of hamiltonian type}.

Let $(\gro{G}{M}, \omega+\eta)$ be a $1$-shifted symplectic groupoid that integrates a Dirac structure ${\bf L}\subset\bT_\eta M$. 

We say that a lagrangian morphism $(\cH,  \Phi, \sigma)$ is of {\em stabilizer type} if $\Phi: \cH\to \cG$ is an immersion that restricts to an injection of unit manifolds. 
 By Prop.~\ref{prop:1lagint}, the corresponding infinitesimal lagrangian morphisms are 
given by immersed submanifolds $\phi: N\hookrightarrow M$ together with $\sigma\in \Omega^2(N)$ such that $d\sigma=\phi^*\eta$ and 
$$
R_{\phi,\sigma}\cap (TN\times \mathbf{L}) \cong 
\{ (T\phi(X),\alpha)\,|\, X\in TN,\, \phi^*\alpha = -i_X\sigma  \}\cap \mathbf{L}|_N\subset \bT M |_N
$$ 
has constant rank; such infinitesimal lagrangian morphisms are called {\em reduction levels} in \cite{balmax}. 

\begin{exa} \label{ex:isomor}
\begin{itemize}
\item[(a)] The simplest examples of lagrangian morphisms of stabilizer type are 
inclusions of isotropy groups $\cG_x\rightrightarrows \{x\}$, for $x\in M$. The corresponding infinitesimal lagrangian morphisms are points in $M$.

\item[(b)] Let $\mathcal{O}$ be an orbit of $\cG$ equipped with its natural 2-form $\omega_\mathcal{O}$ (making the inclusion $\mathcal{O}\hookrightarrow M$ into a strong Dirac map). The inclusion of the restriction $\cG|_{\mathcal{O}}\rightrightarrows \mathcal{O}$, together with $\sigma= -\omega_{\mathcal{O}}$, is a lagrangian morphism in $\cG$. The corresponding infinitesimal lagrangian morphisms are presymplectic leaves of $\mathbf{L}$.
\end{itemize}
When $x\in \mathcal{O}$, the inclusions of $\cG_x$ and $\cG|_\mathcal{O}$ into $\cG$ are Morita equivalent lagrangian morphisms, in the sense of \cite[$\S$ 2.4]{cro:moore}.
\hfill$\diamond$
\end{exa}

\begin{exa}\label{ex:prepois}
When $\mathbf{L}$ is the graph of a Poisson structure $\pi$ on $M$, reduction levels with $\sigma=0$ are
the same as (immersed) submanifolds $N\hookrightarrow M$ such that 
$$
(TN\oplus \mathrm{Ann}(TN))\cap \gra(\pi) \cong TN\cap \pi^\sharp(\mathrm{Ann}(TN))
$$
has constant rank, called {\em pre-Poisson submanifolds} in \cite{zam:pre}; special cases are coisotropic and cosymplectic submanifolds. 
\hfill $\diamond$
\end{exa}

A lagrangian morphism $(\cH, \Phi, \sigma)$ is of {\em hamiltonian type} if $\cH \rightrightarrows N$ is the action groupoid defined by an action of $\cG$ on $N$ along a map $\phi: N\to M$ and $\Phi$ is the natural projection 
$$
\cH = \cG \times_{M} N \to \cG.
$$ 

Starting with an action groupoid, by \eqref{eq:1isotr}, the conditions making the above projection into an isotropic morphism    are that $d\sigma=\phi^*\eta$ and that the graph of the action map $\cG\times_M N \to N$ is isotropic in $(\cG,\omega) \times (N,-\sigma) \times (N,\sigma)$.  The extra condition making it lagrangian is equivalent to 
$$
\ker(T\phi)\cap \ker(\sigma) = 0.
$$
This says that $(\cH, \Phi, \sigma)$ is a lagrangian morphism of hamiltonian type if and only if $(N,-\sigma)$ is a {\em hamiltonian $\cG$-space}  in the sense of \cite{moxu}; in this setting one often refers to the map $\phi:N \to M$ along which the $\cG$-action on $N$ is defined as the ``moment map'' (following \cite{mikwei}). 

At the infinitesimal level, lagrangian morphisms of hamiltonian type correspond to infinitesimal lagrangian morphisms $(N,\phi, \sigma)$ for which the projection  $R_{\phi,\sigma}\cap (TN\times \mathbf{L})\to \mathbf{L}$ is a fiberwise isomorphism, i.e., $R_{\phi,\sigma}$ is a morphism of Manin pairs; following  Example~\ref{ex:strongdiracmap}, we conclude that
an infinitesimal lagrangian morphism $(N,\phi,\sigma)$ of hamiltonian type is the same as a strong Dirac map 
$$
(N,-\sigma)\to (M,\mathbf{L}),
$$ 
i.e., a {\em presymplectic realization} of \cite[$\S$ 7.1]{burint}.

\begin{exa}\label{ex:hamspa}
 When $\cG$ is a symplectic groupoid, hamiltonian $\cG$-spaces coincide with the ones in \cite{mikwei}. The following are special cases:
 \begin{itemize}
 \item  When $\cG$ is $T^*G \rightrightarrows \mathfrak{g}^*$, the cotangent bundle of a Lie group $G$, hamiltonian $\cG$-spaces coincide with classical hamiltonian $G$-spaces. 
 \item For symplectic groupoids arising as the action groupoid with respect to the dressing action of a (complete) Poisson Lie group $G$ on its dual group $G^*$, the corresponding hamiltonian spaces coincide with Poisson $G$-spaces with $G^*$-valued moment maps in the sense of \cite{lu:mom}.
 \end{itemize}
 
 Hamiltonian spaces of the AMM groupoid of Example \ref{exa:amm} are the quasi-hamiltonian $G$-spaces of \cite{alemeimal}. See \cite[$\S$ 3.3]{moxu} and \cite[$\S$ 2.2]{burcra}.
\hfill $\diamond$
\end{exa}

Corollary \ref{cor:lagprod} provides a construction of 0-shifted symplectic groupoids by means of transverse fibred products of 1-shifted lagrangian morphisms  that encompasses various versions of ``symplectic reduction''. 

\begin{itemize}

\item The fibred product of two lagrangian morphisms of hamiltonian type are the {\em intertwiner spaces} of hamiltonian $\cG$-spaces of \cite{moxu}. 

\item For a symplectic groupoid $\cG$, the fibred product of lagrangian morphisms of hamiltonian and stabilizer types was studied in
\cite{cro:red} (see also \cite{balmax}). 

\item The special case of fibred products of lagrangian morphisms of hamiltonian type  with stabilizers corresponding to points or orbits recovers Marsden-Weinstein reduction (for $\cG=T^*G$) and quasi-hamiltonian reduction (for $\cG$ the AMM groupoid). The equivalence between reductions with respect to points and orbits (the so-called ``shift-trick'') follows from their Morita equivalence as lagrangian morphisms, recalled in Example~\ref{ex:isomor}
(see \cite[Prop.~3.8]{cro:moore}).
\end{itemize}

\section{2-shifted symplectic geometry on Lie groups}\label{sec:2shift}

Let $\omega_\bullet=\omega_2+\omega_1+\omega_0$ be a closed 2-shifted 2-form on a Lie groupoid $\gro{G}{M}$, as in Example \ref{ex:2-shif2-form}. 
In this case the map $\Lambda_{\omega_2}:\cT_\cG\to \cT^*_\cG[-2]$  in \eqref{nondeg-pairing} takes the form
\begin{equation}\label{eq:2-non-deg}
\begin{array}{c}
     \xymatrix{ 0\ar[r]\ar[d]^0& \al_\cG \ar[r]^-{\mathtt{a} } \ar[d]^-{\langle \cdot,\cdot \rangle^\flat_{\omega_2}}& TM \ar[d]^-{0} 
     \\ T^*M \ar[r]_-{-\mathtt{a}^* } & \al_\cG^*\ar[r]&0,}
\end{array}
\end{equation}
where $\langle \cdot, \cdot \rangle^\flat_{\omega_2}:\al_\cG \rightarrow \al_\cG^*$ is induced by the fiberwise pairing  $\langle \cdot,\cdot\rangle_{\omega_2}$ on $\al_\cG$,
\begin{equation}\label{eq:2-pairing}
    \langle u,v\rangle_{\omega_2}=-{\omega_2}|_{(\gu(x), \gu(x))}((u,0_{x}), (\mathtt{a}(v) ,v))+{\omega_2}|_{(\gu(x), \gu(x))}((\mathtt{a}(u),u), (v,0_{x})),
\end{equation}
for $x\in M$, and $u,v\in (\al_\cG)_x$. Therefore $\Lambda_{\omega_2}$ is a quasi-isomorphism, i.e., $\omega_\bullet$ is symplectic, if and only if  the restriction of $\langle \cdot,\cdot\rangle_{\omega_2}$ to each isotropy Lie algebra $\ker(\mathtt{a})|_x$ is nondegenerate and the anchor map $\mathtt{a}$ is surjective.

Assuming that $M$ is connected, it follows that a $2$-shifted symplectic groupoid
$(\gro{G}{M}, \omega_\bullet)$ is automatically transitive, and hence 
Morita equivalent to the isotropy subgroup $i:\cG_x\to \cG$, for any $x\in M$.
More is true: $(\cG_x, i^*\omega_\bullet)$ is a $2$-shifted symplectic group, and $(\cG, \omega_\bullet)$ and $(\cG_x, i^*\omega_\bullet)$ are Morita equivalent as 2-shifted symplectic groupoids, in the sense of \cite{cueca}. (By contrast, 
there are several nontrivial examples of 2-shifted symplectic structures on {\em Lie 2-groupoids}, see e.g. \cite{libsev:int, raj:from} and \cite[$\S$ 4.3]{tracou}).

We will henceforth focus on $2$-shifted symplectic groups.

\subsection{2-shifted symplectic groups}\label{subsec:2sg}
 Let $K$ be a Lie group, and let $\Omega\in \Omega^2(K\times K)$ and $\Theta\in \Omega^3(K)$  define a closed 2-shifted 2-form on $K$. Its differentiation under the van Est map (see App.~\ref{app:wei}) is
\begin{equation}\label{eq:VE2shift}
\mathrm{VE}(\Theta)=0,\qquad \mathrm{VE}(\Omega)= \langle \cdot,\cdot \rangle_{\Omega},
\end{equation}
where $\langle \cdot,\cdot \rangle_{\Omega}: \fk\times\fk\to \mathbb{R}$ is the bilinear form \eqref{eq:2-pairing}, which is symmetric in this case; the fact that $\Omega+\Theta$ is closed implies that so is $\langle \cdot,\cdot \rangle_{\Omega} \in S^2\fk^*$ (as an element in the Weil algebra of $\fk$, see Example~\ref{ex:weilla}), which amounts to the $\mathrm{ad}$-invariance of $\langle \cdot,\cdot \rangle_{\Omega}$.

As seen above, $\Omega+\Theta$ is symplectic if and only if $\langle \cdot,\cdot \rangle_{\Omega}$ is nondegenerate. We conclude that 2-shifted symplectic groups differentiate to {\em quadratic Lie algebras}, i.e., Lie algebras equipped with a symmetric, nondegenerate and ad-invariant pairing.
As for integration, for a 1-connected Lie group $K$, the van Est Theorem (see e.g. \cite[Thm. 5.1]{weivanest})  ensures that any quadratic structure on $\fk$ integrates to a 2-shifted symplectic form on $K$, and any two such integrations must be cohomologous.

We note that, given a Lie group $K$, an $\mathrm{Ad}_K$-invariant quadratic structure  on $\fk$ admits a distinguished integration to a 2-shifted symplectic form on $K$, as recalled in the next example.

\begin{exa}\label{ex:BG}
Let $K$  be a Lie group, and let $\langle\cdot,\cdot\rangle$ be an
$\mathrm{Ad}_K$-invariant quadratic structure on its Lie algebra $\fk$. Then\footnote{The sign on $\Theta$ reflects our convention of defining Lie brackets via right-invariant vector fields.} 
\begin{align}  \label{eq:polwie}
    \Omega=-\frac{1}{2}\langle \pr_1^*\theta^l , \pr_2^*\theta^r \rangle\in \Omega^2(K \times K), \quad \Theta= - \frac{1}{12}\langle \theta^l,[\theta^l,\theta^l]\rangle\in \Omega^3(K)
\end{align}  
is a 2-shifted symplectic form on $K$ that integrates $\langle\cdot,\cdot\rangle_\Omega=\langle\cdot,\cdot\rangle$. 

The forms $\Omega$ and $\Theta$ give a model for the 2-shifted symplectic structure on the classifying stack of $K$, see \cite{Lesdiablerets,ptvv}. These forms have appeared multiple times in physics and mathematics, e.g. in connection with Chern-Simons theory \cite{polwie}, moduli spaces of flat connections \cite{weisymmod} and loop group extensions \cite{bryloo}; the associated degree $4$ class in the Bott-Shulman-Stasheff complex of $K$ represents the universal first Pontryagin class of the pairing \cite{shuphd}. 
\hfill $\diamond$
\end{exa}

\subsection{Global and infinitesimal 2-shifted lagrangian structures}\label{subsec:2lag} 
We now discuss isotropic and lagrangian morphisms into a 2-shifted symplectic Lie group $K$. We will consider these objects at both infinitesimal and global levels. Although much of the discussion extends to general 2-shifted symplectic groupoids, we will focus on the case of Lie groups for simplicity, as they include our examples of interest.

\subsubsection*{\underline{Global and infinitesimal isotropic structures}} 
Let $K$ be a Lie group equipped with a 2-shifted symplectic form $\Omega+\Theta$, with
$\Omega\in \Omega^2(K\times K)$ and $\Theta\in \Omega^3(K)$.

Let $\gro{H}{N}$ be a Lie groupoid and $\Phi:\cH\to K$ a groupoid morphism. Following Definition \ref{def:iso}, an {\em isotropic structure on $\Phi$} consists of  forms $\sigma \in\widehat{\Omega}^2(\cH)$ and $\eta\in\Omega^3(N)$ satisfying 
 \begin{equation}\label{eq:2iso}
 \Phi^*\Omega=\delta\sigma=\pr_1^*\sigma+\pr_2^*\sigma-\gm^*\sigma,\quad \Phi^*\Theta=\gs^*\eta-\gt^*\eta-d\sigma\quad \text{and}\quad d\eta=0.
 \end{equation}

 \begin{exa} 
     The action of $K$ on itself by left multiplication defines an action groupoid $K\ltimes K\rightrightarrows K$. The projection $\Phi: K\times K\to K,\ \Phi(g, x)=g$ together with the 2-form $\sigma=\Omega$ and the 3-form $\eta=-\Theta$ determine an isotropic morphism into $K$. 
     (The groupoid morphism $\Phi$ is a model for the universal bundle $EK\to BK$ over the the classifying space of $K$.)
     \hfill $\diamond$
 \end{exa}

The next result describes, for later use, some consequences of the first condition in \eqref{eq:2iso} that extend known properties of multiplicative 2-forms (cf. \cite[$\S$~3]{burint}). 

\begin{lem}\label{lem:infmult}
The condition $\delta\sigma = \Phi^*\Omega$ implies 
the following:
\begin{itemize}
\item[(a)] For $u\in \Gamma(A_\cH)$ and $h\in \cH$,
 $$
 i_{u^r}\sigma |_h = \gt^*(\lambda_\sigma(u))|_h - \Omega|_{(1,\Phi(h))}\big((\Lie(\Phi)(u_{\gt(h)}),0),(0, T\Phi(\cdot))\big);
 $$
\item[(b)] $\gi^*\sigma = -\sigma +  \Phi^*\widehat{\Omega}$, where
$\widehat{\Omega}:=(\id_K,\gi_K)^*\Omega$ (here $(\id_K,\gi_K): K \to K\times K$, $k\mapsto (k,k^{-1})$). 
\end{itemize}
\end{lem}

\begin{proof}
The following are basic identities on Lie groupoids:
$$
u^r|_h = T\gm|_{(\gt(h),h)}(u,0), \quad \;\; Y=T\gm|_{(\gt(h), h)}(T\gt(Y), Y).
$$
It then follows from  $\gm^*\sigma = \pr_1^*\sigma+\pr_2^*\sigma - \Phi^*\Omega$  that
$$
\sigma|_h(u^r,Y) = \sigma(u, T\gt(Y)) - \Omega\big((\Lie(\Phi)(u),0),(0,T\Phi(Y))\big), 
$$
which proves (a).

Note that 
$(\id, \gi)^* (\pr_1^*\sigma+\pr_2^*\sigma-\gm^*\sigma) =\sigma + \gi^*\sigma - \gt^*\gu^*\sigma.$
Since $\gu^*\sigma=0$, we have
$$
\sigma + \gi^*\sigma = (\id,\gi)^*\Phi^*\Omega = \Phi^*(\id_K, \gi_K)^*\Omega,  
$$
which proves (b).
 \end{proof}

To describe the infinitesimal counterpart of an isotropic morphism, we replace the 2-shifted symplectic group $K$ by 
a quadratic Lie algebra $(\fk, \langle\cdot,\cdot\rangle)$, and 
consider a  Lie algebroid morphism  $\varphi:A \rightarrow \mathfrak{k}$.

\begin{defi}\label{def:2iso} An \emph{isotropic structure on $\varphi:A\to \fk$} is a pair $(\lambda, \eta)$ where $\lambda: A \rightarrow T^*N$ is a vector bundle map (covering the identity map on $N$) and $\eta\in \Omega^3(N)$ is a closed 3-form such that
    \begin{align} 
         \lambda([u,v])=&\pounds_{\mathtt{a}(u) }\lambda(v)-i_{\mathtt{a}(v) }d\lambda(u)+i_{\mathtt{a}(v) }i_{\mathtt{a}(u) }\eta+\langle d\varphi(u), \varphi(v)\rangle,\label{eq:imiso} \\ 
        - \langle \varphi(u), \varphi(v) \rangle= &  i_{\mathtt{a}(u)}\lambda (v)+ i_{\mathtt{a}(v)}\lambda (u), \label{eq:imiso2} 
    \end{align}  
for all $u, v \in \Gamma(A)$. \hfill $\diamond$
\end{defi}

We refer to the triple $(\varphi, \lambda, \eta)$ as an {\em infinitesimal isotropic morphism}. We will use the terminology {\em infinitesimal isotropic structure} for an isotropic structure on a Lie algebroid morphism.

Note that when $\fk=\{0\}$, we recover the definition of an $\eta$-closed IM $2$-form on $A$, see $\S$ \ref{subsec:1shift}.

Infinitesimal isotropic morphisms admit a natural interpretation in terms of Courant algebroids.
Regarding a quadratic Lie algebra $(\fk,\langle \cdot,\cdot\rangle)$ as a Courant algebroid (over a point),  consider the product Courant algebroid $\bT_\eta N \times \fk$, whose underlying vector bundle is $(TN\oplus T^*N)\oplus \fk_N$, with anchor $\pr_{TN}$ (the natural projection to $TN$), pairing given by
\begin{equation}\label{eq:prodpair}
(X_1 + \alpha_1 + w_1,  X_2 + \alpha_2 +  w_2 ) \mapsto  i_{X_2}\alpha_1+  i_{X_1}\alpha_2+\langle w_1,w_2\rangle,
\end{equation}
and bracket
 \begin{equation}\label{eq:prodcour}
 \begin{split}
     \llbracket  X_1 + \alpha_1 + w_1, X_2 +  \alpha_2 + w_2\rrbracket 
     =&[X_1,X_2]+ (\pounds_{X_1}\alpha_2-i_{X_2}d\alpha_1+i_{X_2}i_{X_1}\eta+\langle dw_1, w_2\rangle) \\
     & + (\pounds_{X_1}w_2-\pounds_{X_2}w_1+[w_1, w_2]),
 \end{split}
 \end{equation}
for $X_i\in \mathfrak{X}(N)$, $\alpha_i\in \Omega^1(N)$ and $w_i\in C^\infty(N;\mathfrak{k})$, $i=1,2$.
Then, 
for an  isotropic structure $(\lambda, \eta)$ on $\varphi: A \to \fk$, conditions \eqref{eq:imiso} and \eqref{eq:imiso2} are equivalent to saying that the map
$$
(\mathtt{a}, \lambda, \varphi): A \to (TN\oplus T^*N)\oplus \fk_N
$$
preserves brackets and has isotropic image. This observations leads to the following class of examples.

\begin{exa}\label{ex:isobundle}
Any isotropic, involutive subbundle $A \subseteq \bT_\eta N \times \fk$ defines an infinitesimal isotropic morphism $(\varphi, \lambda, \eta)$, where $\varphi: A\to \fk$ and $\lambda: A\to T^*N$ are the restrictions to $A$ of the natural projections $(TN\oplus T^*N)\oplus \fk_N \to \fk$ and $(TN\oplus T^*N)\oplus \fk_N\to T^*N$, respectively.
\hfill $\diamond$
\end{exa}

We now establish a Lie-type correspondence between global and infinitesimal isotropic structures. We will make use of the van Est map, recalled in Appendix \ref{app:wei},
which  provides a way to differentiate elements in  $\widehat{\Omega}^\bullet(\cH_{(\bullet)})$ to elements in the Weil algebra of  $\al_\cH$.

The first key observation is that infinitesimal isotropic structures admit an alternative formulation in terms of Weil algebras.
With the notation of Appendix~\ref{app:wei}, we have the following consequence of
 Lemma~\ref{lem:W12iso}.  

\begin{lem}\label{lem:infisot}
Isotropic structures $(\lambda, \eta)$ on $\varphi: A\to \fk$ are equivalent to pairs $(c,\eta)$, with $c\in W^{1,2}(A)$ and $\eta\in \Omega^3(N)=W^{0,3}(A)$, satisfying 
$$
(d^v + d^h)(c+\eta)= \widehat{\varphi}(\langle \cdot,\cdot \rangle);
$$ 
the explicit correspondence is given by
$$
c_1(|u)=\lambda(u), \qquad  c_0(u) = i_{\mathtt{a}(u)}\eta - d(\lambda(u)),
$$
for $u\in \Gamma(A)$. (Here $\widehat{\varphi}$ is the lift \eqref{eq:lift} of $\varphi$ to Weil algebras.)
\end{lem}

This viewpoint to isotropic structures as coboundaries for the pullback of the quadratic structure $\langle \cdot,\cdot \rangle$ relates  Def.~\ref{def:2iso} to \cite[Def 4.5]{pym:symalg}.

\begin{thm}\label{thm:intiso}
Let $(K, \Omega+\Theta)$ be a 2-shifted symplectic Lie group with corresponding quadratic Lie algebra $(\fk, \langle\cdot,\cdot\rangle_\Omega)$.
Consider a Lie groupoid $\gro{H}{N}$ and a morphism $\Phi: \cH \to K$, and let $\varphi=\Lie(\Phi): A_\cH\to \fk$. 
Then 
\begin{itemize}
\item[(a)] 
If $\sigma + \eta$ is an isotropic structure on $\Phi$, then $(\lambda_\sigma, \eta)$ is an isotropic structure on $\varphi$, where $\lambda_\sigma: A_\cH\to T^*N$ is given by $\lambda_\sigma(u)=(i_u\sigma)|_{TN}$.

\item[(b)] The correspondence between isotropic structures on $\Phi$ and $\varphi$ in  (a),
$$
\sigma+\eta\mapsto (\lambda_\sigma,\eta),
$$ 
is bijective provided $\cH$ is source-simply-connected.
\end{itemize}
\end{thm}

\begin{proof} 
To prove (a),
let $\sigma \in \widehat{\Omega}^2(\cH)$ and $\eta\in \Omega^3(N)$ define an isotropic structure on $\Phi$, and let $c+\eta := \mathrm{VE}(\sigma + \eta)$. Then 
$$
(d^v + d^h)(c+\eta)=\mathrm{VE}(\partial (\sigma + \eta))=\mathrm{VE}(\Phi^*(\Omega+\Theta))=
\widehat{\varphi}(\mathrm{VE}(\Omega+\Theta))= \widehat{\varphi}(\langle \cdot,\cdot \rangle_\Omega),
$$ 
where $\widehat{\varphi}$ is the lift of $\varphi$ as in \eqref{eq:lift}, and we use \eqref{eq:VE2shift} and \eqref{eq:VEfunct}. As noted above, $c+\eta$ is equivalent to an isotropic structure $(\lambda,\eta)$ on $\varphi$ into $(\fk,\langle \cdot,\cdot \rangle_\Omega)$, where
$$
\lambda(u) = c_1(|u) = \mathrm{VE}(\sigma)_1(|u) = \gu^*(i_{u^r}\sigma),
$$
for $u\in \Gamma(A_\cH)$, so (a) is proven.

Differentiating an isotropic structure as above is an injective operation when $\cH$ is source-connected: indeed, if two isotropic structures $\sigma_1 + \eta_1$ and $\sigma_2+\eta_2$ differentiate to the same $c+\eta$, then $\eta_1=\eta_2=\eta$ and $\sigma_1-\sigma_2$ is a multiplicative 2-form on $\cH$ whose image under the van Est map is zero, so they must coincide (e.g. by \cite[Prop.~5.2]{burcab}).

To prove (b), it remains to show that the differentiation of isotropic structures on $\Phi$ to isotropic structures on $\varphi$ is surjective when $\cH$ is source-simply-connected. So let $(\lambda,\eta)$ be an isotropic structure on $\varphi$, and consider the corresponding $c\in W^{1,2}(A_\cH)$ satisfying $(d^v+d^h)(c+\eta)=\widehat{\varphi}(\langle \cdot,\cdot\rangle_\Omega)$, as in Lemma~\ref{lem:infisot}.

If $\cH$ is source-simply-connected, then by the van Est Theorem \cite[Thm. 5.1]{weivanest}, the map in cohomology induced by the van Est map, $H^2(\widehat{\Omega}^2(\cH_{(\bullet)}),\delta) \rightarrow H^2(W^{\bullet,2}(A), d^h)$  is injective.
Since
$$
d^hc = \widehat{\varphi}(\langle \cdot,\cdot\rangle_\Omega) = \mathrm{VE}(\Phi^*\Omega),
$$
there exists $\sigma'\in \widehat{\Omega}^2(\cH)$ such that
$\delta \sigma'=\Phi^*\Omega.$
So $\tau:=\mathrm{VE}(\sigma')-c$ is $d^h$-closed and 
by \cite[Prop.~5.2]{burcab}, there exists a unique multiplicative 2-form $\sigma''\in \widehat{\Omega}^2(\cH)$ (i.e., $\delta \sigma''=0$) such that $\mathrm{VE}(\sigma'')= \tau$. As a consequence, $\sigma:= \sigma'-\sigma''$ satisfies
$\mathrm{VE}(\sigma) = c$ and 
$$
\delta \sigma = \Phi^*\Omega.
$$
Since $d^h\eta + d^vc=0$, we have that
$$
\mathrm{VE}(\Phi^* \Theta+ d\sigma)=-d^v c=d^h \eta= \mathrm{VE}( \delta \eta).
$$
Now note that $\Phi^* \Theta+ d\sigma$ and $ \delta \eta$ are both multiplicative 3-forms on $\cH$, so it follows from the injectivity of the van Est map on multiplicative forms \cite[Prop.~5.2]{burcab} that they must coincide.
Comparing with \eqref{eq:2iso}, we see that $\sigma+\eta$ is an isotropic structure on $\Phi$ whose image under $\mathrm{VE}$ is $c+\eta$. \end{proof}

\begin{coro}\label{prop:2-iso} 
If $(\cH,\Phi, \sigma+\eta)$ is an isotropic morphism into $(K, \Omega+\Theta)$ then the map
\[ 
(\mathtt{a}_\cH,\lambda_\sigma,\Lie(\Phi)): \al_\cH \rightarrow   (TN\oplus T^*N)\oplus \fk_N \] 
is bracket preserving (with respect to \eqref{eq:prodcour}) and has isotropic image (with respect to \eqref{eq:prodpair}).
\end{coro}

\subsubsection*{\underline{Lagrangian structures}} 
Let $(K, \Omega+\Theta)$ be a 2-shifted symplectic Lie group with corresponding quadratic Lie algebra $(\fk, \langle\cdot,\cdot \rangle)$.

Let $(\cH, \Phi, \sigma+\eta)$ be an isotropic morphism into  $(K, \Omega+\Theta)$.
In this case, the chain map $\Lambda_{\Phi,\sigma}$ introduced in Proposition \ref{prop:iso} is given by 
 \begin{equation}\label{eq:2-lag}
 \begin{array}{c}
     \xymatrix{\al_\cH\ar[rr]^-{(\mathtt{a}_{\cH},\Lie(\Phi))}\ar[d]_{\lambda_\sigma}&&TN\oplus \mathfrak{k}_N\ar[d]^-{-\lambda_\sigma^* -  \Lie(\Phi)^*\circ \langle\,,\,\rangle^\flat}
\\
T^*N\ar[rr]_{\mathtt{a}_\cH^*}&&\al_\cH^*
.}
 \end{array}
\end{equation} 
By Def.~ \ref{def:gen-lag}, the  isotropic structure $\sigma+\eta$ on $\Phi: \cH \to K$ is {\em lagrangian} if the map \eqref{eq:2-lag} is a quasi-isomorphism. We can express this condition more explicitly as follows.

\begin{prop}\label{prop:lag}
Let $(\cH, \Phi, \sigma+\eta)$ be an isotropic morphism into a 2-shifted symplectic group $(K, \Omega+\Theta)$. The following are equivalent:
\begin{enumerate}
    \item[(a)] The isotropic structure is lagrangian.
    \item[(b)] The map 
    \begin{equation}\label{eq:j}
(\mathtt{a}_\cH,\lambda_\sigma, \Lie(\Phi)): \al_\cH \rightarrow   (TN\oplus T^*N)\oplus \fk_N
    \end{equation}
    is injective and its image is a Dirac structure in $\bT_\eta N\times \fk$.
    \item[(c)] The following conditions hold:
        \begin{itemize}
            \item $\dim \cH-2 \dim N=\frac{1}{2}\dim K$,
            \item $\ker (\sigma)_h\cap \ker (T\Phi)_h \cap \ker (T \gs)_h \cap \ker (T\gt)_h=\{0\}$, \quad $\forall h\in \cH$. 

            (And it is enough to check this condition for $h\in N$.)
        \end{itemize}
\end{enumerate}
\end{prop}
\begin{proof}
By \eqref{eq:exseq-lag}, (a) is equivalent to the exactness of the following sequence:
    \begin{align}  
    \xymatrix{0 \ar[r] & \al_\cH \ar[rrr]^-{(\mathtt{a}_\cH,\lambda_\sigma, \Lie (\Phi))} &&&  (TN\oplus T^*N)\oplus\fk_N  \ar[rrrr]^-{-\lambda_\sigma^*-\mathtt{a}^*_\cH - \Lie (\Phi)^*\circ \langle\cdot, \cdot\rangle^\flat }&&&& \al_\cH^* \ar[r] & 0 },\label{eq:exseq2lag} 
    \end{align}
which is equivalent to the map $(\mathtt{a}_\cH,\lambda_\sigma, \Lie (\Phi))$ being injective and  that its image is a lagrangian subbundle of $(TN\oplus T^*N)\oplus\fk_N$. By Corollary \ref{prop:2-iso}, this lagrangian subbundle is necessarily a Dirac structure, so (a) and (b) are equivalent.

Let us now consider the kernel condition in (c). Any element in $\ker(T\gs)_h\cap \ker(T\Phi)_h\cap \ker(T\gt)_h$ is the right translation $v^r=(T\gr_h)_{\gt(h)}(v)$, for 
$$
v\in (A_\cH)_{\gt(h)}\cap \ker(T\Phi)_{\gt(h)}\cap \ker(T\gt)_{\gt(h)} =  \ker(\Lie(\Phi))_{\gt(h)}\cap \ker(\mathtt{a}_\cH)_{\gt(h)}.
$$
Moreover, for $v\in \ker(T\Phi)_{\gt(h)}$, it follows from Lemma \ref{lem:infmult} that
$$
i_{v^r}\sigma = \gt^*(\lambda_\sigma(v)),
$$
so that $v^r\in \ker(\sigma)$ if and only if $v\in \ker(\lambda_\sigma)$. It follows that 
$$
\ker (\sigma)_h\cap \ker (T\Phi)_h \cap \ker (T \gs)_h \cap \ker (T\gt)_h=\{0\}
$$
if and only if the map \eqref{eq:j} is injective.
Noticing that the dimension condition in (c) is equivalent to $\rank A_\cH=\frac{1}{2} \rank(TN \oplus T^*N \oplus{\fk}_N)$, we conclude that (b) and (c) are equivalent.
\end{proof}

Note that, if a $2$-shifted symplectic group $(K,\Omega+\Theta)$ admits  a lagrangian morphism then the quadratic structure $\langle \cdot,\cdot \rangle$ on $\fk$ must have split signature.

For a more clear comparison with quasi-symplectic groupoids (see \eqref{eq:1shiftedclosed} and \eqref{eq:1shiftnondeg}),
we consolidate the constituents and properties of 2-shifted lagrangian morphisms into the following definition.

\begin{defi}\label{def:2shiftlag}
A lagrangian morphism into $(K, \Omega+\Theta)$ consists of a morphism of Lie groupoids $\Phi: \cH \to K$ together with $\sigma \in \widehat{\Omega}^2(\cH)$ and a closed 3-form $\eta \in \Omega^3(N)$ satisfying
$$
\pr_1^*\sigma+\pr_2^*\sigma-\gm^*\sigma =  \Phi^*\Omega,\qquad \gs^*\eta-\gt^*\eta-d\sigma=\Phi^*\Theta,
$$
and such that
$$
\dim \cH-2 \dim N=\frac{1}{2}\dim K,
            \qquad \ker (\sigma)_h\cap \ker (T\Phi)_h \cap \ker (T \gs)_h \cap \ker (T\gt)_h=\{0\}, \;\; \forall h\in \cH.
$$
\hfill $\diamond$
\end{defi}

This definition recovers that of quasi-symplectic groupoids when $K$ is a point.
In general, it is convenient to think of such lagrangian morphisms as ``relative quasi-symplectic groupoids'' with respect to a morphism to $K$.

\begin{rema}[VB-groupoids viewpoint]\label{rem:VB}
Lagrangian morphisms into $(K, \Omega+\Theta)$ have also a global formulation in terms of VB-groupoids, generalizing the global description of quasi-symplectic forms on Lie groupoids in Remark~\ref{rem:qsympmorita}. For an isotropic morphism $(\cH, \Phi, \sigma+\eta)$, we consider the VB-groupoids $T\cH$ and $\Phi^* T^*K$ over $\cH$, see e.g. \cite{burcabhoy}. Let 
$$
T\cH\oplus_{\Omega} \Phi^* T^*K \rightrightarrows TN\oplus {\fk}_N
$$ 
be their direct sum, with its natural VB-groupoid structure modified by $\Omega$ as follows:
$$
\gm_{\Omega}((V_1,\xi_1)_{h_1}, (V_2,\xi_2)_{h_2}):= (T\gm_\cH(V_1,V_2),\zeta)_{h_1h_2},
$$
where $\zeta\in (T^*K)_{\Phi(h_1)\Phi(h_2)}$ is defined by the condition $\gm_K^*\zeta = \pr_1^*\xi_1+ \pr_2^*\xi_2 - i_{(T\Phi(V_1),T\Phi(V_2))}\Omega$ (the condition $\delta\Omega=0$ ensures that $\gm_\Omega$ is a VB-groupoid multiplication). The fact that $\delta \sigma = \Phi^*\Omega$ implies that 
\begin{equation}
   \widehat{\Lambda}_{\Phi,\sigma}: T\cH\oplus_{\Omega} \Phi^* T^*K \to T^*\cH, \quad (V,\xi) \mapsto i_V\sigma - \Phi^*\xi, \label{eq:iso str as vb map}
\end{equation}
is a morphism of VB-groupoids which (together with the identification $\fk\cong \fk^*$ via $\langle\, \cdot, \cdot\,\rangle$)  induces the morphism ${\Lambda}_{\Phi,\sigma}$ in  \eqref{eq:2-lag} at the level of core complexes.
As a consequence of \cite[Thm.~3.5]{mat:morvb}, $(\Phi, \sigma+\eta)$ is a lagrangian morphism (equivalently, \eqref{eq:2-lag} is a quasi-isomorphism) if and only if
$\widehat{\Lambda}_{\Phi,\sigma}$ is a weak equivalence. When $K$ is a point, this recovers the characterization of quasi-symplectic groupoids in \cite[Prop. 5.4]{mat:morvb}.
\hfill $\diamond$\end{rema}

We pass to the description of infinitesimal lagrangian structures.
Given a  quadratic Lie algebra $(\fk,\langle\cdot,\cdot\rangle)$ and a Lie algebroid morphism $\varphi: A\to \fk$, following
Prop.~\ref{prop:lag} (b) the lagrangian condition on an isotropic structure $(\lambda,\eta)$ on $\varphi$ is that the map
\begin{equation}\label{eq:inf2lag}
(\mathtt{a},\lambda,\varphi): A \to (TN\oplus T^*N) \oplus \fk_N
\end{equation} 
embeds $A$ as a Dirac structure  in $\bT_\eta N \times \fk$. By identifying $A$ with its image under this map, we see that any lagrangian structure $(\lambda, \eta)$ on $\varphi$ is equivalent to a Dirac structure 
$\bl \subset \bT_\eta N\times {\fk},$
with $\varphi$ and $\lambda$ identified with the natural projections $\bl \to \fk$ and $\bl\to T^*N$, respectively (cf. Example~\ref{ex:isobundle}).

This leads to the 
following definition.

\begin{defi}\label{def:inflag}
    A (infinitesimal) \emph{lagrangian morphism} into a quadratic Lie algebra $(\fk,\langle\cdot,\cdot\rangle)$ is a Dirac structure in a Courant algebroid of type $\bT_\eta N\times {\fk}$, for a given closed 3-form $\eta$.
\end{defi}

This viewpoint to infinitesimal lagrangian morphisms as Dirac structures is in line with the more general decription of 2-shifted lagrangian structures as Dirac structures in twisted Courant algebroids in \cite{pym:symalg}.

The Lie-type correspondence between quasi-symplectic groupoids and  twisted Dirac structures \cite{burint}   (recalled in $\S$ \ref{subsubsec:quasisym}) extends to the context of lagrangian morphisms through a refinement of the correspondence between global and infinitesimal isotropic structures in Theorem~\ref{thm:intiso}:

\begin{coro} \label{cor:int2lag} 
Consider a  2-shifted symplectic group $K$ and its corresponding quadratic Lie algebra $\fk$. Then:
\begin{itemize}
\item[(a)] Any lagrangian morphism $(\cH, \Phi, \sigma+\eta)$ into $K$ defines  a lagrangian morphism  into $\fk$ given by the Dirac structure
$$
\bl = (\mathtt{a}_\cH, \lambda_\sigma, \mathrm{Lie}(\Phi))(A_\cH) \subseteq \bT_\eta N \times \fk.
$$
\item[(b)] Consider a lagrangian morphism into $\fk$ defined by a Dirac structure ${\bf L} \subseteq \mathbb{T}_{\eta} N\times {\mathfrak{k}}$ that is integrable as a Lie algebroid, and let  $\gro{H}{N}$ be its source-simply-connected integration. Then there is a unique lagrangian morphism $(\cH, \Phi, \sigma+\eta)$ with the property that $\Lie(\Phi)$ is the projection ${\bf L} \rightarrow \fk$ and $\lambda_\sigma$ is the projection ${\bf L}\rightarrow T^*N$. 
\end{itemize}    
\end{coro}

Part (a) follows from Prop.~\ref{prop:lag} (b), and part (b) follows from Theorem~\ref{thm:intiso} (b).

\begin{rema}\label{rem:Lopp}
It is clear that $(\cH,\Phi, \sigma+\eta)$ is a lagrangian morphism into $(K, \Omega+\Theta)$ if and only if $(\cH,\Phi, -(\sigma+\eta))$ is a lagrangian morphism into $(K, -(\Omega+\Theta))$; at the infinitesimal level, if the former corresponds to the Dirac structure $\mathbf{L}\subseteq \bT_\eta N \times \fk$, the latter corresponds to
$$
\overline{\mathbf{L}}:= \{((X, \alpha),k)\;|\; ((X, -\alpha),k)\in \mathbf{L}  \} \subseteq \bT_{-\eta}N \times \overline{\fk}.
\vspace{-5mm}$$

\hfill $\diamond$
\end{rema}

\begin{rema}
As a natural extension of Def.~\ref{def:inflag}, one can consider Dirac structures on {\em heterotic Courant algebroids} of principal bundles, as in \cite[\S 3.4]{barhek}; while this paper focuses on langrangian structures on (strict)  morphisms of groupoids, Dirac structures in these more general transitive Courant algebroids arise as lagrangian structures on {\em generalized} groupoid morphisms (also called {\em Hilsum-Skandalis maps}  \cite[$\S$ 2.5]{hs:genmor}) . 
\hfill $\diamond$ 
\end{rema}

\subsection{Examples of lagrangian morphisms} 
We identify below some examples of global and infinitesimal 2-shifted lagrangian morphisms that naturally arise  in Poisson geometry.

\begin{exa}[Lagrangian subgroups]\label{ex:lagsub}
Let $(\mathfrak{d}, \langle \cdot,\cdot\rangle)$ be a quadratic Lie algebra.
The simplest example of a lagrangian morphism into $\mathfrak{d}$ is given by a lagrangian subalgebra (i.e., a Dirac structure) $\g\subset \fd$. In this case $(\fd,\g)$ is referred to as a {\em Manin pair}.

Let $D$ be a Lie group integrating $\fd$ with the property that $\langle \cdot,\cdot\rangle$ is $\mathrm{Ad}_D$-invariant, and equip $D$ with the 2-shifted symplectic form \eqref{eq:polwie}. Let $\phi_G: G\to D$ be a morphism of Lie groups integrating the inclusion $\g\to \fd$.  Then $\sigma=0$ is a lagrangian structure on  $\phi_G: G\rightarrow D$. 
\hfill $\diamond$
\end{exa}

\begin{exa}[Exact action Courant algebroids]\label{ex:exactaction}
Suppose that a Lie algebra $\fk$ acts on a manifold $M$ via
 $\mathtt{a}: \fk_M \to TM$. We denote by $\fk\ltimes M$ the corresponding action Lie algebroid, and by 
$$
\varphi: \fk\ltimes M \to \fk
$$ 
the Lie algebroid map given by the natural projection. As recalled in App.~\ref{subsec:appaction}, if $\fk$ has a quadratic structure and the $\fk$-action on $M$
has coisotropic stabilizers, then $\fk_M$ also carries the structure of an {\em action Courant algebroid}  \cite{coupoi}. Let us assume that the $\fk$-action on $M$ is transitive (i.e., $\mathtt{a}: \fk_M \to TM$ is surjective) and has lagrangian stabilizers, so that the action Courant algebroid $\fk_M$ is exact. We claim that isotropic splittings of this Courant algebroid are the same as lagrangian structures on the map $\varphi$ into the quadratic Lie algebra $\overline{\fk}$:
\begin{equation}\label{eq:lagspli}
\begin{Bmatrix}
\mbox{Isotropic splittings} \\
\mbox{of $\fk_M$}
\end{Bmatrix} \;\; \rightleftharpoons \;\;    
\begin{Bmatrix}
\mbox{Lagrangian structures} \\
\mbox{on $\varphi: \fk\ltimes M \to \overline{\fk}$}
\end{Bmatrix}.
\end{equation}
\smallskip

To establish this correspondence, recall (see Example~\ref{ex:twistedCourant}) that an isotropic splitting $s: TM\to \fk_M$ of  the exact sequence
$$
0 \to T^*M \stackrel{\mathtt{a}^*}{\to} \fk_M^*\cong \fk_M\stackrel{\mathtt{a}}{\to} TM\to 0
$$
defines a closed 3-form $\eta$ on $M$ by \eqref{eq:exacteta}, and an isomorphism of Courant algebroids,
\begin{equation}\label{eq:CAiso}
(\mathtt{a},s^*): \fk_M \stackrel{\sim}{\to} \bT_\eta M.
\end{equation}
The diagonal embedding $\fk_M\hookrightarrow \fk_M\oplus \fk_M$ identifies the action Lie algebroid $\fk\ltimes M$ with a Dirac structure in the product Courant algebroid $\fk_M \times \overline{\fk}$; by \eqref{eq:CAiso} the image of the map 
$$
(\mathtt{a}, s^*, \mathrm{id}): \fk_M \to (TM\oplus T^*M) \oplus \fk_M 
$$ 
is a Dirac structure in $\bT_\eta M\times \overline{\fk}$, and therefore
$(\varphi, s^*,\eta)$ is a lagrangian morphism into $\overline{\fk}$. On the other hand, given any lagrangian structure $(\lambda,\eta)$ on $\varphi$ into $\overline{\fk}$, condition \eqref{eq:imiso2} ensures that the map $(\mathtt{a}, \lambda): \fk_M\to TM\oplus T^*M$ is injective, and hence an isomorphism of vector bundles (since they have the same rank). Conditions \eqref{eq:imiso} and \eqref{eq:imiso2} say that this map is an isomorphism from the action Courant algebroid $\fk_M$ to $\bT_\eta M$, which implies that $s=\lambda^*: TM\to \fk_M^*\cong \fk_M$ is an isotropic splitting, with $\eta$ the corresponding 3-form (see Example~\ref{ex:twistedCourant}). \hfill $\diamond$
\end{exa}

\begin{exa}[Dressing actions on homogeneous spaces]\label{ex:homspa}
As a special case of the setup in Example~\ref{ex:exactaction}, we consider a Manin pair $(\fd, \g)$ (as in Example~\ref{ex:lagsub}), a Lie group $D$ integrating $\fd$, and a closed subgroup $G$ of $D$ with Lie algebra $\g$. 

The quotient $S=D/G$ (with respect to the action of $G$ on $D$ by right multiplication) carries a transitive $D$-action by left multiplication, sometimes called the {\em dressing action}.  The infinitesimal dressing action $\mathtt{a} : \fd_S \to TS$ has lagrangian stabilizers, so we obtain an action Courant algebroid $\fd_S$ that is exact. 
 Following Example~\ref{ex:exactaction}, isotropic splittings of $\fd_S$ are equivalent to lagrangian structures on the projection 
$$
\varphi: \fd\ltimes S \to \overline{\fd}.
$$
At the global level, considering a 2-shifted symplectic group $D$ integrating $\fd$ and the action groupoid $D\ltimes S$, such infinitesimal lagrangian structures should correspond (at least upon suitable conditions for integration) to lagrangian structures 
on the projection $\Phi: D\ltimes S \to D$.

As a concrete example of such global lagrangian morphism,  
let $(G, \Omega+\Theta)$ be a 2-shifted symplectic group, and consider 
$D=G^2 = \overline{G}\times G$ with $G$ identified with its diagonal subgroup. In this case $S\cong G$, and the action groupoid $D\ltimes S$ can be identified with the groupoid of arrows $G^I$  of Example \ref{ex:arro-grou};  an explicit lagrangian structure $\sigma + \eta$ on the projection 
$$
G^2\ltimes G\to G^2
$$ 
is given in that example by
$$
\eta=\Theta, \quad \mbox{ and } \quad \sigma_{((g_1,g_2), \gamma)} = (g_1\gamma g_2^{-1}, g_2)^*\Omega - (g_1,\gamma)^*\Omega.
$$
Infinitesimally, this lagrangian morphism into $\overline{G}\times G$ corresponds to the Dirac structure $\mathbf{L}\subseteq  \bT_\Theta G\times (\overline{\g} \times \g)$ given by 
\begin{equation}\label{eq:dir-arr}
\mathbf{L}|_g = \{(u^r_g-v^l_g,\Omega|_{(1,g)}((0,\cdot),(u,0))-\Omega|_{(g,1)}((\cdot,0),(0,v)),u,v) \; | \; (u,v)\in \g\times \g\},
\end{equation}
which corresponds under \eqref{eq:lagspli} to  the isotropic splitting (see Example~ \ref{ex:carcouca}) $TG\to (\g\times \overline{\g})_G$,
$$
X_g\mapsto ( \Omega_{|(1,g)}((0,X),(\cdot,0)), \Omega_{|(g,1)}((X,0),(0,\cdot))) \in (\g\times \g)^* \cong \g\times \g.
$$ \hfill $\diamond$
\end{exa}

The next two examples illustrate infinitesimal lagrangian morphisms given by Dirac structures 
$$
\mathbf{L}\subset \bT M \times \fd,
$$
where $\fd=\g \oplus \g^*$ is a quadratic Lie algebra that is the double of a Lie quasi-bialgebra (see $\S$ \ref{app:c1}).

\begin{exa}[Quasi-Poisson manifolds]\label{ex:qpoi}
Given a Lie quasi-bialgebra $(\g, F, \chi)$,  a \emph{quasi-Poisson $\g$-manifold} \cite{alekos}  is a manifold $M$ equipped with a $\g$-action $\rho:\g_M\to TM$ and a bivector field $\pi\in \mathfrak{X}^2(M)$ satisfying
\begin{align*}
&\frac{1}{2}[\pi,\pi]=\rho(\chi)\quad\text{and}\quad \pounds_{\rho(u)}\pi=-\rho(F(u)) \qquad \forall u\in \g.
\end{align*}

The following are special cases:
\begin{itemize}
\item When $\g=\{0\}$, we just have a Poisson manifold $(M,\pi)$; 
\item If $F=0$, $\chi=0$, we have a $\g$-invariant Poisson structure on $M$;
\item When $\chi=0$,  $(\g,F)$ is a Lie bialgebra, $(M,\pi)$ is a Poisson manifold and $\rho$ is referred to as a  {\em Poisson $\g$-action}.
\end{itemize}

Consider the map 
$$
\Lambda^\sharp:T^*M\oplus\g_M\to TM\oplus \g^*_M,\quad \Lambda^\sharp(\alpha,u)=\big(\pi^\sharp(\alpha)+\rho(u), -\rho^*(\alpha)\big).
$$
Let $\fd=\g\oplus \g^*$ be the Drinfeld double of $(\g, F, \chi)$. It is proven in \cite[Theorem 4.1]{quadir} that the defining conditions for a quasi-Poisson $\g$-manifold are equivalenty expressed by the fact that graph of $\Lambda^\sharp$ is a Dirac structure
\begin{equation}\label{eq:qpoisdirac}
\gra(\Lambda^\sharp) = \{((\pi^\sharp(\alpha)+\rho(u),\alpha),(u,-\rho^*(\alpha)))\,|\, \alpha\in T^*M,\, u\in \g  \}\subset \bT M\times {\fd},
\end{equation}
so any quasi-Poisson $\g$-manifold  is codified by a lagrangian morphism into $\fd$. We will study the global counterparts of such infinitesimal lagrangian morphisms in $\S$ \ref{sec:qpoisson}.\hfill $\diamond$
\end{exa}

\begin{exa}[Generalized dynamical $r$-matrices]\label{exa:r-matrix}
 Let $(M,\pi)$ be a Poisson manifold, and consider $T^*M$ with its induced Lie algebroid structure. Let $\g$ be a Lie algebra and $\vartheta\in\mathfrak{X}^1(M)\otimes\g$ 
be such that $\vartheta^\sharp: T^*M \to \g$ is a morphism of Lie algebroids.
A \emph{generalized dynamical $r$-matrix coupled with $(M,\pi)$} via $\vartheta$ \cite[Def.~4.4]{xu:local} is a function $\mathfrak{r}\in C^\infty(M,\wedge^2\g)$ such that  
 \begin{itemize}
\item $\frac{1}{2}[\vartheta,\vartheta]=[\mathfrak{r},\vartheta]-\pi^\sharp(d\mathfrak{r})$
\item $\chi:= \mathrm{Alt} (\vartheta^* d\mathfrak{r}) + \frac{1}{2}[\mathfrak{r},\mathfrak{r}] \in C^\infty(M,\wedge^3\g)$ is constant and ad-invariant, i.e., $\chi\in (\wedge^3\g)^\g$,
\end{itemize}
 where $\mathrm{Alt}$ denotes total skew-symmetrization. These objects encompass classical and dynamical $r$-matrices, as well as Poisson manifolds, see e.g. \cite[Ex.~4.6]{xu:local}.

One can verify (see the proof of \cite[Thm.~3.2]{xu:local}) that all defining conditions for a generalized dynamical $r$-matrix coupled with $(M,\pi)$ via $\vartheta$ are equivalent to 
\begin{equation}\label{eq:cobound}
\frac{1}{2}[\Lambda,\Lambda]=\chi \in (\wedge^3\g)^\g,
\end{equation}
where $\Lambda:= \pi +\vartheta + \mathfrak{r} \in \Gamma(\wedge^2 (TM\oplus \g_M))$. But \eqref{eq:cobound} holds if and only if the graph of the map
$$
\Lambda^\sharp: T^*M \oplus \g^*_M \to TM \oplus \g_M, \quad \Lambda^\sharp(\alpha,\xi)=(\pi^\sharp(\alpha)-{\vartheta^\sharp}^*(\xi), \vartheta^\sharp(\alpha) + \mathfrak{r}^\sharp(\xi)),
$$
defines a Dirac structure 
\begin{equation}\label{eq:Lsharp}
\gra(\Lambda^\sharp)\subset \bT M\times {\fd},
\end{equation}
where $\fd$ is the double of the Lie quasi-bialgebra $(\g, 0, \chi)$. (More generally, for a Lie quasi-bialgebra $(\g, F, \chi)$ and Dirac structures defined by elements $\Lambda \in \Gamma(\wedge^2 (TM\oplus \g_M))$ as in  \eqref{eq:Lsharp},  see \cite[Thm.~3.2]{xu:local}.)

\hfill $\diamond$
\end{exa}

Other interesting examples of Dirac structures on transitive Courant algebroids can be obtained along the lines of \cite[$\S$ 6.1]{cort:dar}.

\subsection{Fibred products and quasi-symplectic groupoids}\label{sec:2fibpro}
We will consider a special case of the fibred product construction in Theorem \ref{thm:comp-lag} for lagrangian morphisms into 2-shifted symplectic groups, pointing out its infinitesimal interpretation in terms of (reduction of) Dirac structures.

We will make use of the following special setup,  see App.~\ref{app:c1}. 
 Let $E_1$ and $E_2$ be Courant algebroids and $\fk$ a quadratic Lie algebra. Consider Dirac structures $\mathbf{L_1}\subset E_1\times \fk$ and $\mathbf{L}_2\subset  E_2 \times \overline{\fk}$, and let
 $$
 \mathbf{L}_1\times_{\fk} \mathbf{L}_2 = \{ ((e_1,w),(w,e_2))\,|\, (e_1,w)\in \mathbf{L}_1,\, (e_2,w)\in \mathbf{L}_2 \} \subset (E_1\oplus \fk_{M_1})\times (\overline{\fk}_{M_2}\oplus E_2).
 $$
Let $\mathfrak{p}: (E_1\oplus \fk_{M_1})\times (\overline{\fk}_{M_2}\oplus E_2) \to E_1\times E_2$ denote the natural projection.

\begin{lem}\label{lem:diracred}
If $\mathbf{L}_1\times_{\fk} \mathbf{L}_2$ is a vector bundle, then its image $\mathbf{L}:=\mathfrak{p}(\mathbf{L}_1\times_{\fk} \mathbf{L}_2)$ is a Dirac structure in $E_1\times E_2$. In particular, this occurs when the images of the projections $\mathbf{L}_1\to \fk$ and $\mathbf{L}_2\to \fk$ generate $\fk$, in which case $\mathfrak{p}$ maps $\mathbf{L}_1\times_{\fk} \mathbf{L}_2$ isomorphically onto $\mathbf{L}$. 
\end{lem}

\begin{proof}
One can regard the Courant algebroid $E_1\times E_2$ as the result of a ``coisotropic reduction'' (in the sense of \cite[$\S$ 2.1]{coupoi}) of the Courant algebroid $E_1\times \fk \times \overline{\fk} \times E_2$ with respect to $C= E_1\times \fk_\nabla \times E_2$, where $\fk_\nabla$ is the diagonal in $\fk\times \fk$: note that $C^\perp = 0 \times \fk_\nabla \times 0$, so that
$$
E_1\times E_2=C/C^\perp.
$$
Reducing the Dirac structure $\mathbf{L}_1 \times \mathbf{L}_2$ amounts to considering
the image  of $(\mathbf{L}_1 \times \mathbf{L}_2)\cap C =  \mathbf{L}_1\times_{\fk} \mathbf{L}_2$ under the projection $C\to C/C^\perp$, which is a Dirac structure whenever $(\mathbf{L}_1 \times \mathbf{L}_2)\cap C$ has constant rank, see e.g.  \cite[Thm.~7.11]{bcmz:red} and \cite{bur:brief}. This proves the first statement.

Note that $\mathbf{L}_1\times_{\fk} \mathbf{L}_2$ is the kernel of the map $\mathbf{L}_1\times \mathbf{L}_2\to \fk$ given by the difference of the projections $\mathbf{L}_1\to \fk$ and $\mathbf{L}_2\to \fk$; if this map is onto then 
$$
\rank(\mathbf{L}_1\times_\fk \mathbf{L}_2) = \rank(\mathbf{L}_1\times \mathbf{L}_2) - \dim(\fk) = \rank(E_1\times E_2)/2 = \rank(\mathbf{L}),
$$ 
which proves the second statement.
\end{proof}

The previous lemma can be directly applied to fibred products of lagrangian morphisms, at infinitesimal and global levels.

Let $\fk$, $\fk_1$ and $\fk_2$ be quadratic Lie algebras. Consider Lie algebroids $A_1\to N_1$, $A_2\to N_2$, and morphisms
$$
A_1\stackrel{(\varphi_1,\varphi)}{\longrightarrow} \fk_1\times \fk, \quad 
A_2\stackrel{(\psi_2,\psi)}{\longrightarrow} \fk_2\times \fk.
$$
Assuming that $A_1\times A_2\stackrel{\varphi-\psi}{\longrightarrow}\fk$ is onto, it follows that  $A_1\times_\fk A_2=\{(u,v)\in A_1\times A_2\,|\, \varphi(u)=\psi(v)\}$ carries a natural Lie algebroid structure such that $(\varphi_1,\psi_2): A_1\times_\fk A_2 \to \fk_1\times \fk_2$ is a morphism (see e.g. \cite[$\S$ A.2]{burcabhoy}).

\begin{thm}\label{thm:inf2lagprod}
Let $(\lambda_1,\eta_1)$ be a lagrangian structure on $(\varphi_1,\varphi)$ into $\fk_1\times \fk$, and $(\lambda_2,\eta_2)$ a lagrangian structure on $(\psi_2,\psi)$ into $\fk_2\times \overline{\fk}$. Suppose that $A_1\times A_2\stackrel{\varphi-\psi}{\longrightarrow}\fk$ is onto. Then $$(\lambda,\eta):=(\pr_1^*\lambda_1 + \pr_2^*\lambda_2, \pr_1^*\eta_1 + \pr_2^*\eta_2)$$ is a lagrangian structure on $(\varphi_1,\psi_2): A_1\times_\fk A_2 \to \fk_1\times \fk_2$.
\end{thm}

\begin{proof}
It is a direct verification that the pair $(\lambda,\eta)$ defines an isotropic structure on $(\varphi_1,\psi_2): A_1\times_\fk A_2 \to \fk_1\times \fk_2$.

For the isotropic morphisms $((\varphi_1,\varphi),\lambda_1, \eta_1)$, $((\psi_2,\psi),\lambda_2, \eta_2)$ and $((\varphi_1,\psi_2),\lambda,\eta)$, denote by $j_1$, $j_2$ and $j$ the corresponding maps defined in \eqref{eq:inf2lag}.
We know that $j_1$ and $j_2$ identify $A_1$ and $A_2$, respectively, with Dirac structures $\mathbf{L_1}\subset (\bT_{\eta_1}N_1\times \fk_1)\times \fk$ and $\mathbf{L_2}\subset (\bT_{\eta_2}N_2\times \fk_2)\times \overline{\fk}$. With these identifications, the map $j$ becomes the projection of $\mathbf{L}_1\times_\fk \mathbf{L}_2$ into 
$$
(\bT_{\eta_1}N_1\times \fk_1)\times (\bT_{\eta_2}N_2\times \fk_2) \cong \bT_\eta(N_1\times N_2)\times (\fk_1\times \fk_2),
$$
which is an isomorphism onto a Dirac structure by Lemma~\ref{lem:diracred}. Therefore $(\lambda,\eta)$ is a lagrangian structure.
\end{proof}

For the global formulation of the previous result, let $(K,\Omega+\Theta)$ and $(K_i,\Omega_i+\Theta_i)$, $i=1,2$, be 2-shifted symplectic Lie groups.

\begin{thm}\label{thm:2lagprod} Consider lagrangian morphisms 
\begin{equation*}
    (\cH_1,(\Phi_1,\Phi),\sigma_1+\eta_1) \quad \text{into} \quad {K}_1\times {K}\quad \text{and}\quad  (\cH_2,(\Psi_2,\Psi),\sigma_2+\eta_2)\quad \text{into}\quad K_2\times \overline{K}.
\end{equation*}  
If $\Phi: \cH_1\to K$ and $\Psi: \cH_2\to K$ are transverse maps, then
$$
\Big(\mathcal{H}_1 \times_{K} \mathcal{H}_2, (\Phi_1,\Psi_2), (\pr_1^*\sigma_1+\pr_2^*\sigma_2)+(\pr_1^*\eta_1+\pr_2^*\eta_2)\Big)
$$
is a lagrangian morphism into $K_1\times K_2$. 
\end{thm}

\begin{proof}
As in the infinitesimal case, the fact that $\sigma:= \pr_1^*\sigma_1+\pr_2^*\sigma_2 \in \Omega^2(\mathcal{H}_1 \times_{K} \mathcal{H}_2)$ and $\eta:= \pr_1^*\eta_1+\pr_2^*\eta_2 \in \Omega^3(N_1\times N_2)$ define an isotropic structure on  $(\Phi_1,\Psi_2)$ is a direct calculation. 

To show that $\sigma+\eta$ is a lagrangian structure, it remains to check
the condition in Prop.~\ref{prop:lag} (b), i.e., 
that the infinitesimal isotropic structure $(\lambda_\sigma,\eta)$ on $\mathrm{Lie}(\Phi_1,\Psi_2)$ is lagrangian. Using the natural identification of the Lie algebroid of  $\mathcal{H}_1 \times_{K} \mathcal{H}_2$ with $A_{\mathcal{H}_1}\times_{\fk} A_{\mathcal{H}_2}$, this is a consequence of Theorem~\ref{thm:inf2lagprod}.
\end{proof}

When $K_1$ and $K_2$ are points, Thms.~\ref{thm:inf2lagprod} and \ref{thm:2lagprod} give a way to produce Dirac structures and quasi-symplectic groupoids integrating them. 

\begin{coro}\label{cor:2shiftedfibprod}
Suppose that $(\cH_i,\Phi_i, \sigma_i+\eta_i)$, $i=1,2$, are lagrangian morphisms into $K$ such that $\Phi_1$ and $\Phi_2$ are transverse maps. Then
\begin{itemize}
\item[(a)] the subbundle of $T(N_1\times N_2)\oplus T^*(N_1\times N_2)$ given by
$$
\{((X_1,X_2),(\alpha_1,\alpha_2)) \,|\, \exists k\in \fk \,\mbox{ with }\, (X_1,\alpha_1,k)\in \mathbf{L}_1,\, (X_2,-\alpha_2,k)\in \mathbf{L}_2 \}
$$ 
is a Dirac structure in $\bT_\eta(N_1\times N_2)$, where $\eta = \pr_1^*\eta_1-\pr_2^*\eta_2$. 
\item[(b)] The Lie groupoid $\cH_1\times_K \cH_2 \rightrightarrows N_1\times N_2$ together with  $\sigma = \pr_1^*\sigma_1 - \pr_2^*\sigma_2$ is a quasi-symplectic groupoid integrating the above $\eta$-twisted Dirac structure.
\end{itemize}
\end{coro}

We now illustrate the construction of quasi-symplectic groupoids in the previous corollary  with examples arising from the groupoid of arrows  of Example \ref{ex:arro-grou}. 
Other applications to integration of Poisson and Dirac structures of interest will be discussed in $\S$ \ref{sec:qpoisson}.
(An application to the integration of Manin triples in transitive Courant algebroids can be found in \cite{tracou}.)

Let $(G, \Omega+\Theta)$ be the 2-shifted symplectic Lie group of Example \ref{ex:BG}, defined by an $\mathrm{Ad}_G$-invariant symmetric, nondegenerate bilinear form on $\g$.
Recall that the groupoid of arrows of $G$ is the action groupoid $G^I=(G\times G)\ltimes G\rightrightarrows G$ with respect to the action $(g_1,g_2) \gamma=g_1\gamma g_2^{-1}$.  We know from  Example \ref{ex:arro-grou} that the evaluation morphism $G^I \to \overline{G}\times {G}$, $(g_1,\gamma,g_2)\mapsto (g_1,g_2)$, has a lagrangian structure given by 
$$
\sigma=(g_1\gamma g_2^{-1},g_2)^*\Omega-(g_1,\gamma)^*\Omega\in\Omega^2(G^I)\quad\text{and}\quad \eta=\Theta\in\Omega^3(G).
$$
Following Cor.~ \ref{cor:int2lag}, its corresponding Dirac structure 
in $\bT_\Theta G\times\bar{\g}\times\g$
is given by 
$$
\mathbf{L}_{G^I} |_g = \left \{((u^r-v^l,\frac{1}{2}\langle \theta^r, u\rangle+\frac{1}{2}\langle \theta^l, v\rangle )_g , u,v) \ |  \,  \, u, v\in \g  \right \}.
$$
We will consider the lagrangian fibred product of the groupoid of arrows with other lagrangian morphisms into $\overline{G}\times G$.

\begin{exa}[The AMM groupoid revisited]\label{ex:ammasfibredproduct} 
The diagonal $G_\Delta$ in ${\overline{G}\times {G}} $ is lagrangian with respect to $\sigma=0$, and its corresponding Dirac structure is
$$
\bl_\Delta=\{(u,u) \ |\ u\in\g \}\subseteq \bar{\g}\times\g.
$$
Since the maps from $G^I$ and $G_\Delta$ to $\overline{G}\times G$ are transverse, we can take their lagrangian fibred product as in Corollary \ref{cor:2shiftedfibprod}; the resulting Dirac structure in $\bT_\Theta G$ is the Cartan-Dirac structure on $G$, 
$$
\mathbf{L}_{G^I}\times_{\bar{\g}\times\g}\bl_\Delta=\bl_{\scriptscriptstyle{CD}},
$$
and the quasi-symplectic groupoid $G^I\times_{\overline{G}\times G}G_{\Delta}$ is the AMM-groupoid of Example \ref{exa:amm}, given by
the action groupoid $G\ltimes G\rightrightarrows G$ with respect to the conjugation action $(g,  \gamma)\mapsto  g\gamma g^{-1}$ equipped with the  $1$-shifted symplectic structure \eqref{eq:ammgpd}. 

For a related construction of the AMM-groupoid in terms of Chern-Simons theory, see \cite{saf:qua}.
\hfill $\diamond$ 
\end{exa}

\begin{exa}[Integration of the Gauss-Dirac structure]\label{ex:gauss-dirac}
Suppose that $G$ is the complexification of a  compact, connected Lie group $K$.
(In this setting, a natural choice of quadratic structure on $\g$ defining $\Omega+\Theta$
is the real part of the complexification of a positive definite invariant inner product on $\Lie(K)$.)

Let $\fd=\bar{\g}\oplus{\g}$. The choice of opposite Borel subalgebras $\mathfrak{b}_+,\mathfrak{b}_-\subset \g$  gives rise to the lagrangian Lie subalgebra
 $$
 \mathfrak{s}=\{(u,v)\in \mathfrak{b}_+\oplus\mathfrak{b}_-\ |\ \pr_{\mathfrak{t}}(u)+\pr_{\mathfrak{t}}(v)=0\}\subset \fd,
 $$ 
 where $\pr_{\mathfrak{t}}:\mathfrak{b}_\pm\to \mathfrak{t}$ are the projections to the Cartan subalgebra $\mathfrak{t}=\Lie(T)$.  The image of $\mathfrak{s}$  in $\bT_{\Theta} G$ under the trivialization \eqref{eq:carcousplitting} defines the {\em Gauss-Dirac structure} on $G$, see \cite[$\S$ 3.6]{purspi}

If $B_\pm$ are the Borel subgroups determined by $\mathfrak{b}_\pm$, then $\mathfrak{s}$ is integrated by the Lie group
$$
S=\{ (b_+,b_-)\in B_+\times B_-\ |\ \pr_{T}(b_+)\pr_{T}(b_-)=1\},
$$ 
where $\pr_{T}:B_\pm\to T$ are the projections. Just as in Example \ref{ex:ammasfibredproduct}, the inclusion $S\hookrightarrow  \overline{G}\times G$ is a lagrangian morphism with respect to $\sigma=0$ that is transverse to $G^I\to \overline{G}\times G$.
Following Corollary \ref{cor:2shiftedfibprod}, 
the quasi-symplectic groupoid $G^I\times_{\overline{G}\times G} S$ is the action groupoid 
$S\ltimes G\rightrightarrows G$, with respect to the action  
$$
((b_+,b_-),\gamma) \mapsto b_+\gamma b_-^{-1},
$$  
equipped with the $1$-shifted symplectic structure
$$
\sigma=(b_+\gamma b_-^{-1},b_-)^*\Omega-(b_+,\gamma)^*\Omega\quad\text{and}\quad \eta=\Theta.
$$
Its corresponding Dirac structure $\bl_{G^I}\times_{\overline{\g}\times\g}\mathfrak{s}\subset \bT_\Theta G$ coincides with the Gauss-Dirac structure.
\hfill $\diamond$ 
\end{exa}

\section{From 2-shifted lagrangian morphisms to quasi-Poisson groupoids}\label{sec:lag-to-poi}

It is a general fact that $n$-shifted lagrangian morphisms have an associated $(n-1)$-shifted Poisson structure \cite{saf:der}. In this section, we show how to obtain the $1$-shifted Poisson structures corresponding to lagrangian morphisms into $2$-shifted symplectic groups.
Specifically, building on the fact that $1$-shifted Poisson structures on Lie groupoids are equivalent to quasi-Poisson groupoids \cite{bon:shif}, we provide a procedure to obtain quasi-Poisson groupoids (up to twists) from lagrangian morphisms. 

A key observation that underpins our procedure is that any lagrangian morphism $(\cH, \Phi, \sigma+\eta)$ into $K$ integrating $\mathbf{L} \subset E:= \bT_\eta N \times \fk$ yields a {\em morphism of multiplicative Manin pairs} (see $\S$ \ref{app:mmmp}) 
$$
R: (\bT \cH, T\cH)\dashrightarrow (E\times \overline{E}, \mathbf{L}\times \mathbf{L}).
$$
From the point of view of \cite{lbmein:int}, $R$ is an ``integration'' of the Manin pair $(E, \mathbf{L})$. The link with 1-shifted Poisson structures arises from the fact that quasi-Poisson groupoids can also  be formulated as morphisms of Manin pairs, as we now recall.

\subsection{{Quasi-Poisson groupoids as morphisms of multiplicative Manin pairs}}
A  {\em quasi-Poisson groupoid} \cite{ilx:univer} is a Lie groupoid $\gro{G}{M}$ equipped with a multiplicative bivector field $\pi\in\fX^2(\cG)$ and $\chi\in\Gamma(\wedge^3A_\cG)$ satisfying 
\begin{equation}\label{eq:qpoisg}
\frac{1}{2}[\pi, \pi]=\chi^r+\gi(\chi)^l,\quad \quad \quad [\pi, \chi^r]=0.
\end{equation}
Here we denote by $\gi: A_\cG=\ker(T\gs)|_M \to \ker(T\gt)|_M$  the isomorphism induced by the groupoid inversion, keeping  the same notation for its extension to exterior powers.
(When $\cG =G$ is a Lie group, $\gi(u)=-u$ for $u\in \g$, so 
the first condition in \eqref{eq:qpoisg} reads $\frac{1}{2}[\pi, \pi]=\chi^r - \chi^l$.)

We will make use of the related notions of {\em Lie quasi-bialgebroids}  and their {\em quasi-Poisson actions}, recalled in $\S$ \ref{app:c1} and $\S$ \ref{subsec:cmorph}.

Quasi-Poisson groupoids are the global counterparts of {\em Lie quasi-bialgebroids} \cite{ilx:univer}  (see also \cite[$\S$ 3.3]{burdru}):
the  Lie quasi-bialgebroid corresponding to $(\cG, \pi, \chi)$ is $(A, d_*, \chi)$, where $d_*: \Gamma(\wedge^\bullet A)\to \Gamma(\wedge^{\bullet+1}A)$ is defined by the conditions 
$$
\pounds_{u^r}\pi = [\pi, u^r] = - (d_*u)^r, \qquad i_{\gt^*df}\pi = [\pi, \gt^*f] = - (d_*f)^r
$$
for $f\in C^\infty(M)$ and $u\in \Gamma(A)$. In this case we say that $(\cG, \pi, \chi)$ {\em integrates} $(A, d_*, \chi)$.

To link quasi-Poisson groupoids with morphisms of Manin pairs, we first note that quasi-Poisson groupoids can be described by quasi-Poisson actions. 

\begin{lem}\label{lem:quasigrpact}
Given a Lie groupoid $\cG\rightrightarrows M$ with Lie algebroid $A$, a multiplicative $\pi\in \mathfrak{X}^2(\cG)$ makes $(\cG, \pi, \chi)$ into a quasi-Poisson groupoid integrating a Lie quasi-bialgebroid $(A, d_*, \chi)$ if and only if 
the action of the product Lie algebroid $A\times A$ on $\cG$ along $(\gt,\gs):\cG\to M\times M$ by $(u,v)\mapsto u^r+\gi(v)^l$ defines a quasi-Poisson action of the Lie quasi-bialgebroid $(A, -d_*, \chi)\times (A, d_*, \chi)$ on $(\cG, \pi)$. 
\end{lem}

The lemma follows from a direct comparison of \eqref{eq:qpoisg} with the defining properties of quasi-Poisson actions recalled in Example~\ref{ex:quasipact}.

The next important fact, outlined in  Example~\ref{ex:quasipact}, is that any quasi-Poisson action of a Lie quasi-bialgebroid can be equivalently expressed as a morphism of Manin pairs. 

Denote the Drinfeld double of the Lie quasi-bialgebroid $(A, d_*, \chi)$ by $E= A\oplus A^*$, so that the Drinfeld double of $(A, -d_*, \chi)$ is identified with $\overline{E}$.
In the specific context of Lemma~\ref{lem:quasigrpact}, the quasi-Poisson action is encoded in a morphism  of Manin pairs  (see \eqref{eq:qpoisdirac2})
\begin{equation}\label{eq:qpoismmap}
R: (\bT \cG, T\cG) \dashrightarrow ({E}\times \overline{E}, A\times A)
\end{equation}
over the groupoid morphism $(\gt,\gs):\cG\to M\times M$ that is {\em multiplicative} ($\S$ \ref{app:mmmp}).  

Conversely, any given morphism of  Manin pairs \eqref{eq:qpoismmap} over $(\gt,\gs):\cG\to M\times M$ defines  a quasi-Poisson action of $(A, -d_*, \chi)\times (A, d_*, \chi)$ on $\cG$ via
\begin{itemize}
\item the Lie-algebroid action of $A\times A$ on $\cG$ given by the map $\rho: (\gt,\gs)^*(A\times A) \to T\cG$ defined by the condition
\begin{equation}\label{eq:MPaction2}
((\rho(u,v),0),(u,v))\in R;
\end{equation}

\item the bivector field $\pi$ on $\cG$ given by $\pi^\sharp: T^*\cG\to T\cG$, where $\pi^\sharp(\alpha)$ is uniquely determined by the property that there exists $(\xi,\zeta)\in A^*\times A^*$ such that 
\begin{equation}\label{eq:biv2}
((\pi^\sharp(\alpha), \alpha), (\xi,\zeta))\in R.
\end{equation}
Equivalently, we can write 
$$
\gra(\pi^\sharp) = R\circ (A^*\times A^*),
$$
where ``$\circ$'' denotes composition of relations (see Example~\ref{ex:quasipact}).
\end{itemize}
By Lemma~\ref{lem:quasigrpact}, the given morphism of Manin pairs $R$ defines the structure of a quasi-Poisson groupoid on $\cG$ if the induced action of $A\times A$ on $\cG$ in \eqref{eq:MPaction2} is $(u,v)\mapsto u^r+\gi(v)^l$ and $R$ is multiplicative (in which case $\pi$ defined by \eqref{eq:biv2} is multiplicative).
The next result is a reformulation of this fact using the equivalence between Lie quasi-bialgebroids and quasi-Manin triples (see $\S$ \ref{app:c1}). 

\begin{prop}\label{prop:qpManin}
Let $\gro{\cG}{M}$ be a Lie groupoid whose Lie algebroid $A$ fits into a Manin pair $(E,A)$. Let $R: (\bT \cG, T\cG) \dashrightarrow({E}\times \overline{E}, A\times A)$
be a morphism of multiplicative Manin pairs over the groupoid morphism $(\gt,\gs): \cG \to M\times M$ such  that the induced action of $A\times A$ on $\cG$  is $(u,v)\mapsto u^r+\gi(v)^l$. Then, for any choice of isotropic complement $\mathbf{C}$ for $A$ in $E$, the induced bivector field $\pi$ on $\cG$,
$$
\gra(\pi^\sharp) = R\circ (\mathbf{C} \times \mathbf{C}),
$$
makes $\cG$ into a quasi-Poisson groupoid integrating the quasi-bialgebroid defined by the quasi-Manin triple $(E, A, \mathbf{C})$. 
\end{prop}

Moreover, quasi-Poisson groupoids corresponding to different choices of isotropic complements are {\em twist equivalent}, as in \cite[$\S$ 4.4]{ilx:univer}.

\subsection{{2-shifted lagrangian morphisms as morphisms of CA-groupoids}}
Let $(K,\Omega+\Theta)$ be a 2-shifted symplectic Lie group and  $(\gro{H}{N}, \Phi, \sigma+\eta)$ an isotropic  morphism.

Consider the closed 2-shifted 2-form on the  pair groupoid $N\times N\rightrightarrows N$ given by
$$0\in\widehat{\Omega}^2(N\times N\times N)\quad\text{and}\quad \widetilde{\eta}:=\pr_1^*\eta - \pr_2^*\eta\in\widehat{\Omega}^3(N\times N)$$
 and the exact CA-groupoids
$$
\CAg{(N\times N)}{}{\widetilde{\eta}}
 \quad \mbox{ and } \quad \CAg{K}{\Omega}{\Theta},
$$
obtained by twisting the standard CA-groupoids in each case by the given closed 2-shifted 2-forms
as in Example~\ref{ex:stacagpd}.

Consider the groupoid morphism 
$$
\Psi:\mathcal{H}\rightarrow (N\times N) \times K, \qquad \Psi(h)=((\gt(h),\gs(h)),\Phi(h)).
$$ 

\begin{prop}\label{prop:coumor}
Given an isotropic morphism  $(\gro{H}{N}, \Phi, \sigma+\eta)$ into a 2-shifted symplectic group  $(K,\Omega+\Theta)$, the relation
\begin{equation}\label{eq:Rps}
R_{\Psi,-\sigma}=\{((U, \Psi^*\xi-i_U\sigma),(T\Psi(U), \xi)) \; | \; U\in T\cH, \,  \xi\in T^*(N^2\times K)\}
\end{equation}
is a morphism of CA-groupoids from $\bT \cH$ to $\CAg{(N\times N)}{}{\widetilde{\eta}}\times \CAg{K}{\Omega}{\Theta}$. Furthermore,
$(\cH, \Phi, \sigma+\eta)$ is lagrangian if and only if $R_{\Psi,-\sigma}\cap ((T\cH\oplus 0)\times 0)=0$.
\end{prop}

\begin{proof}
The isotropic condition $\partial (\sigma+\eta) =  \Phi^*(\Omega+\Theta)$ is equivalent to
$$
\partial \sigma = -\delta\eta + \Phi^*(\Omega+\Theta) = (\gt^*\eta - \gs^*\eta) +  \Phi^*(\Omega+\Theta),
$$
which can be written as
$$
\partial \sigma = \Psi^*(\widetilde{\eta} + \Omega + \Theta).
$$
As explained in Example~\ref{ex:stacagpd}, this last condition implies that the relation
$R_{\Psi,-\sigma}$
is a morphism of CA-groupoids $\bT \cH \dashrightarrow\CAg{(N\times N)}{}{\widetilde{\eta}}\times \CAg{K}{\Omega}{\Theta}$. 

The last assertion of the proposition is a direct verification using Proposition \ref{prop:lag}.
\end{proof}

\begin{rema}\label{rem:vb2iso} 
As a morphism of CA-groupoids, $R_{\Psi,-\sigma}$ has a VB-groupoid structure (over $\gra(\Psi)$) as a VB-subgroupoid of $\bT \cH\times \bT N^2 \times \bT K$ (see $\S$~\ref{app:mmmp}). We note that the VB-groupoid $T\cH \oplus_\Omega \Phi^*T^*K$ in Remark~\ref{rem:VB} is a VB-subgroupoid of $R_{\Psi,-\sigma}$ via 
$$
T\cH \oplus_\Omega \Phi^*T^*K\hookrightarrow R_{\Psi,-\sigma}, \qquad (U,\xi)\mapsto 
((U, \Phi^*\xi -i_U\sigma),(T\Psi(U), (0,0,\xi))).
$$
The composition of this inclusion with the natural projection of $R_{\Psi,-\sigma}$ on $T^*\cH$ agrees with the morphism of VB-groupoids $\widehat{\Lambda}_{\Phi,\sigma}$ in Remark~\ref{rem:VB}, modulo a sign. 
\hfill $\diamond$
\end{rema}

Through the isomorphism of Courant algebroids $\bT_{-\eta}N \cong \overline{\bT_\eta N}$ via $(X,\alpha)\mapsto (X, -\alpha)$, see Remark \ref{rem:Lopp}, we obtain an identification of CA-groupoids of $\CAg{(N\times N)}{}{\widetilde{\eta}}$ with the pair groupoid $\bT_\eta N \times \overline{\bT_{\eta} N}$ (Example~\ref{ex:pairCA}); upon this identification, $R_{\Psi,-\sigma}$ defines a morphism of CA-groupoids 
$$
\bT \cH \dashrightarrow \bT_\eta N \times \overline{\bT_{\eta} N} \times \CAg{K}{\Omega}{\Theta}.
$$ 
By composing this morphism with  the isomorphism of CA-groupoids given by
\[ 
R':= \text{id} \times (\gt_\Omega,\gs_\Omega): (\bT_\eta N \times \overline{\bT_\eta N}) \times \CAg{K}{\Omega}{\Theta}\to (\bT_\eta N \times \overline{\bT_\eta N}) \times (\fk\times\overline{\fk}) \cong  {E}\times\overline{{E}},
\] 
where we used the identification $\fk\cong \fk^*$ via the pairing, see Example~\ref{ex:carcouca}, we obtain a morphism of CA-groupoids
\begin{equation}\label{eq:R} 
R:= R'\circ R_{\Psi,-\sigma} :\mathbb{T}\cH \dashrightarrow  E \times  \overline{E} 
\end{equation}
covering $(\gt,\gs): \cH \to N\times N$.

\begin{thm}\label{thm:intmp}
Let $(\gro{H}{N},\Phi,\sigma+\eta)$ be a lagrangian morphism to $(K,\Omega+\Theta)$, 
and let us identify  $A_\cH$ with a Dirac structure in $E:= \bT_\eta N \times \fk$
via \eqref{eq:j}. Then $R$ is a morphism of multiplicative Manin pairs,
$$
R: (\mathbb{T}\cH, T\cH) \dashrightarrow  (E \times  \overline{E}, A_\cH\times A_\cH), 
$$
whose induced $A_\cH\times A_\cH$-action on $\cG$ is $(u,v)\mapsto u^r+\gi(v)^l$.
\end{thm}

\begin{proof}

To verify that $R$ is a morphism of multiplicative Manin pairs that induces the action $(u,v)\mapsto u^r+\gi(v)^l$, we proceed with the following steps. 
\begin{enumerate}
    \item The first step is showing that, for any $u\in \Gamma(A_\cH)$, the  sections
    $(u,0)\in \Gamma(E \times  \overline{E} )$ (recall that we are identifying $A_\cH$ with its image in $E$ under \eqref{eq:j}) and $(u^r,0)\in \Gamma(\bT \cH)$  satisfy 
    \begin{equation}
   (u^r, 0)|_{\epsilon(N)} \sim_R (u,0)|_{N_\Delta}, \label{eq:courel}
\end{equation}
where $N_\Delta\hookrightarrow N\times N$ is the diagonal; i.e., the sections are $R$-related when restricted to the corresponding groupoid units. This will be proven below.

    \item 
    Assuming that \eqref{eq:courel} holds, it follows from the multiplicativity of $R$ that\footnote{Noticing that $(u^r,0)$ and $(u,0)$ are {\em core sections} of the VB-groupoids $\bT \cH$ and $E\times E$  \cite[\S 3.2]{raj:vb-group}, we use the the following general fact. For each $i=1,2$, let $\Gamma_i\rightrightarrows E_i$ be a VB-groupoid over $\cG_i\rightrightarrows M_i$, let $\sigma_i$ be a core section of $\Gamma_i$, and consider a relation $R$ given by a VB-subgroupoid of $\Gamma_1\times \Gamma_2$ over a morphism $\varphi: \cG_1\to \cG_2$. Then $\sigma_1|_{M_1}\sim_{R} \sigma_2|_{M_2}$ implies that $\sigma_1 \sim_{R} \sigma_2$; this is a consequence of the properties
    $\sigma_1|_g = \sigma_1|_{\gt(g)}\cdot 0_g$ and $\sigma_2|_{\varphi(g)} = \sigma_2|_{\gt(\varphi(g))}\cdot 0_{\varphi(g)}$ of core sections \cite[Eq. (3.3)]{raj:vb-group}, and $0_g\sim_R 0_{\varphi(g)}$, for all $g\in \cG_1$.}
    $$
    (u^r, 0) \sim_R (u,0).
    $$
Using groupoid inversion on each side, it also follows from the multiplicativity of $R$ that 
 $( \gi(u)^l,0) \sim_R (0,u)$ for all $u \in \Gamma(A_\cH)$. Therefore
\begin{equation}\label{eq:mrel}
    u^r+\gi(v)^l \sim_R (u,v), \qquad \forall \, (u,v)
\in \Gamma(A_\cH \times  A_\cH). 
\end{equation}
\item As a last step we show that $R$ is a morphism of Manin pairs
(in this case, \eqref{eq:mrel} ensures that $R$ induces the desired action); i.e., we verify that the projection $$R\cap (T\cH \times (A_{\cH}\times A_{\cH})) \to A_{\cH}\times A_{\cH}$$ is a fiberwise isomorphism. The fact that it is a fiberwise surjection follows from \eqref{eq:mrel}, while injectivity follows from the lagrangian property of 
$R_{\Psi,-\sigma}$ in Prop.~\ref{prop:coumor} and the fact that $R'$ is a fiberwise isomorphism.
\end{enumerate}

Therefore the proof will be complete once we prove that \eqref{eq:courel} holds. Since $R$ is the composition of $R_{\Psi,-\sigma}$ and $R'$, we will consider each of these relations at a time. 

Let $\varphi=\mathrm{Lie}(\Phi)$. Take $u\in \Gamma(A_\cH)$, and define $\alpha \in \Gamma(\fk^*)$ by
$$
{\alpha}:= - \Omega|_{(1,1)}((\varphi(u),0),(0,\cdot)).
$$
Recalling $(\gt_\Omega, \gs_{\Omega})$ from \eqref{eq:stCA},
a direct calculation shows that
$$
\gs_\Omega|_1(\varphi(u),\alpha)=0, \quad \gt_\Omega|_1(\varphi(u),\alpha)= \langle \varphi(u), \cdot \rangle,
$$
where in the second identity we use how the pairing $\langle\cdot,\cdot\rangle$ is defined in terms of $\Omega$, see  \eqref{eq:2-pairing}, This  implies that 
$$
\Upsilon=((\mathtt{a}(u) , \lambda_\sigma(u)),(0,0),(\varphi(u),\alpha)) \in \Gamma((\bT_\eta N \times \overline{\bT_\eta N} \times \bT_\Theta K)|_{N_\Delta\times\{1\}})
$$
is $R'$-related to the section
$$
(((\mathtt{a}(u), \lambda_\sigma(u)),(0,0)),(\varphi(u)),0) 
$$
of $((\bT_\eta N \times \overline{\bT_\eta N}) \times (\fk\times\overline{\fk}))|_{N_\Delta}$.
Since we are using \eqref{eq:j} to embed $A_\cH$ in $E$, 
in the notation of \eqref{eq:courel} this means that
$$
\Upsilon \sim_{R'} (u,0)|_{N_\Delta},
$$

We finally note that
$(u,0)|_{\epsilon(N)}\sim_{R_{\Psi,-\sigma}} \Upsilon$ is a consequence of the identity
\[ 
i_u \sigma = \gt^*\lambda_\sigma(u) +\varphi^*\alpha,
\]
which holds by Lemma~\ref{lem:infmult} (a). This concludes the proof of \eqref{eq:courel}, and hence of the theorem. 
\end{proof}

As a morphism of multiplicative Manin pairs, $R$ is an integration of the Manin pair $(E, A_\cH)$ in the sense of \cite[$\S$ 3.4]{lbmein:int}. From this viewpoint, Theorem~\ref{thm:intmp} connects two possible types of integrations of a Dirac structure in $\bT_\eta N \times \fk$: one as an infinitesimal $2$-shifted lagrangian morphism, and the other as a Manin pair.

\subsection{{Quasi-Poisson groupoids induced from 2-shifted lagrangian morphisms}}\label{sec:1shifpoi}

The next result is an immediate consequence of Proposition~\ref{prop:coumor} and Theorem~\ref{thm:intmp}.

\begin{coro}\label{cor:quapoi}
Given a lagrangian structure on a morphism $\cH \to (K, \Omega+\Theta)$, 
each choice of isotropic complement $\bc$ for $A_\cH$ in $E= \bT_\eta N \times \fk$ induces a quasi-Poisson structure on $\gro{H}{N}$ that integrates the Lie quasi-bialgebroid determined by the quasi-Manin triple $(E, A_\cH, \bc)$. 
\end{coro}

As shown in \cite{ilx:univer}, a Lie groupoid can always be equipped with a (unique) quasi-Poisson structure integrating a Lie quasi-bialgebroid when it is source-simply-connected; we note, however, that this condition is not assumed in the previous corollary.

\begin{exa}[Quasi-symplectic versus quasi-Poisson groupoids]
   When $K=pt$, Cor.~\ref{cor:quapoi} recovers the construction in \cite[Theorem 4.3]{buriglsev}, which shows how quasi-symplectic groupoids can be converted into quasi-Poisson groupoids satisfying the property that their Lie quasi-bialgebroids have Drinfeld doubles which are exact Courant algebroids (and vice-versa). Equivalently, in the terminology of shifted geometry, this result can be interpreted as a correspondence between $1$-shifted symplectic structures and ``nondegenerate'' $1$-shifted Poisson structures on Lie groupoids (up to twists). 
  \hfill $\diamond$
\end{exa}

The next example shows that one can derive from Cor.~\ref{cor:quapoi} the classical fact that Lie quasi-bialgebras integrate to quasi-Poisson Lie groups \cite{kos:quasilb}.

\begin{exa}[Integration of Lie quasi-bialgebras]\label{ex:qpliegroup}
Consider a quasi-Manin triple $(\fd,\g,\fc)$, and let $G$ be Lie group integrating $\g$.
By Corollary~\ref{cor:quapoi},  $G$ can be equipped with a quasi-Poisson bivector integrating the Lie quasi-bialgebra defined by $(\fd,\g,\fc)$ whenever we have a 2-shifted symplectic group $D$ integrating the quadratic Lie algebra $\fd$ along with a morphism $\phi_G: G\to D$ integrating the inclusion $\g\hookrightarrow \fd$ and carrying a lagrangian structure.

We can always equip a Lie group $D$ integrating $\fd$ with the 2-shifted symplectic structure \eqref{eq:polwie} when the pairing of $\fd$ is $\mathrm{Ad}_D$-invariant, e.g. when $D$ is connected. Then any $G$ admitting a morphism $\phi_G: G \to D$ integrating the inclusion
inherits a  quasi-Poisson structure (since any such morphism is lagrangian, see Example~\ref{ex:lagsub}); this is the case when $G$ is 1-connected.
    \hfill $\diamond$
\end{exa}

\begin{exa}[Dynamical Poisson groupoids]
Consider a Manin pair $(\fd, \g)$, let $D$ be a Lie group integrating $\fd$ equipped with the 2-shifted symplectic structure \eqref{eq:polwie} and $G\subseteq D$ a Lie subgroup integrating $\g$. Let $N$ be a manifold, consider the Lie groupoid 
$$
\cH = (N\times N) \times G \rightrightarrows N
$$ 
(the direct product of the pair groupoid with the Lie group $G$), and the morphism $\Phi: \cH \to D$ given by $\Phi((x,y), g)=g$. Then $\sigma=0$ and $\eta=0$ define a lagrangian structure on $\Phi$, with corresponding Dirac structure
$$
A_\cH = TN\times \g \subseteq \bT N \times \fd.
$$

When $\fd=\g\oplus \g^*$ is the Drinfeld double of a Lie quasi-bialgebra, any complement of $A_\cH$ in  $\bT N \times \fd$ is of the form $\mathbf{C}=\gra(\Lambda^\sharp)$, for a skew-symmetric map $\Lambda^\sharp: T^*N\times \g^*\to TN\times \g$, and induces, by Cor.~\ref{cor:quapoi},  a quasi-Poisson structure on $(N\times N)\times G$. In the special case where the Lie quasi-bialgebra has trivial bracket on $\g^*$, any complement $\mathbf{C} \subseteq \bT N \times \fd$ that is a Dirac structure is given by a generalized $r$-matrix coupled with a Poisson structure on $N$ as in Example~\ref{exa:r-matrix}, and
makes $(N\times N)\times G$ into a Poisson groupoid (generalizing the {\em dynamical Poisson groupoids} of \cite{eti:rmat}, see \cite[Ex.~4.6]{xu:local}).
\hfill $\diamond$
\end{exa}

\section{Integration of quasi-Poisson manifolds}\label{sec:qpoisson}

This section studies a special class of 2-shifted lagrangian structures corresponding to quasi-Poisson manifolds, described in Example \ref{ex:qpoi}. 
We start by revisiting quasi-Poisson manifolds from a more intrinsic perspective \cite{quadir,buriglsev} (see also \cite{manpairev}).

\begin{defi}\label{def:qpoiL}
Let $(\fd,\g)$ be a Manin pair. A \emph{quasi-Poisson $\g$-manifold} (with respect to $(\fd,\g)$) is a manifold $N$ together with a
Dirac structure 
$\mathbf{L}\subseteq \bT N \times \fd$ 
with the additional property that the projection map on the second factor restricts to a fibrewise isomorphism 
\begin{equation}\label{eq:qpoismanin}
\mathbf{L}\cap (T N \times \fd) \stackrel{\sim}{\longrightarrow} \g.
\end{equation}
In other words, $\mathbf{L}$ is  
a morphism of Manin pairs 
$
(\bT N, TN)\dashrightarrow (\overline{\fd},\g),
$
see $\S$ \ref{subsec:cmorph}.
 \hfill $\diamond$
\end{defi}

 As recalled in Example~\ref{ex:quasipact}, $\mathbf{L}$ induces a $\g$-action on $N$, $\rho: \g_N\to TN$, defined by the condition
$$
((\rho(u), 0), u) \in \mathbf{L}, \qquad u\in \g.
$$

The connection with the original notion of quasi-Poisson manifold (recalled in Example \ref{ex:qpoi}) relies on the choice of an isotropic complement of $\g$ in $\fd$. Indeed, when $\fd=\g\oplus \g^*$ is the double of a Lie quasi-bialgebra,
condition \eqref{eq:qpoismanin} is equivalent to 
$$
\mathbf{L}=\gra(\Lambda^\sharp)
$$ 
for a map $\Lambda^\sharp:T^*N\oplus\g_N\to TN\oplus \g^*_N$ of the form
$$
\Lambda^\sharp:T^*N\oplus\g_N\to TN\oplus \g^*_N,\quad \Lambda^\sharp(\alpha,u)=\big(\pi^\sharp(\alpha)+\rho(u), -\rho^*(\alpha)\big),
$$
where $\pi$ is a bivector field on $N$ and $\rho: \g_N\to TN$; in this case the involutivity of $\mathbf{L}$ in $\bT N \times \fd$ is equivalent to $\rho$ being a $\g$-action on $N$ with respect to which $(N,\pi)$ is a quasi-Poisson $\g$-manifold in the sense of Example \ref{ex:qpoi} (see \cite{quadir}).

Since quasi-Poisson manifolds are special types of infinitesimal lagrangian morphisms, we now identify the global 2-shifted lagrangian morphisms that arise as their integrations.

\subsection{Multiplicative $D$-valued moment maps}\label{sec:Dvalmm}
Let $(\fd,\g)$ be a Manin pair, and let $(D, \Omega+\Theta)$ be a $2$-shifted symplectic group  integrating $(\fd, \langle\cdot,\cdot\rangle)$.

Recall from Example \ref{ex:carcouca} that the exact CA-groupoid $\CAg{D}{\Omega}{\Theta}$ is isomorphic to the action CA-groupoid $(\fd \oplus \overline{\fd})_D \rightrightarrows \fd$ via $\mathbf{e}_\Omega: (\fd \oplus \overline{\fd})_D {\to} \CAg{D}{\Omega}{\Theta}$ (see \eqref{eq:splitting exact CA}),
$$
  \mathbf{e}_\Omega(u+v,d)=\left( (u^r-v^l)|_d, \Omega|_{(1,d)}((0,\cdot),(u,0))-\Omega|_{(d,1)}((\cdot,0),(0,v))\right). 
$$

Through this isomorphism, the lagrangian subalgebra $\g\subseteq \fd$ induces a multiplicative Dirac structure on $D$,
$$
{\bf A}_\g:={\bf e}_\Omega((\g \oplus \g)_D)\subset \CAg{D}{\Omega}{\Theta}.
$$

\begin{defi}\label{def:multDvalued}
A {\em multiplicative $D$-valued moment map} is a Lie groupoid morphism $\Phi: (\gro{H}{N})\to D$ together with $\sigma \in \Omega^2(\cH)$ such that
\begin{itemize}
\item[(a1)] $(\cH, \Phi, \sigma +0)$ is a lagrangian morphism to $(D,\Omega+\Theta)$, 
\item[(a2)] $\Phi: (\cH, \sigma) \to (D, {\bf A}_\g)$ is a strong Dirac morphism.\hfill $\diamond$
\end{itemize}

\end{defi}

While (a1) says that $\Phi$ is equipped with a 2-shifted lagrangian structure, 
condition (a2) says  that $\Phi$ is also a 1-shifted infinitesimal lagrangian morphism to $(D, {\bf A}_\g)$ (see Example~\ref{ex:strongdiracmap}).  The terminology ``multiplicative $D$-valued moment map'' is inspired by \cite{manpairev} and reflects the viewpoint that Dirac maps should be regarded as moment maps \cite{burcra1}.

To make the previous definition explicit, note that (a1) and (a2) amount to the following conditions:
\begin{enumerate}
\item[(b1)] $\gu^*\sigma=0, \quad \delta\sigma= \Phi^*\Omega, \quad d\sigma+\Phi^*\Theta=0$,
\item[(b2)] $\dim(\cH)-2 \dim(N)=\frac{1}{2}\dim(D),$ 
\item[(b3)] $\ker(\sigma)\cap \ker(T\Phi)=0,$
\item[(b4)] for any $(u,v)\in \g\times \g$, there is a vector field $(u,v)_\cH$ on $\cH$ that is $\Phi$-related to the vector field $u^r-v^l$ on $D$ and satisfies
$$
i_{(u,v)_\cH}\sigma = \Phi^* (\zeta(u,v)), 
$$
where $\zeta:\g \times \g \to \Omega^1(D)$ is given by
\begin{equation}\label{eq:zeta}
\zeta(u,v) |_d=\Omega|_{(1,d)}((0,\cdot),(u,0))-\Omega|_{(d,1)}((\cdot,0),(0,v)).
\end{equation}
\end{enumerate}

Moreover,  $(u,v)_\cH$ is uniquely determined by $(u,v)$ (due to the condition $\ker(\sigma)\cap \ker(T\Phi)=0$), and the assignment $(u,v)\mapsto (u,v)_\cH$ defines a $\g\times\g$-action on $\cH$.

Our main result in this section, Theorem~\ref{thm:intqpoi} below, shows that, within the infinitesimal--global correspondence 
\smallskip
$$
\{ \mbox{Dirac structures in $\bT_\eta N \times \fd$}\} \quad \rightleftharpoons \quad \{\mbox{Lagrangian morphisms $(\cH,\Phi,\sigma+\eta)$ in $D$}\}  
$$
in Cor.~\ref{cor:int2lag},
multiplicative $D$-valued moment maps provide the global counterparts of quasi-Poisson manifolds; see Fig.~\ref{fig:table of integration}.

\begin{thm}\label{thm:intqpoi} 
Let $(\fd,\g)$ be a Manin pair. Let $(D, \Omega+\Theta)$ be a $2$-shifted symplectic group  integrating $(\fd, \langle\cdot,\cdot\rangle)$, and  consider a lagrangian morphism
$(\cH, \Phi, \sigma +0)$ into $D$. Then its corresponding Dirac structure ${\bf L} \subseteq \mathbb{T}N\times {\fd}$ satisfies \eqref{eq:qpoismanin} if and only if
\begin{equation}
\Phi: (\cH,\sigma)\to (D ,\bf A_\g)\label{eq:gloqpoi}
\end{equation}
is a strong Dirac morphism (i.e., $\Phi$ is a multiplicative $D$-valued moment map). 
\end{thm}

\begin{proof}
Recall that $\mathbf{L}$ is the image of the injective map $(\mathtt{a},\lambda_\sigma, \varphi): A_\cH \to \bT N \times \fd$, where $\varphi=\mathrm{Lie}(\Phi)$.

We will use the following consequences of the isotropic condition $\delta\sigma=\Phi^*\Omega$: for $a\in \Gamma(A_\cH)$, $b\in \Gamma(\ker(T\gt)|_N)$,
and $\zeta$ defined in \eqref{eq:zeta}, we have
\begin{align}\label{eq:isotropicinf}
i_{a^r}\sigma  &= \gt^*(\lambda_\sigma(a)) + \Phi^*(\zeta(\varphi(a),0))\\ 
\label{eq:isotropicinf2}
 i_{b^l}\sigma &= \gs^*(\gu^* i_b\sigma) - \Phi^*(\zeta(0,T\Phi(b)));
\end{align}
the equality in \eqref{eq:isotropicinf} is Lemma~\ref{lem:infmult} (a), while \eqref{eq:isotropicinf2} is its analog for left-invariant vector fields and can be proven similarly.

Recall that $\Phi$ being a strong Dirac morphism means that it satisfies properties (b3) and (b4) above.

\noindent {\underline{\em Proof that if $\Phi$ is strong Dirac, then \eqref{eq:qpoismanin} holds}}. 

To verify that \eqref{eq:qpoismanin} holds, we must check that 
\begin{itemize}
\item[(a)] $\mathbf{L}\cap (TN\times 0)=0$, 
\item[(b)] $\mathbf{L}\cap(TN\times \g)\to \g$ is onto. 
\end{itemize}
The condition in (a) is equivalent to $\ker(\varphi)\cap \ker(\lambda_\sigma)=0$, and to verify that it holds, note that if $a\in \ker(\varphi)\cap \ker(\lambda_\sigma)$ then \eqref{eq:isotropicinf} implies that $a\in  \ker(T\Phi)\cap \ker(\sigma)$, so $a=0$. 

To check (b), take $u\in \g$. Then $(u,0)_\cH$ is a vector field on $\cH$ satisfying 
$$
i_{(u,0)_\cH}\sigma = \Phi^*(\zeta(u,0)), \quad \mbox{and} \quad T\Phi((u,0)_\cH) =u^r.
$$
Consider $x\in N$ and $a = (u,0)_\cH |_x$.
We claim that $T\gs(a)=0$. Indeed, for any  $b\in \ker(T\gt)|_x$,
by the first identity above and \eqref{eq:isotropicinf2},  we have
$$
-i_{b}\Phi^*(\zeta(u,0)) =  i_{a}i_{b}\sigma = \sigma(b,T\gs(a)) - i_{u}\zeta(0,T\Phi(b)).
$$
Note that $-i_{b}\Phi^*(\zeta(u,0)) = \Omega|_{(1,1)}((u,0),(0,T\Phi(b))) = -i_{u}\zeta(0,T\Phi(b))$, hence $(i_{T\gs(a)}\sigma) |_{\ker(T\gt)}=0$. Since
$(i_{T\gs(a)}\sigma) |_{TN}=0$ (recall that $\gu^*\sigma=0$), it follows that 
$T\gs(a)\in \ker(\sigma)$. Since $T\Phi(T\gs(a))=0$,  $T\gs(a)\in \ker(\sigma)\cap \ker(T\Phi)=0$, proving the claim.  Knowing that $a\in A_\cH$, it follows from \eqref{eq:isotropicinf} that 
$$
\Phi^*(\zeta(u,0))=i_a\sigma = \gt^*(\lambda_\sigma(a)) + \Phi^*(\zeta(u,0)),
$$
so $a\in \ker(\lambda_\sigma)$. Therefore $(\mathtt{a}(a),0,u)\in \mathbf{L}$, which proves (b) above.

\noindent {\underline{\em Proof that if \eqref{eq:qpoismanin} holds, then $\Phi$ is a strong Dirac map}}. 

We start by checking that $\ker(\sigma)\cap \ker(T\Phi)=0$. Taking $X\in \ker(\sigma)\cap \ker(T\Phi) |_h$, 
it follows from \eqref{eq:isotropicinf} that 
$$
0= i_Xi_{a^r}\sigma = i_{T\gt(X)}\lambda_\sigma(a)
$$
for all $a\in  A_\cH|_{\gt(h)}$. Since $\mathbf{L}$ projects surjectively onto $T^*N$ (to see that, note that the kernel of this projection is $ \mathbf{L}\cap (TN\times \fd)$, whose rank is  $\dim(\g)=\frac{1}{2}\dim(\fd)$ by \eqref{eq:qpoismanin}, while the rank of $\mathbf{L}$ is $\dim(N) + \frac{1}{2}\dim(\fd)$), the map  $\lambda_\sigma$ has full rank, and therefore $T\gt(X)=0$. As a consequence, 
$$
X = b^l |_h,
$$ 
for $b\in \ker(T\gt)\cap \ker(T\Phi)|_{\gs(h)}$ (note that $T\Phi(b)=0$ since $0=T\Phi(b^l |_h) = T\Phi(b)^l |_{\Phi(h)}$). 
It follows from \eqref{eq:isotropicinf2} that $\gs^*(\gu^*i_b\sigma)=0$; applying the inversion map, we see from Lemma~\ref{lem:infmult} (b) that  this holds if and only if $\gt^*\lambda_\sigma(a)=0$, where $a=\gi_*(b)\in \ker(T\gs)\cap \ker(T\Phi) |_{\gt(h)}$. We conclude that $a\in \ker(\lambda_\sigma)\cap \ker(\varphi) = 0$, so $b=0$ and hence $X=0$, proving that $\ker(\sigma)\cap \ker(T\Phi)=0$. 

It remains to show that given any $(u,v)\in \g\oplus\g$ and $h\in \cH$, there exists $Z\in T\cH |_h$ such that 
\begin{equation}\label{eq:diraccond}
i_Z\sigma = \Phi^*(\zeta(u,v))|_h, \quad \;  \mbox{ and }\quad \;  T\Phi(Z)= (u^r-v^l) |_{\Phi(h)}.
\end{equation}
Since $\mathbf{L}$ satisfies property (b) above, given $u\in \g$ there exists $a\in A_\cH |_y$ such that $\lambda_\sigma(a)=0$ and $\varphi(a)=T\Phi(a)=u$. By \eqref{eq:isotropicinf},
\begin{equation}\label{eq:inneract1}
i_{a^r}\sigma = \Phi^*(\zeta(u,0)).
\end{equation}
Note that
$$
\gu^*(i_{\gi_*a}\sigma) = \gu^*\gi^*(i_{\gi_*a}\sigma)  = \gu^* i_a(\gi^*\sigma)
=- \lambda_\sigma(a)=0,
$$
where the third equality  follows from Lemma~\ref{lem:infmult} (b). By \eqref{eq:isotropicinf2},
\begin{equation}\label{eq:inneract2}
i_{(\gi_*a)^l}\sigma = -\Phi^*(\zeta(0,T\Phi(\gi_*a))) = \Phi^*(\zeta(0,u)).
\end{equation}
Now given $(u,v)\in \g\oplus \g$ and $h\in \cH$, take $a\in A_\cH|_{\gt(h)}$, $b\in A_\cH|_{\gs(h)}$ with $\lambda_\sigma(a)=\lambda_\sigma(b)=0$, and $T\Phi(a)=u$, $T\Phi(b)=v$. Then $Z=(a^r+ (\gi_*b)^l)|_h$ satisfies \eqref{eq:diraccond}.
\end{proof}

Note that Theorem~\ref{thm:intqpoi} acquires its simplest form when $D$ is a point: in this case, quasi-Poisson manifolds are usual Poisson manifolds, and multiplicative $D$-valued moment maps are symplectic groupoids, so one recovers the classical integration of Poisson manifolds \cite{macxu2}; see Fig.~\ref{fig:table of integration}.

\begin{rema}[Isotropic complements and quasi-Poisson structures]\label{rem:isocompl} Given a Manin pair $(\fd,\g)$ and a  multiplicative $D$-valued moment map $\Phi: \cH\to D$ with corresponding Dirac structure $\bl\subset \bT N\times\fd$, we saw that there is an induced $\g$-action on $N$ and
the choice of an isotropic complement $\fc$ of $\g$ in $\fd$ induces a bivector field on $N$ making it into a quasi-Poisson manifold in the sense of Example~\ref{ex:qpoi}; note that $\fc$ also induces an isotropic complement $\bc=TN\oplus\fc$ of $\bl$ in $ \bT N\times\fd$ as well as an isotropic complement ${\bf e}_\Omega((\fc\oplus\fc)_D)$ of $\mathbf{A}_\g$ in $\CAg{D}{\Omega}{\Theta}$, and those give rise to quasi-Poisson bivector fields on $\cH$ and $D$ (by Corollary \ref{cor:quapoi}). Therefore, once $\fc$ is fixed, one obtains an alternative formulation of multiplicative $D$-valued moment maps in terms of quasi-Poisson structures that we plan to explore in a separate paper.\hfill $\diamond$
\end{rema}

The appearance of a ``moment map'' in the integration of a quasi-Poisson manifold vastly extends the classical fact that any action on a manifold $M$ lifts canonically to a hamiltonian action on its cotangent bundle $T^*M$. This principle holds, more generally, for Poisson actions \cite{mommap}, and its connection with Theorem~\ref{thm:intqpoi} will be explained in Prop.~\ref{prop:P}.

\subsection{Global properties}\label{subsec:equiv}
We will consider multiplicative $D$-valued moment maps satisfying additional global properties. For their description, we need to introduce the following global counterpart of a given Manin pair $(\fd,\g)$. Suppose that $D$ is a Lie group integrating the Lie algebra $\fd$, $G$ is a Lie group integrating $\g$, and 
$$
\phi_G: G\to D, \qquad g\mapsto \bar{g}
$$ 
is a morphism integrating the inclusion $\g\hookrightarrow \fd$. Consider the $G$-action on $\fd$ by $g\mapsto \mathrm{Ad}_{\bar{g}}$.

\begin{defi}\label{def:groupair}
We say that $(D,\phi_G)$ is a {\em $G$-equivariant group pair} for the Manin pair $(\fd,\g)$ if the $G$-action on $\fd$
\begin{itemize}
\item preserves the pairing on $\fd$, and
\item extends the adjoint action of $G$ on $\g$: $\mathrm{Ad}_{\bar{g}}(u)= \mathrm{Ad}_g(u)$, for $u\in \g$.
\end{itemize}
Due to the latter property, we will simplify the notation by writing $\mathrm{Ad}_g$ in place of $\mathrm{Ad}_{\bar{g}}$.
\hfill $\diamond$
\end{defi}

For example,  $G$ could be a connected subgroup of $D$, or any subgroup of $D$ in case $D$ is connected.

Fix a Manin pair $(\fd,\g)$, consider any $2$-shifted symplectic group  $(D, \Omega+\Theta)$  integrating $(\fd, \langle\cdot,\cdot\rangle)$, and
let $\Phi: (\cH,\sigma) \to D$ be a multiplicative $D$-valued moment map. We collect some of its properties.

By \eqref{eq:qpoismanin}, the map $\varphi=\mathrm{Lie}(\Phi)$ restricts to an isomorphism $\ker(\lambda_\sigma)\stackrel{\sim}{\to} \g_N$. Let 
$$
\psi: \g_N \to \ker(\lambda_\sigma)\subseteq A_\cH
$$
be the inverse map. Then:
\begin{itemize}
\item  The  $\g$-action $\rho$ on $N$ coincides with $\mathtt{a}\circ \psi: \g_N\to TN$, and $\psi$ defines a morphism of Lie algebroids
\begin{equation}\label{eq:psiact}
\psi: \g\ltimes N \to A_\cH,
\end{equation}
where $\g\ltimes N$ is the action Lie algebroid.

\item The $\g\times \g$-action on $\cH$ (see (b4) after Def.~\ref{def:multDvalued}) is given by 
\begin{equation}\label{eq:innerLAact}
(u,v)\mapsto \psi(u)^r+ (\gi_*(\psi(v)))^l,
\end{equation}
as a consequence of \eqref{eq:inneract1} and \eqref{eq:inneract2}.

\item The 2-form $\sigma$ is not invariant in general, but it follows from \eqref{eq:isotropicinf} that
\begin{equation}\label{eq:weakinv}
\gu^* (\pounds_{\psi(u)^r}\sigma)=0.
\end{equation}
\end{itemize}

It turns out that, in many examples, multiplicative $D$-valued moment maps satisfy global versions of these properties, which we now formulate.

\begin{defi}\label{def:GGequiv}
Let $\Phi: (\cH,\sigma) \to D$ be a multiplicative $D$-valued moment map, and suppose that $(D,\phi_G)$ is a $G$-equivariant group pair  for $(\fd,\g)$.
We say that $\Phi: (\cH,\sigma) \to D$ is {\em global} (with respect to $(D,\phi_G)$) if the following is satisfied:
\begin{itemize}
\item[(e1)] the $\g$-action on $N$ integrates to a $G$-action;
\item[(e2)] the Lie-algebroid morphism $\psi$ in \eqref{eq:psiact} integrates to a groupoid morphism $\Psi: G\ltimes N \to \cH$ such that $\Phi\circ \Psi$ is  the composition of the projection $G\times N \to G$ with $\phi_G: G \to D$; 
\item[(e3)] the map $\Phi$ is $G\times G$-equivariant with respect to the $G\times G$-action on $\cH$ given by
\begin{equation}\label{eq:actGG}
(g_1,g_2)h = \Psi(g_1,\gt(h))h\Psi(g_2,\gs(h))^{-1}
\end{equation}
(that integrates the $\g\oplus\g$-action \eqref{eq:innerLAact}) and the $G\times G$-action on $D$ by $(g_1,g_2) d = \bar{g}_1 d \bar{g}_2^{-1}$;
\item[(e4)] for each $g\in G$, the map $\Psi_g:=\Psi(g,\cdot): N \to \cH$ satisfies 
$$
\Psi_g^*\sigma=0.
$$ 
This condition is the global counterpart of \eqref{eq:weakinv}.\hfill $\diamond$
\end{itemize}

\end{defi}

The next remark elaborates on condition (e4).

\begin{rema}[Isotropic bisections]\label{rem:isobis}
Recall that a {\em bisection} of a Lie groupoid $\cH\rightrightarrows N$ is a map $\gb: N\to \cH$ such that $\gs\circ \gb=\id$ and $\gt\circ \gb$ is a diffeomorphism of $N$, see \cite[\S 1.4]{macgen}. We say that a bisection $\gb$ is {\em isotropic} with respect to a 2-form $\sigma\in \Omega^2(\cH)$ if $\gb^*\sigma =0$.
Condition (e4) says that each map $\Psi_g$ is an isotropic bisection of $(\cH,\sigma)$ (note that $\gs(\Psi_g(x))=x$ and $\gt(\Psi_g(x))=g x$). 

On a Lie groupoid $\cH\rightrightarrows N$, any bisection $\gb:N\to \cH$ gives rise to a left-translation 
$$
\gl_{\gb}:\cH\to \cH, \quad\;\;  h\mapsto \gb(\gt(h))h=\gm\circ (\gb\circ \gt,\id)(h), 
$$
and a right-translation $\gr_{\gb}$ defined analogously. If $\sigma\in \Omega^2(\cH)$ and $\gb$ is isotropic, then 
\begin{equation}\label{eq:invar}
\gl_{\gb}^*\sigma=\sigma - (\gb\circ \gt,\id)^*\delta\sigma, \quad\; 
\gr_{\gb^{-1}}^*\sigma = \sigma - (\id, \gi\circ \gb\circ \gs)^*\delta \sigma.
\end{equation}
In particular, $\sigma$ is invariant by left and right translations whenever it is multiplicative. 
\hfill $\diamond$
\end{rema}

Let $\Phi: (\cH,\sigma) \to D$ be a global multiplicative $D$-valued moment map. The next result describes how $\sigma$ fails to be invariant with respect to the $G\times G$-action \eqref{eq:actGG} on $\cH$.

\begin{lem}\label{prop:propertiesGG}
For each $g\in G$, we have
$$
(g,1)^*\sigma=\sigma - \Phi^* s_{0,g}^*\Omega, \quad\;\; 
(1,g)^*\sigma = \sigma - \Phi^*s_{1,g^{-1}}^*\Omega,
$$
where $s_{0,g}, s_{1,g}: D\to D\times D$ are given by $s_{0,g}(d)=(\bar{g},d)$, $s_{1,g}(d)=(d,\bar{g})$. 
\end{lem}

\begin{proof}
In the notation of Remark~\ref{rem:isobis}, we have
$$
(g,1)h =  \gl_{\Psi_{g}}(h), \qquad (1,g)h =  \gr_{\Psi_{g}^{-1}}(h).
$$
Since $\delta \sigma=\Phi^*\Omega$ and, for each $g\in G$, $\Phi\circ \Psi_g=\bar{g}$, the result follows from \eqref{eq:invar}.
\end{proof}

Considering the tangent and cotangent lifts of the $G$-action on $N$ to $\bT N$ and the adjoint action of $G$ on $\fd$, we have a natural $G$-action on $\bT N \times \fd$ by Courant automorphisms. We will use the simplified notation $X\mapsto gX$  for the
tangent $G$-action on $TN$, and similarly
$\alpha\mapsto g\alpha$ for the cotangent action on $T^*N$.

\begin{prop}\label{prop:Linv}
If a multiplicative $D$-valued moment map $\Phi: (\cH,\sigma) \to D$ is global, then its corresponding Dirac structure $\mathbf{L}\subseteq \bT N \times \fd$ is $G$-invariant.
\end{prop}

\begin{proof} 
Consider the diagonal $G$-action on $\cH$, 
$$
(g,g)h = \Psi(g,\gt(h))h \Psi(g,\gs(h))^{-1} = \gr_{\Psi_{g}^{-1}}(\gl_{\Psi_{g}}(h)),
$$
in the notation of Remark~\ref{rem:isobis}.
This action is by Lie-groupoid automorphisms, so it differentiates to a $G$-action on $A_\cH$, covering the $G$-action on $N$. Since $\mathbf{L}$ is the image of the map $(\mathtt{a}, \lambda_\sigma, \varphi): A_\cH\to \bT N \times \fd$, to prove the proposition it suffices to verify that this map is $G$-equivariant. 
The $G$-equivariance of $\mathtt{a}: A_\cH\to TN$ follows from the fact that the $G$-action on $A_\cH$ is by Lie-algebroid automorphisms. The $G$-equivariance of $\varphi=\mathrm{Lie}(\Phi): A_\cH\to \fd$ follows from the $G$-equivariance of $\Phi$. Regarding  $\lambda_\sigma: A_\cH\to T^*N$, first note that, for $a\in A_\cH|_y$, $X\in TN|_y$,
$$
((g,g)^*\sigma)(a,X)= \sigma(a,X),
$$
as a consequence of Lemma~\ref{prop:propertiesGG} (using that $\Phi\circ\gl_{\Psi_g} = \gl_{\bar{g}} \circ \Phi$ and $T\Phi(X)=0$). It follows that, writing $Y=g X \in TN|_{gy}$,
$$
i_Y\lambda_\sigma(ga) = \sigma(g a, Y) = \sigma(a,g^{-1}Y)= i_Y (g \lambda_\sigma(a)),
$$
showing the $G$-equivariance of $\lambda_\sigma$.
\end{proof}

\subsection{Reduction}\label{subsec:poisquot} 
We now describe a ``reduction'' procedure for multiplicative $D$-valued moment maps with respect to lagrangian subgroups in $D$. As usual in classical reduction, this is a two-step construction: Prop.~\ref{prop:integL} is a ``restriction'' step, while Thm.~\ref{thm:reducSG} is a subsequent quotient operation. As a result of this procedure, we obtain integrations of Dirac and  Poisson structures associated with quasi-Poisson manifolds.

Let $(\fd,\g)$ be a Manin pair and $N$ be a quasi-Poisson $\g$-manifold defined by $\mathbf{L} \subseteq \bT N \times \fd$, as in Def.~\ref{def:qpoiL}.
We will consider the following simplified set-up:
\begin{itemize}
\item A 2-shifted symplectic group $(D,\Omega + \Theta)$ integrating $(\fd,\langle\cdot,\cdot\rangle)$, and a   multiplicative $D$-valued moment map $\Phi: (\cH,\sigma) \to D$ integrating $\mathbf{L}$;

\item A lagrangian morphism $L\to D$, with respect to the zero 2-form on $L$, integrating a lagrangian subalgebra $\fl \subseteq \fd$. 
\end{itemize}

The next result is a special case of Corollary~\ref{cor:2shiftedfibprod}.
\begin{prop}\label{prop:integL}
Suppose that $\Phi: \cH \to D$ is transverse to $L \to D$. Then
$$
\mathbf{L}_\fl := \{(X,\alpha)\,|\, \exists \, l\in \fl\, \mbox{ with } ((X,\alpha), l) \in \mathbf{L} \} \subseteq \bT N 
$$
is a Dirac structure, and the Lie groupoid $\cG_L:= \cH\times_D L\rightrightarrows N$ equipped with $\sigma_L:=\pr_1^*\sigma\in \Omega^2(\cG_L)$ is a quasi-symplectic groupoid integrating $\mathbf{L}_\fl$.
\end{prop}
When the morphism $L\to D$ is injective,  we can view $L\subseteq D$, and 
$$
\cG_L=\Phi^{-1}(L) \stackrel{\iota}{\hookrightarrow} \cH, \quad \sigma_L = \iota^*\sigma.
$$

Now consider the restriction of the $\g$-action $\rho$ on $N$ to the Lie subalgebra $\h := \g\cap \fl \subseteq \g$. Then
\begin{equation}\label{eq:psil}
\h\ltimes N \to \mathbf{L}_\fl, \qquad (u,x)\mapsto (\rho(u),0)
\end{equation}
is a Lie-algebroid morphism whose image coincides with the kernel of the Dirac structure $\mathbf{L}_\fl$, 
$$
\ker(\mathbf{L}_\fl) = \mathbf{L}_\fl \cap TN = \rho(\h).
$$ 

We henceforth assume that 
\begin{itemize}
\item the multiplicative $D$-valued moment map $\Phi$ is global, and
\item $H\subseteq G$ is a Lie subgroup integrating $\h$ equipped with a  morphism $H\to L$, $h\mapsto \bar{h}$, integrating the inclusion $\h\hookrightarrow \fl$ such that its composition with $L\to D$ coincides with $\phi_G|_H: H \to D$ and $\mathrm{Ad}_H(\fl)\subseteq \fl$. (These conditions hold e.g. if $H$ is connected.) 
\end{itemize}

We consider the $H$-action on $N$ given by restricting its $G$-action to $H$, and its lift to an $H$-action on $\bT N$. We have a morphism
$$
\Psi_L: H\ltimes N \to \cG_L, \quad\; (h,y)\mapsto (\Psi(h,y),\bar{h}),
$$
that defines an $H\times H$-action on $\cG_L$ by
$$
(h_1,h_2) \gamma = \Psi_L(h_1,\gt(\gamma))\gamma\Psi_L(h_2,\gs(\gamma))^{-1}.
$$

\begin{lem}\label{lem:Hinv}
The Dirac structure $\mathbf{L}_\fl$ is $H$-invariant, and the 2-form $\sigma_L$ on $\cG_L$ is $H\times H$-invariant.
\end{lem}

\begin{proof}
The $H$-invariance of $\mathbf{L}_\fl$ is a direct consequence of the $G$-invariance of $\mathbf{L}$ (Prop.~\ref{prop:Linv}) and the $\Ad_H$-invariance of $\fl$.

For each $h\in H$, $(\Psi_L)_h:= \Psi_L(h, \cdot):N \to \cG_L$ satisfies $(\Psi_L)_h^*\sigma_L=0$. The $H\times H$-invariance of $\sigma_L$ follows from Remark~\ref{rem:isobis}, since each $(\Psi_L)_h$ is an isotropic bisection and $\sigma_L$ is multiplicative.
\end{proof}

Assume that the $H$-action on $N$ is free and proper, so that $N\to N/H$ is a principal $H$-bundle.
The fact that the tangent distribution to the $H$-orbits coincides with the kernel of $L_\fl$ together with the $H$-invariance of $L_\fl$ implies that the quotient $N/H$ inherits a Poisson structure (given by the pushforward of $L_\fl$), see e.g. \cite[Lemma 6.5]{meipoi}. 

In the terminology of \cite[Def.~4.6]{intpoihom}, the quasi-symplectic groupoid $(\cG_L,\sigma_L)$ is an {\em $H$-admissible} integration of the Dirac structure $L_\fl$.
The next result is a consequence of \cite[Prop. 2.12, Cor.~4.9]{intpoihom}.

\begin{thm}[Reduction]\label{thm:reducSG}
 The orbit space ${\cG}:=\cG_L/(H\times H)$ is a Lie groupoid over $N/H$ characterized by the property that the quotient map $p: \cG_L\to \cG/(H\times H)$ is a groupoid morphism and a submersion. There is a unique symplectic structure ${\omega}$ on ${\cG}$ such that $p^*{\omega}=\sigma_L$, and $({\cG},{\omega})$ is a symplectic groupoid integrating the Poisson structure on $N/H$.
 \end{thm}

We note that $(\cG_L, \sigma_L)$ and $({\cG}, {\omega})$ are Morita equivalent quasi-symplectic groupoids \cite{xumor}, see \cite[$\S$ 2.4]{cueca}.

\begin{rema}[Partial reduction]\label{rem:redpar}
The reduction procedure in Theorem \ref{thm:reducSG} admits a more general formulation.
The diagonal morphism $D\to {D}\times \overline{D}$ (see Example \ref{ex:diag}) yields a lagrangian morphism 
$$
D^{n-1}\times L\to ({D}^{n-1}\times \overline{D}^{n-1})\times \overline{D}={D}^{n-1}\times \overline{D}^n,
$$
that we regard as a ``lagrangian correspondence'' from $D^{n-1}$ to $D^n$.
Given a global multiplicative $D^n$-valued moment map  $\Phi:(\cH,\sigma)\to D^n$ that is transverse to the previous map, by Theorem \ref{thm:2lagprod} we obtain  a lagrangian morphism
$$
\cH\times_{D^n} (D^{n-1}\times L) \to D^{n-1}
$$
given by the natural projection. Taking $H\subseteq G$ as above, we obtain by restriction an action of $H\times H=(\{1_{G^{n-1}}\}\times H)^2\subset G^{n}\times G^{n}$ on $\cH\times_{D^n} (D^{n-1}\times L)$ such that the previous map is $H\times H$-invariant. Generalizing Theorem \ref{thm:reducSG}, one can show 
(making use of Lemma~\ref{prop:propertiesGG}) that the induced map 
\[ 
\frac{\cH\times_{D^n} (D^{n-1}\times L)}{(H\times H)} \to D^{n-1} 
\]
is a global multiplicative $D^{n-1}$-valued moment map; see Example \ref{ex:homsp} for an application of this construction. 
\hfill $\diamond$
\end{rema}

In the remainder of this section, we will illustrate Prop.~\ref{prop:integL} and Theorem~\ref{thm:reducSG} in examples.

\subsection{Examples from Poisson actions}

We now consider the special setting of quasi-Poisson manifolds defined by {\em Poisson actions} \cite{luphd,lu:mom}; we will provide multiplicative $D$-valued moment maps for the action of a Poisson Lie group on itself by right translations, and for Poisson actions of complete  Poisson Lie groups.

We will need the following global version of the Drinfeld double of a Lie bialgebra (briefly recalled at the end of $\S$\ref{app:c1}). 
Given a Poisson Lie group $(G, \pi_G)$ with corresponding Manin triple $(\fd=\g\oplus \g^*, \g, \g^*)$, 
a {\em Drinfeld double} of $(G,\pi_G)$ \cite[Def.~6.2]{intpoihom} is a $G$-equivariant group pair $(D,\phi_G)$ for $(\fd,\g)$ (Def.~\ref{def:groupair}) such that the pairing $\langle\cdot,\cdot \rangle$ on $\fd$ is $\Ad_D$-invariant and the $G$-action on $\fd$ is given by
$$
\mathrm{Ad}_{\bar{g}}(u+\xi) = \Ad_g(u) + i_{\Ad^*_{g^{-1}}(\xi)}(\gr_{g^{-1}}(\pi_G|_g)) + \Ad_{g^{-1}}^*\xi.
$$
(The existence of a Drinfeld double of $(G,\pi_G)$ is ensured if $G$ is 1-connected, but may occur without this condition.)

\subsubsection*{\underline{Integration of affine Dirac structures and Poisson homogeneous spaces revisited}}

Let $(G,\pi_G)$ be a Poisson Lie group with Drinfeld double $(D, \phi_G)$.
The action of $G$ on itself by right translations, 
$g\mapsto \gr_{g^{-1}}$, 
is a Poisson action of  $(G, \pi_G)$ on $(G,-\pi_G)$. The corresponding infinitesimal action, 
$$
\rho: \g_G\to TG, \quad\; (u,g)\mapsto -u^l|_g,
$$ 
makes $(G,-\pi_G)$ into a Poisson $\g$-manifold (as in Example~\ref{ex:qpoi}) with Dirac structure
$$
\mathbf{L}= \{((-\pi_G^\sharp(\xi^l)- u^l, \xi^l), (u, \xi)) \, |\, u\in \g, \, \xi \in \g^*\} \subseteq \bT G \times \fd.
$$

We will now present integrations of this Dirac structure by multiplicative $D$-valued moment maps, making contact with the main results in \cite{intpoihom}. To match conventions, we will work with the ``opposite'' Dirac structure
$$
\mathbf{L}_G:= \overline{\mathbf{L}}= \{(-(\pi_G^\sharp(\xi^l)+ u^l, \xi^l), (u, \xi)) \, |\, u\in \g, \, \xi \in \g^*\} \subseteq \bT G \times \overline{\fd},
$$
see Remark~\ref{rem:Lopp}.

By fixing an integration $G^*$ of $\g^*$ together with a morphism $\phi_{G^*}: G^*\to D$, $z\mapsto \bar{z}$, integrating the inclusion $\g^*\hookrightarrow \fd$, we can define the Lie groupoid 
 $\cG_D\rightrightarrows G$,
$$
\mathcal{G}_D :=\{(z, g_1,g_2, d)\in G^*\times G \times G \times D\;|\; \bar{z}\bar{g}_1 = \overline{g}_2 d^{-1} \},
$$
with structure maps are given by
\begin{align}
&\gs(z, g_1,g_2, d) = g_2,\quad \gt(z, g_1,g_2, d)=g_1,\quad  \gu(g) = (1,g,g,1),\label{eq:struc1}\\
&(z,g_1,g_2,d)^{-1} = (z^{-1}, g_2,g_1, d^{-1}),\quad  \gm((z, g_1,g, d), (z', g,g_2, d')) = (z'z, g_1, g_2, dd').\label{eq:struc2}
\end{align}

Consider the 2-form $\omega_D\in \Omega^2(\cG_D)$, 
$$
\omega_D:=\frac{1}{2}\Big(\langle\theta^l_{G,1}, \theta^r_D\rangle+\langle\theta^r_{G^*},\theta^r_{G,2}\rangle\Big),
$$ 
where $\theta^r_{G^*}, \theta^l_{G,1}, \theta^r_{G,2}, \theta^r_D$ are the pullbacks to $\cG_D$ of the left/right invariant Maurer-Cartan forms on each factor of $G^*\times G \times G \times D$. 

Endow $D$ with the 2-shifted symplectic structure $-(\Omega+\Theta)$, with $\Omega$ and $\Theta$ as in \eqref{eq:polwie}.
\begin{prop}\label{prop:actmult}
The map 
$$
\Phi: (\cG_D,\omega_D)\to D, \quad\; (z, g_1,g_2, d) \mapsto d,
$$
is a global multiplicative $D$-valued moment map integrating $\mathbf{L}_G$.
\end{prop}

\begin{proof}
The fact that $\omega_D$ is a lagrangian structure on the morphism $\Phi$ that integrates the Dirac structure $\mathbf{L}_G$ is proven in \cite[$\S$6.2.3]{intpoihom}. 
It follows from Theorem~\ref{thm:intqpoi} that $\Phi: (\cG_D,\omega_D)\to D$ is a multiplicative $D$-valued moment map. 

It remains to verify $G\times G$-equivariance, i.e., conditions (e1)--(e4) in Def.~\ref{def:GGequiv}. 
For the $G$-action on $G$ via $a\mapsto \gr_{g^{-1}}(a)=ag^{-1}$, the map
$$
\Psi: G\ltimes G \to \cG_D, \; \; (g,a)\mapsto (1, ag^{-1}, a, \bar{g})
$$
is a groupoid morphism integrating 
$$
\psi: \g\ltimes G\stackrel{\sim}{\longrightarrow} \mathbf{L}_G\cap{(TG\times \fd)}\subseteq \mathbf{L}_G, \;\; u \mapsto ((-u^l,0), u),
$$
and satisfying $\Phi(\Psi(g,a))=\bar{g}=\phi_G(g)$, so (e1) and (e2) are satisfied. It follows from \eqref{eq:struc2} that $\Phi$ is $G\times G$-equivariant, so (e3) holds. 

For each $g\in G$ and $\Psi_g=\Psi(g,\cdot): G \to \cG_D$, we have that $T\Psi_g(v^r)=(0, v^r, v^r,0)$, which we can use to verify  that $\Psi_g^*\omega_D =0$, so (e4) holds.
\end{proof}

Let us now fix a lagrangian subalgebra $\fl \in \fd$ and an integrating Lie group $L$ equipped with a morphism $\phi_L: L\to D$, $l\mapsto \bar{l}$, integrating the inclusion $\fl\hookrightarrow \fd$. Since $\Phi: \cG_D\to D$ is a surjective submersion, we have the following immediate consequence of Prop.~\ref{prop:integL}: 
\begin{coro}[Theorem~6.6 in \cite{intpoihom}]\label{cor:GL}
The Lie groupoid 
$$
\cG_L:=\cG_D\times_D L = \{(z, g_1,g_2, l)\in G^*\times G \times G \times L\;|\; \bar{z}\bar{g}_1 = \overline{g}_2 \bar{l}^{-1} \},
$$ 
equipped with the 2-form 
$$
\omega_L:=\pr_1^*\omega_D = \frac{1}{2}\Big(\langle\theta^l_{G,1}, \theta^r_L\rangle+\langle\theta^r_{G^*},\theta^r_{G,2}\rangle\Big)
$$  
is a quasi-symplectic groupoid integrating the Dirac structure
$$
(\mathbf{L}_G)_\fl := \{(\pi_G^\sharp(\xi^l)+ u^l, \xi^l)\, |\, (u, \xi) \in \fl\}\subseteq \bT G.
$$
\end{coro}

Dirac structures of the form $(\mathbf{L}_G)_\fl$ are called {\em $(G,\pi_G)$-affine Dirac structures} (\cite[Def.~5.5]{intpoihom}) since they extend affine Poisson structures in the sense of  \cite{luphd}, see \cite[Lem.~5.3]{intpoihom}. 

\begin{exa}[Lu-Weinstein double symplectic groupoid]\label{ex:LuWeins}
For the particular choice $\mathfrak{l}=\g^*$, the Lie group $L=G^*$ acquires the structure of a Poisson Lie group (see \cite[Remark~6.4]{intpoihom}).
The corresponding affine Dirac structure coincides with $\pi_G$, so
$(\cG_{G^*},\omega_{G^*})$ is an integration of $(G,\pi_G)$. 
In this special case, $(\cG_{G^*},\omega_{G^*})$ carries an additional groupoid structure over $G^*$, with source, target and multiplication maps given by
$$
(z_1,g_1,g_2,z_2)\mapsto z_2^{-1}, \;\; (z_1,g_1,g_2,z_2)\mapsto z_1, \;\; (z_1,g_1,g_2,z_2)(z_2^{-1},g'_1,g'_2,z'_2) = (z_1,g_1g_1',g_2g_2',z'_2),
$$
making it into a double symplectic groupoid  \cite{luwei2}. 
\hfill $\diamond$
\end{exa}

For a closed Lie subgroup $H\subseteq G$, we say that a Poisson structure $\pi$ on $G/H$ is {\em homogeneous} \cite{dripoi} if the action map $(G,\pi_G)\times (G/H,\pi)\to (G/H,\pi)$, $(g, g'H)\mapsto gg'H$, is a Poisson map.
As explained in \cite[$\S$ 5.2]{intpoihom}, Poisson homogeneous spaces are closely related to affine Dirac structures: indeed, the (Dirac-geometric) pullback of a homogeneous Poisson structure on $G/H$ under the quotient map $G\to G/H$ is the affine Dirac structure $(\mathbf{L}_G)_\fl$ corresponding to the lagrangian subalgebra
$$
\fl = \{ u+\xi\,|\, u\in \g,\, \xi \in \mathrm{Ann}(\h), i_\xi(\pi|_{p(1)})= u+\h\} \subseteq \fd.
$$
Let $L$ be a Lie group integrating $\fl$ equipped with a morphism $\phi_L: L\to D$ integrating the inclusion $\fl\hookrightarrow \fd$, and consider the quasi-symplectic groupoid $(\cG_L,\omega_L)$ integrating  $(\mathbf{L}_G)_\fl$ as in Cor.~\ref{cor:GL}. 

 Note that $\fl\cap \g = \h $. Suppose that we have a morphism $H\to L$ integrating $\h\hookrightarrow \fl$ so that its composition with $\phi_L$ is the map $\phi_G|_H: H\to D$. In this case $\pi$ is the pushfoward of $(\mathbf{L}_G)_\fl$. We are now in the setup of Theorem~\ref{thm:reducSG}: $\cG_L$ carries an $H\times H$-action such that $\cG:=\cG_L/(H\times H)$ is a Lie groupoid equipped with a multiplicative symplectic form $\omega$ determined by the property that $p^*\omega = \omega_L$, where $p: \cG_L\to \cG$ is the quotient map. 

\begin{coro}[Theorem~6.18 in \cite{intpoihom}]
The symplectic groupoid $(\cG,\omega)$ integrates the Poisson homogeneous space $(G/H,\pi)$.
\end{coro}

These integrations of Poisson homogeneous spaces are  illustrated in many examples in \cite{intpoihom}.

\subsubsection*{\underline{Actions of complete Poisson groups}}\label{subsec:lagpoi} 
Let $(G,\pi_G)$ be a Poisson Lie group with Drinfeld double $(D,\phi_G)$, and suppose that we have a Lie group $G^*$ integrating $\g^*$ and a morphism $\phi_{G^*}: G^*\to D$ integrating the inclusion $\g^*\hookrightarrow \fd$; in such a case $G^*$ carries a Poisson Lie group structure integrating the Lie bialgebra structure on $\g^*$ (\cite[Rem.~6.4]{intpoihom}).
We will make the further special assumption that
\begin{equation}\label{eq:complete}
G^*\times G\to D, \;\;\; (z,g)\mapsto \phi_{G^*}(z)\phi_G(g)
\end{equation}
is a diffeomorphism. The existence of a Drinfeld double with all these properties is guaranteed if both $G$ and $G^*$ are {\em complete} simply-connected Poisson Lie groups, see \cite[$\S$ 2.5]{luphd}. For simplicity, we identify $G$ and $G^*$ with their images in $D$.

In this setup we obtain {\em dressing actions} of $G^*$ on $G$ and of $G$ on $G^*$, 
$$
G\times G^*\to G,\;\; (g,z)\mapsto g^z,\qquad G\times G^*\to G^*,\;\; (z,g)\mapsto {}^gz,
$$
determined by the condition
\[
gz={}^gz g^z, \qquad \;\; \forall\, (g,z)\in G \times G^*.
\]
Upon the identification $D\cong G^*\times G$, the multiplication on $D$ is expressed in terms
of the dressing actions by
\begin{equation}\label{eq:Dgrp}
(z_1,g_1)(z_2,g_2)=(z_1({}^{g_1}z_2),(g_1^{z_2})g_2). 
\end{equation}

Consider $(\cG_{G^*},\omega_{G^*})$ as in Example \ref{ex:LuWeins}, and notice that we have a diffeomorphism
\begin{equation}\label{eq:DGG*}
D= G^*\times G \stackrel{\sim}{\longrightarrow} \cG_{G^*}, \qquad (z,g)\mapsto ({}^{g}z,g^z,g,z^{-1}). 
\end{equation}
In terms of this map, the double groupoid structure on $\cG_{G^*}$ is given by the action groupoids $G\ltimes G^*\rightrightarrows G^*$ and $G\rtimes G^*\rightrightarrows G$ defined by the dressing actions.

We will use \eqref{eq:DGG*} to endow $D$ with a 2-shifted symplectic structure.

Let $\Omega \in \Omega^2(D^2)$ be the pullback of $\omega_{G^*}$ under the map $D^2\to \cG_{G^*}$ given by the composition of the map $D^2\to D$, $((z_1,g_1),(z_2,g_2))\mapsto (z_2,g_1)$ with the diffeomorphism \eqref{eq:DGG*}. We will use the notation
$$
\Omega|_{((z_1,g_1),(z_2,g_2))}=(g_1,z_2)^*\omega_{G^*}.
$$

\begin{lem}[Proposition 4.4 in \cite{cueca}]\label{lem:Omega} The forms $\Omega\in \Omega^2(D^2)$ and $0\in \Omega^3(D)$
    define a 2-shifted symplectic form on $D$ integrating the natural pairing on $\fd=\g\oplus \g^*$.
\end{lem}

\begin{rema} The more conceptual explanation for the previous lemma is that the Artin-Mazur codiagonal (or bar) construction \cite{raj:from}, when applied to the double groupoid structure on $\cG_{G^*}$, produces the nerve of $D$ equipped with $\Omega+0$ \cite[Thm. 4.8]{cueca} and intertwines the total differentials, see \cite[Lemma 4.20]{tracou}.  \hfill $\diamond$
\end{rema}

We continue to assume that $(G,\pi_G)$ is a Poisson Lie group with a Drinfeld double $(D,\phi_G)$ and morphism $\phi_{G^*}:G^*\to D$ such that \eqref{eq:complete} is a diffeomorphism.

Suppose that we have a Poisson action of $(G,\pi_G)$ on a Poisson manifold $(M,\pi)$.
We denote the infinitesimal action by $\rho: \g_M \to TM$, and let 
$$
\mathbf{L} \subseteq \bT M \times \fd
$$
be the corresponding Dirac structure \eqref{eq:qpoisdirac}. 
We henceforth assume that
\begin{itemize}
\item  $(\cG,\omega)$ is a symplectic groupoid integrating $(M,\pi)$;
\item the $G$-action on $M$ lifts to a $G$-action on $\cG$, denoted by  $(g,\gamma)\mapsto {}^g \gamma$ for $g\in G$ , $\gamma\in \cG$;
\item the lifted action on $\cG$ admits a $G$-equivariant moment map
\[ 
\mu:\cG \rightarrow G^* 
\]
which is a Lie groupoid morphism and satisfies 
\[ \gt ({}^g\gamma)=g\gt(\gamma), \quad \gs({}^g\gamma)=g^{\mu(\gamma)} \gs(\gamma), \quad {}^g \gm(\gamma_1,\gamma_2)=\gm({}^g\gamma_1,{}^{g^{\mu(\gamma_1)}}\gamma_2). \]
\end{itemize}

It is proven  in \cite{ferigl,mommap} that  such a lifted action on $\cG$ exists whenever $\cG$ is source-simply-connected. 

In this setup, we now describe a multiplicative $D$-valued moment map integrating $\mathbf{L}$ by enhancing a construction in \cite{lupoiact}.

It was shown in \cite[$\S$9]{lupoiact} that there is a Lie groupoid $\cH=\cG \bowtie G \rightrightarrows M$ integrating the underlying Lie algebroid of $\mathbf{L}$ given as follows: the space of arrows is $\cG \times G$ with structure maps
\[ 
\gs (\gamma,g)=g^{-1}\gs_\cG(\gamma), \quad\gt(\gamma,g)=\gt_\cG(\gamma), \quad (\gamma_1,g_1)(\gamma_2,g_2)=(\gamma_1({}^{g_1}\gamma_2),(g_1^{\mu(\gamma_2)})g_2).
\]
Using \eqref{eq:Dgrp} and the properties of $\mu$, one checks that the map
$$
\Phi:\cH \to D=G^*\times G, \quad\;\; (\gamma,g)\mapsto (\mu(\gamma),g)
$$ 
is a groupoid morphism. Denote by $\pr_\cG: \cH=\cG \times G \to \cG$ the natural projection, and endow $D$ with the 2-shifted symplectic structure $-\Omega$ (from Lemma~\ref{lem:Omega}).

\begin{prop}\label{prop:P}
The morphism $\Phi: \cH\to D$ and the 2-form
$\sigma := \pr_\cG^*\omega \in \Omega^2(\cH)$ define a global multiplicative  $D$-valued moment map integrating the Dirac structure $\mathbf{L}$ associated with the Poisson action on $M$.
\end{prop}

\begin{proof} 
Recall that, since $(\cG,\omega)$ is a hamiltonian $G$-space with $G^*$-valued moment map, it is a hamiltonian space for the symplectic groupoid $(G\ltimes G^*,\omega_{G^*})$ (Example \ref{ex:hamspa}), so it satisfies
$$
\mathcal{A}^*\omega = \bar{\mu}^*\omega_{G^*} + \pr_\cG^*\omega,
$$
where $\mathcal{A}: G\times \cG\to \cG$ denotes the $G$-action  on $\cG$, $\mathcal{A}(g,\gamma)={}^g\gamma$, and $\bar{\mu}$ is the morphism $G\ltimes \cG \to G\ltimes G^*$, $\bar{\mu}(g,\gamma)= (g,\mu(\gamma))$. 
We will re-write the previous equation with the short-hand notation
$$
({}^g\gamma)^*\omega = (g,\mu(\gamma))^*\omega_{G^*} + \gamma^* \omega.
$$
Using this equation and the fact that $\omega$ is multiplicative, we can verify that $\sigma$ is an isotropic structure on $\Phi$ into $(D,-\Omega)$:
\[ 
\delta_{\cH} \sigma|_{((\gamma_1,g_1),(\gamma_2,g_2))}=\gamma_1^*\omega+\gamma_2^*\omega-
(\gamma_1({}^{g_1}\gamma_2))^*\omega = \gamma_2^*\omega - ({}^{g_1}\gamma_2)^*\omega = -(g_1,\mu(\gamma_2))^*\omega_{G^*},
\]
which says that $\delta_\cH \sigma = -\Phi^*\Omega$.

As a vector bundle, $A_\cH=\{(V,u)\in T\cG|_M\times\g\,|\, T\gs_\cG(V)=\rho(u)\}$, and we can can identify it with $A_\cG \times \g\subseteq T\cG|_M\times\g$ via the map $(V,u)\mapsto (V-\rho(u),u)$. Upon this identification, the anchor of $A_\cH$ is $\mathtt{a}_\cH(V,u) =\mathtt{a}_\cG(V) + \rho(u)$, and $\lambda_\sigma(V,u) = \lambda_\omega(V)$.
It is then  clear that the map $(\mathtt{a}_\cH,\lambda_\sigma,\Lie(\Phi))$ takes $A_{\cH}$ to the Dirac structure ${\bf L}$.
It follows that $\Phi$ is a multiplicative $D$-valued moment map integrating $\mathbf{L}$.

To verify the $G\times G$-equivariance of $\Phi$, note that 
$$
\Psi:G\ltimes M \to \cH, \quad (g,x) \mapsto (\epsilon(g\cdot x),g)
$$ 
is a groupoid morphism satisfying properties (e1) and (e2) of Def.~\ref{def:GGequiv}. Property (e3) follows from
\[ 
\Phi(\Psi(g_1,\gt(\gamma))(\gamma,g))=\Phi((\epsilon(g_1\cdot \gt(\gamma)),g)(\gamma,g))=(1,g_1)\Phi(\gamma,g), 
\]
and similarly for the right multiplication. Finally, the fact that $\gu^*\omega=0$ implies that
$\Psi^*\sigma=0$, so in particular property (e4) is verified.
\end{proof}

Following $\S$ \ref{subsec:poisquot}, we can use the global multiplicative $D$-valued moment map $\Phi: \cH\to D$ to produce quasi-symplectic groupoids associated to lagrangian morphisms $L\to D$  (Prop.~\ref{prop:integL}), as well as symplectic groupoids integrating Poisson quotients $M/H$ (Thm.~\ref{thm:reducSG}). Two simple cases of these constructions are given by $L=G^*$ and $L=G$:
\begin{itemize}
\item For $L=G^*$, the corresponding quasi-symplectic groupoid $(\cH\times_D G^*, \pr_1^*\sigma)$ recovers the symplectic groupoid $(\cG,\omega)$;
\item For $L=G$ (with the appropriate regularity conditions described in  $\S$ \ref{subsec:poisquot}), we obtain the quasi-symplectic groupoid $(\mu^{-1}(1)\times G, \pr_1^*\iota^*\omega)$, where $\iota:\mu^{-1}(1)\to \cG$ is the inclusion, and its quotient (setting $H=G$) gives the symplectic groupoid $\mu^{-1}(1)/G$ over $M/G$; this construction recovers \cite[Thm. 2]{ferigl}. 
\end{itemize}

\begin{rema} Prop.~\ref{prop:P} describes how to obtain a multiplicative $D$-valued moment map integrating $\mathbf{L}$ from a symplectic groupoid $\cG$ integrating the Poisson manifold $(M,\pi)$ equipped with a lifted hamiltonian Poisson action. As indicated by the first point above, we can use Prop.~\ref{prop:integL} to go in the opposite direction: starting with an integration $\cH\to D$ of $\mathbf{L}$, we can obtain 
a symplectic groupoid integrating  $(M,\pi)$ via the fibred product  of lagrangian morphisms $\cG:=\cH\times_D G^*$, together with its $G^*$-valued moment map (given by the projection).
This correspondence between moment maps with values in $D$ and $G^*$ can be regarded as a ``multiplicative'' refinement of the correspondence in 
\cite[$\S$ 5]{manpairev}. \hfill $\diamond$
\end{rema}

\subsection{Examples from quasi-Poisson $\g_\Delta$-manifolds}\label{subsec:gdelta}

Given a Lie algebra $\g$ with quadratic structure $\langle \cdot,\cdot\rangle_\g$,
we obtain a Lie quasi-bialgebra $(\g, F, \chi)$ with cobracket $F:\g\to \wedge^2\g$ and trivector $\chi\in \wedge^3\g$ given by 
$$
F=0, \qquad \chi(u^\vee,v^\vee,w^\vee)=\frac{1}{4}\langle[u,v],w\rangle_\g,
$$
where we use $^\vee$ for the identification $\g\cong \g^*$ via $\langle \cdot,\cdot\rangle_\g$. Its Drinfeld double $\fd=\g\oplus \g^*$ can be identified with the quadratic Lie algebra $\g\oplus \overline{\g}$ (the direct product of $\g$ with itself, but with the second summand equipped with the opposite pairing) via the map
\begin{equation}\label{eq:g+g}
\fd=\g\oplus \g^* \stackrel{\sim}{\longrightarrow} \g\oplus \overline{\g}, \quad (u,v^\vee)\mapsto \left(u+\frac{v}{2}, u -\frac{v}{2} \right).
\end{equation}
In terms of quasi Manin triples, $(\fd, \g, \g^*)$ is identified with $(\g\oplus \overline{\g}, \g_{\Delta},\mathfrak{c})$, where  $\g_\Delta$ is the diagonal lagrangian subalgebra, and $\mathfrak{c}$ is the antidiagonal lagrangian subspace.

We refer to quasi-Poisson manifolds with respect to such Lie quasi-bialgebras \cite{alemeikos} as {\em quasi-Poisson $\g_\Delta$-manifolds}. We will illustrate some examples of their integrations by multiplicative $D$-valued moment maps.

\begin{rema}[Quasi-hamiltonian groupoids]
We note that \cite{liblasev} offers an alternative integration of quasi-Poisson $\g_\Delta$-manifolds via {\em quasi-hamiltonian groupoids}. A comparison between these two approaches is presented in the subsequent paper \cite{part2}.
\end{rema}

In what follows we assume that $G$ is a Lie group integrating $\g$ such that the quadratic structure on $\g$ is $\mathrm{Ad}_G$-invariant. We will regard $G$ equipped with the canonical 2-shifted symplectic structure $\Omega_G+\Theta_G$ of Example \ref{ex:BG}, and denote by $\overline{G}$ the same Lie group equipped with the opposite symplectic structure. In this context, the relevant target for our $D$-valued moment maps is the 2-shifted symplectic Lie group $D=G\times \overline{G}$.

By equipping $G$ with the null bivector field $\pi=0$ and the $\g\times\g$-action
given by $(u_1,u_2)\mapsto u_1^r-u_2^l$, we obtain a quasi-Poisson $\g_\Delta^2$-manifold
that we denote by $(G,0)$.
We will present a multiplicative $D^2$-valued moment map integrating this quasi-Poisson manifold defined by the following data:
\begin{enumerate}
    \item The Lie groupoid $\cH=P(G)\times G^2 \rightrightarrows G$, the direct product of the pair groupoid $P(G)\rightrightarrows G$ and the Lie group $G^2$.
    \item The morphism $\Phi:\cH\to D^2$,
    \[ 
    \Phi(a_1,a_2,g_1,g_2)=((a_1g_2a_2^{-1},g_1),(a_1^{-1}g_1a_2,g_2)),\quad  \forall (a_1,a_2)\in P(G),\, (g_1,g_2)\in G^2. 
    \]
    \item The 2-form $\sigma \in \Omega^2(\cH)$,
    \[ 
    \sigma|_{(a_1,a_2,g_1,g_2)}=(a_1,a_1^{-1}g_1a_2)^*\Omega_G-(g_1,a_2)^*\Omega_G-(a_1,g_2)^*\Omega_G+(a_1g_2a_2^{-1},a_2)^*\Omega_G. 
    \]
\end{enumerate}

\begin{prop}\label{pro:integration of G,0} We have that $\Phi:(\cH,\sigma)\to D^2$ is a global multiplicative $D^2$-valued moment map that integrates the quasi-Poisson $\g_\Delta^2$-manifold $(G,0)$. 
\end{prop}
\begin{proof} The condition $d\Omega_G+\delta \Theta_G=0$ on the nerve of $G$ implies that
\begin{equation*}\begin{aligned}
    d\sigma |_{(a_1,a_2,g_1,g_2)}= &-(a_1^{-1}g_1a_2)^*\Theta_G-a_1^*\Theta_G+g_1^*\Theta_G+a_2^*\Theta_G+a_1^*\Theta_G+g_2^*\Theta_G-(a_1g_2a_2^{-1})^*\Theta_G-a_2^*\Theta_G\\
    =&-\Phi^*((\Theta_G,-\Theta_G),(\Theta_G,-\Theta_G)),
\end{aligned}\end{equation*}
as desired. Similarly, the condition $\delta\Omega_G=0$ on the nerve of $G$ yields that 
$$\delta \sigma=\Phi^*((\Omega_G,-\Omega_G),(\Omega_G,-\Omega_G))$$ on the nerve of $\cH$. Hence $\sigma$ is an isotropic structure on $\Phi$. 

For $(X, (u_1,u_2))\in A_\cH|_a = T_aG \times \g^ 2$, note that
\begin{eqnarray}
    \lambda_{\sigma}(X, (u_1,u_2))&=-\langle u_2,\theta^l\rangle+\langle u_1,\theta^r\rangle-\frac{1}{2}\langle \theta^r(X),\theta^r\rangle-\frac{1}{2}\langle \theta^l(X),\theta^l\rangle, \label{eq:lambda sigma}\\
    \varphi(X, (u_1,u_2))&=((\Ad_a u_2+\theta^r(X),u_1),(-\theta^l(X)+\Ad_{a^{-1}}u_1,u_2)), \label{eq:varphi sigma}
\end{eqnarray}
where $\varphi=\mathrm{Lie}(\Phi)$.
So we immediately see that the map \eqref{eq:j} is injective in this case, and, by a dimension count, its image is a Dirac structure 
$$
\mathbf{L}\subseteq \bT G \times \fd^2;
$$
this shows that $\sigma$ is a lagrangian structure on $\Phi$. 

Using \eqref{eq:lambda sigma} and \eqref{eq:varphi sigma},  
we obtain an injective map 
\begin{equation}\label{eq:injla}
(\g\times \g)_G  \to {\bf L},\qquad (u_1,u_2) |_a\mapsto ((u_1^r-u_2^l)|_a,0,((u_1,u_1),(u_2,u_2))).
\end{equation}
It follows that $\mathrm{rank}({\bf L}\cap( TG \times \g^2_\Delta))\geq 2\dim(\g)$; on the other hand, note that $ \lambda_\sigma:{\bf L}\to T^*G$ is fibrewise surjective (it has full rank when restricted to the space of vectors of the form $(X,(0,0))$) and ${\bf L}\cap( TG \times \g^2_\Delta)$ lies in the kernel of this map, hence $\mathrm{rank}({\bf L}\cap( TG \times \g^2_\Delta))\leq 2\dim(\g)$, and therefore equality holds. As a consequence, the previous map 
$(\g\times \g)_G  \to {\bf L}$ is an isomorphism onto ${\bf L}\cap( TG \times \g^2_\Delta)$, and this assures that  $\bf L$ makes $G$ into a quasi-Poisson manifold (as in Def.~\ref{def:qpoiL}). We now check that it is $(G,0)$.

We have a canonical isotropic complement $\fc^2$ of $\g_\Delta^2$ in $\fd^2$, where
$\fc \subseteq \fd$ is the antidiagonal subspace, which defines a quasi-Poisson bivector field $\pi$ on $G$ by the composition of relations $\mathbf{L}\circ \fc^2 = \gra(\pi^ \sharp)$. This means that $X=\pi^ \sharp(\alpha)$ if and only if there exists $(u_1,u_2)\in \g^2$ such that $\lambda_\sigma(X, (u_1,u_2))=\alpha$ and $\varphi(X, (u_1,u_2))\in \fc^2$.
By \eqref{eq:varphi sigma}, we deduce from the latter condition  that
\begin{align}\label{eq:calc biv sigma}
    \Ad_a u_2+\theta^r(X)=-u_1, \quad -\theta^l(X)+\Ad_{a^{-1}} u_1=-u_2;
\end{align} 
applying $\mathrm{Ad}_a$ to the second equation and comparing with the first one, we see that $\theta^r(X)=0$, which implies that $\pi=0$. Therefore $\mathbf{L}$ is indeed the Dirac structure corresponding to the quasi-Poisson $\g_\Delta^2$-manifold $(G,0)$.

We finally check the $(G\times G)^2$-equivariance of $\Phi$, as in Def.~\ref{def:GGequiv}. 
The action groupoid $G^2\ltimes G \rightrightarrows G$ corresponding to the action $(g_1,g_2)\cdot a=g_1ag_2^{-1}$ embeds in $\cH$ via the map $\Psi((g_1,g_2),a) = (g_1ag_2^{-1},a,g_1,g_2)$. We can show through direct calculations that $
\Psi$ is an integration of the Lie-algebroid morphism $(\g\times\g)\ltimes G\to A_\cH$ from \eqref{eq:injla} and that $\Psi^*\sigma=0$. \
\end{proof}

\begin{rema}[Representation varieties]
The quasi-Poisson manifold $(G,0)$ arises in the context of representation varieties \cite{alemeikos,quisur}, where it corresponds to a disc with two marked points.
As we will show in the upcoming work \cite{part2}, the lagrangian morphism $\Phi$ in Proposition \ref{pro:integration of G,0} may be used as a building block to obtain integrations of all quasi-Poisson structures of representation varieties. \hfill $\diamond$
\end{rema}

We give two applications of the multiplicative moment map $\Phi: \cH\to D^2$ integrating $(G,0)$.

\begin{exa}[The Lu-Weinstein groupoid revisited]\label{ex:luwei revisited} Let $G$ be a complex reductive Lie group. Let $B,B_-\subset G$ be a pair of opposite Borel subgroups with intersection the maximal torus $T$. With these choices, $G$ becomes a Poisson Lie group with the so-called {\em standard Poisson structure} \cite{dri:tria, sem:what}.

Denoting by $\pr_{T}:B_\pm\to T$ the projections,
we have the following distinguished lagrangian subgroup of $G\times \overline{G}$:
\begin{align*}
    G^*:=\{(b,b')\in B\times B_-\ |\ \pr_{T}(b)\pr_{T}(b')=1 \}, 
\end{align*}
which also appeared in Example \ref{ex:gauss-dirac}. 
Consider the lagrangian subgroup $L=G^*\times (G^*)^\vee \subseteq D^2$, where $(G^*)^\vee=\{(a_1,a_2)\ |\ (a_2,a_1)\in G^*\}$. 
By Prop.~\ref{prop:integL}, we obtain a quasi-symplectic groupoid as the fibred product 
\[
\cG_L=\cH\times_{D^2} L.
\] 
This is in fact a symplectic groupoid integrating the standard Poisson structure on $G$ -- it is isomorphic to the (opposite) Lu-Weinstein groupoid of Example \ref{ex:LuWeins}.
(This symplectic groupoid also agrees, up to a covering map, with the one constructed in \cite[Prop. 7]{quahammer} as the moduli space of appropriate connections on the sphere.)
\hfill $\diamond$
\end{exa}

\begin{exa}[Integration of homogeneous spaces]\label{ex:homsp} Let $Q\subset G$ be a closed connected coisotropic subgroup. The null bivector field makes $G/Q$ into a quasi-Poisson $\g_\Delta$-manifold with respect to the left $G$-action, see e.g. \cite[Example 1]{liblasev}. For example, if $G$ is as in the previous example and $Q$ is a parabolic subgroup, then $G/Q$ is a partial flag variety.

We denote by ${\bf L}$  the Dirac structure corresponding to $(G,0)$ and by ${\bf L}_Q$ the Dirac structure corresponding to $G/Q$. Consider the connected integration $Q^\perp\subset Q$ of the Lie algebra ideal $\mathfrak{q}^\perp\subset \mathfrak{q}$, which is then a normal subgroup, and the lagrangian subgroup $L\subset D$ given by
\[ L=\{(gc,g)\in D\ | \ (g,c)\in Q\times Q^\perp \}. 
\]
Following Remark \ref{rem:redpar}, we have a lagrangian morphism
$D\times L \to (D\times \overline{D})\times \overline{D}$, and the reduction $\Phi^{-1}( D\times L)/Q^2$ inherits a global multiplicative $D$-valued moment map, which one can check to be an integration of ${\bf L}_Q$. 
(This space can be viewed as a moduli space of logarithmic connections on a disc with a single pole as in \cite{boapar}; details will be provided in \cite{part2}.)
\hfill $\diamond$
\end{exa}

\appendix

\section{The Weil algebra and the van Est map}\label{app:wei}
Here we review the definition of the Weil algebra of a Lie algebroid and the van Est map \cite{weivanest, cabdru, lib:VE}. 

Let $(A\to M, [\cdot,\cdot], \mathtt{a})$ be a Lie algebroid. Following  \cite{weivanest}, the Weil algebra of $A$, denoted by $W(A)=\bigoplus_{p,q} W^{p,q}(A)$, is the bigraded differential algebra generated by the antisymmetric $\mathbb{R}$-multilinear maps
\[ (c_0,c_1,\dots,c_p) \in W^{p,q}(A),\quad c_i:\Gamma (A)^{p-i} \rightarrow \Omega^{q-i}(M;S^i(A^*)), \quad i=0\dots p, 
\]  
subject to the conditions
\[ c_i(u_1,u_2,\dots, fu_{p-i})=f c_i(u_1,u_2,\dots u_{p-i}) -df \wedge i_{u_{p-i}} c_{i+1}(u_1,u_2,\dots u_{p-i-1}), \quad \forall i=0\dots p-i, \] 
for all $u_i\in \Gamma (A)$ and all $f\in C^\infty(M)$, where $i_{u_{p-i}}$ denotes the contraction with the section ${u_{p-i}}$; in what follows, we will use the notation 
$$
i_{v} c_{i+1}(u_1,\dots u_{p-i-1})=c_{i+1}(u_1,\dots u_{p-i-1}|v).
$$ 
Note that, by the properties of the $c_i$'s, in order to define the differentials
\[ d^v:W^{p,q}(A) \rightarrow W^{p,q+1}(A), \quad d^h: W^{p,q}(A) \rightarrow W^{p+1,q}(A) \]
on a sequence $(c_0,c_1,\dots,c_p) \in W^{p,q}(A)$, it is enough to specify them on $c_0,$ 
\begin{align*} 
(d^vc)_0(u_1,\dots,u_p) &=(-1)^p d (c_0(u_1,\dots,u_p)); \\
 (d^hc)_0(u_1,\dots,u_{p+1})&=\sum_{i}(-1)^{i+1} \pounds_{\mathtt{a}(u_i) } c_0(u_1,\dots,\widehat{u}_i,\dots, u_{p+1}) + \\
 &+\sum_{i<j}(-1)^{i+j}c_0([u_i,u_j],\dots,\widehat{u}_i,\dots,\widehat{u}_j,\dots,  u_{p+1}), 
 \end{align*}
 since for the higher terms we have the following identities:
 \begin{equation*}
    (d^vc)_k(u_1,\dots,u_{p-k}|v) =(-1)^{p-k} d (c_k(u_1,\dots,u_{p-k}|v))+c_{k-1}(u_1,\dots,u_{p-k},v|v),
 \end{equation*}
  \begin{equation*}
    (d^hc)_k(u_1,\dots,u_{p-k+1}|v) = d_{CE} (c_k)(u_1,\dots,u_{p-k+1}|v)+ (-1)^{p-k} i_{\ga(v)}c_{k-1}(u_1,\dots,u_{p-k+1}|v).
 \end{equation*}

Another description of the Weil algebra of $A$ directly in terms of generators and relations, along with natural interpretations of the differentials $d^v$ and $d^h$, can be found in \cite{MeinPike}.

\begin{exa}\label{ex:weilla}
The Weil algebra of a Lie algebra $\mathfrak{g}$ is $W^{p,q}(\mathfrak{g})=\wedge^{p-q}\mathfrak{g}^*\otimes S^{q} \mathfrak{g}^*$, where $d^v$ takes generators of $\wedge^1\mathfrak{g}^*$ to the corresponding generators in $S^1\mathfrak{g}^*$, and $d^h$ is defined by the Chevalley-Eilenberg differential with coefficients on symmetric powers of the coadjoint representation. In particular, any $\kappa \in S^q\g^*$ is $d^v$-closed, and it is $d^h$-closed if and only if it is coadjoint invariant, see e.g. \cite{equder} for details.
\hfill $\diamond$
\end{exa}

The construction of the Weil algebra of a Lie algebroid is functorial, in the sense that any Lie algebroid map $\varphi: A \to B$ naturally lifts to a morphism
\begin{equation}\label{eq:lift}
\widehat{\varphi}: W^{\bullet,\bullet}(B)\to W^{\bullet,\bullet}(A)
\end{equation}
that is compatible with the differentials, see e.g. \cite{lib:VE}.
 
 \begin{rema}\label{rem:super}
    The Weil algebra has a natural interpretation in terms of supergeometry \cite{raj:sup}. From this perspective a Lie algebroid  $(A\to M, [\cdot,\cdot], \mathtt{a})$ is a degree $1$ $Q$-manifold $(A[1], Q=d_{CE})$ and the Weil algebra is identified with  the algebra of differential forms of the $Q$-manifold, i.e.   $(W^{\bullet,\bullet}(A), d^h, d^v)=(\Omega^{\bullet,\bullet}(A[1]), \pounds_Q, d_{dR})$, so the lift of Lie algebroid maps \eqref{eq:lift} corresponds to the pullback operation on forms. 
    \hfill $\diamond$
 \end{rema}

We will be interested in the following particular situation.

\begin{exa}\label{ex:weil}
Let $\g$ be a Lie algebra, and $A$ a Lie algebroid.
Given  a morphism $\varphi: A\to \mathfrak{g}$, its lift  
$$
\widehat{\varphi}: \wedge^{p-q}\mathfrak{g}^*\otimes S^{q} \mathfrak{g}^* \to W^{p,q}(A)
$$ 
is uniquely determined by the property that, for $\xi\in \wedge^1\mathfrak{g}^*$, 
$(\widehat{\varphi}(\xi))_0: \Gamma(A)\to C^\infty(M)$, is given by $u\mapsto \xi(\varphi(u))$. A direct calculation then shows that for $\kappa \in S^2\mathfrak{g}^*$, $\widehat{\varphi}(\kappa)=(\tau_0,\tau_1,\tau_2) \in W^{2,2}(A)$ is given by 
\begin{align*}
&\tau_0: \Gamma(A)^2\to \Omega^2(M), \quad \tau_0(u_1,u_2) = d (\kappa(d \varphi(u_1), \varphi(u_2))),\\
& \tau_1: \Gamma(A)\to \Omega^1(M; A^*), \quad \tau_1(u | v) = - \kappa(d \varphi(u), \varphi(v)),\\
& \tau_2 \in \Gamma(S^2A^*), \quad  \tau_2= \varphi^*\kappa.
\end{align*}\vspace{-10mm}

\hfill $\diamond$
\end{exa}

The following result is a generalization of \cite[Prop.~6.2]{weivanest} and will be used to prove Theorem~\ref{thm:intiso}.

\begin{lem}\label{lem:W12iso}
Let $\mathfrak{g}$ be a Lie algebra and $\varphi: A\to \mathfrak{g}$ a Lie algebroid map.
Let $\kappa\in S^2\mathfrak{g}^*$, 
$c\in W^{1,2}(A)$ and $\eta\in \Omega^3(M)=W^{0,3}(A)$. Then $(d^v+d^h)(c+\eta) = \widehat{\varphi}(\kappa)$ if and only if the following holds:
\begin{itemize}
\item[(i)] $d\eta=0$;
\item[(ii)] $c_0(u)= i_{\mathtt{a}(u)}\eta-d c_1(|u)$, for all $u\in \Gamma(A)$;
\item[(iii)] for all $u, v \in \Gamma(A)$,
\begin{align*} 
         c_1(|[u,v])=&\pounds_{\mathtt{a}(u) }c_1(|v)-i_{\mathtt{a}(v) }dc_1(|u)+i_{\mathtt{a}(v) }i_{\mathtt{a}(u) }\eta+\kappa( d\varphi(u), \varphi(v)), \\ 
        -\varphi^*\kappa(u,v)= &  i_{\mathtt{a}(u)}c_1(|v)+ i_{\mathtt{a}(v)}c_1(|u). 
    \end{align*}
\end{itemize}

\begin{proof} By considering bidegrees, one sees that $(d^v+d^h)(c+\eta) = \widehat{\varphi}(\kappa)$ is equivalent to
$$\text{(i')}\ d^v\eta=0,\quad \text{(ii')}\ d^h\eta=-d^vc\quad \text{and}\quad \text{(iii')}\ d^hc=\widehat{\varphi}(\kappa).
$$
Conditions (i) and (i') are the same. 

Writing (ii') in terms of its components in $W^{1,3}(A)$, 
$(d^h\eta)_i=-(d^vc)_i$, $i=0,1$, we obtain
$$
\pounds_{\ga(u)}\eta = d(c_0(u)) \quad \mbox{ and } \quad d(c_1(|u))+ c_0(u)= i_{\ga(u)}\eta.
$$
Assuming that $d\eta=0$, we see that the first equation follows from the second, in which case (ii) and (ii') are equivalent.

Writing (iii') componentwise, $(d^hc)_i=(\widehat{\varphi}(\kappa))_i$ for $i=0,1,2$, we obtain (see Example~\ref{ex:weil})
\begin{itemize}
\item $d(\kappa( d\varphi(u), \varphi(v))) = \pounds_{\mathtt{a}(u) }c_0(v)- \pounds_{\mathtt{a}(v) }c_0(u)-c_0([u,v])$,
\item $- \kappa( d\varphi(u), \varphi(v)) = \pounds_{\mathtt{a}(u) }c_1(|v)+ i_{\mathtt{a}(v)} c_0(u) - c_1(|[u,v])$,
\item $\varphi^*\kappa (u,v) = - i_{\mathtt{a}(u)}c_1 (|v)-  i_{\mathtt{a}(v)}c_1 (|u) $.
\end{itemize}
If (ii) holds, then the last two equations coincide with the ones in (iii), and if (i) and (ii) hold, then the first equation follows from the second one (by taking exterior derivatives). 
Therefore, assuming that (i) and (ii) hold, (iii) and (iii') are equivalent.

Therefore conditions (i). (ii) and (iii) are equivalent to (i'), (ii') and (iii'), which concludes the proof.
\end{proof}
\end{lem}

The van Est map $\mathrm{VE}:(\widehat{\Omega }^q(\mathcal{G}_{(p)} ), \delta, d) \rightarrow (W^{p,q}(A_\cG), d^h, d^v)$  is a morphism of double complexes that relates the Bott-Shulman-Stasheff complex of a Lie groupoid $\gro{G}{M} $ with the Weil algebra of its Lie algebroid $A_\cG\to M$ (see \cite{weivanest,lib:VE, raj:sup}), i.e. on $\widehat{\Omega }^q(\cG_{(p)})$ we get
 \[ 
\mathrm{VE} \circ d=(-1)^p d^v \circ \mathrm{VE}\quad\text{and}\quad \mathrm{VE}\circ  \delta =d^h\circ \mathrm{VE}.   
\]
The van Est map is also functorial (see \cite{lib:VE}): for a Lie groupoid morphism $\Phi: \cH\to \cG$,  
\begin{equation}\label{eq:VEfunct}
\mathrm{VE}\circ \Phi^*= \widehat{\mathrm{Lie}(\Phi)}\circ \mathrm{VE}.
\end{equation}

\section{Courant algebroids and Manin pairs}\label{app:coualg}

In this Appendix we collect basic facts about Courant algebroids and Dirac structures \cite{liuweixu} that are used in the main body of the article.

\subsection{Courant algebroids and Dirac structures}\label{app:c1}

A {\em Courant algebroid} $(E\to M, \langle\cdot,\cdot\rangle, \mathtt{a}, \llbracket\cdot,\cdot \rrbracket)$ consists of a vector bundle $E\rightarrow M$, a symmetric and non-degenerate pairing $\langle\cdot,\cdot \rangle: E\otimes E\to \mathbb{R}_M$, an anchor map $\mathtt{a}:E\rightarrow TM$ and an $\bR$-bilinear bracket $\llbracket\cdot,\cdot \rrbracket:\Gamma(E)\times \Gamma(E)\rightarrow \Gamma(E)$, referred to as the {\em Courant bracket}, satisfying 
\begin{itemize}
\item $\llbracket e_1, f e_2\rrbracket = f \llbracket e_1, e_2\rrbracket + (\pounds_{\mathtt{a}(e_1)}f)e_2$,
\item $\llbracket e_1,\llbracket e_2,e_3\rrbracket\rrbracket=\llbracket\llbracket e_1,e_2\rrbracket,e_3\rrbracket+\llbracket e_2,\llbracket e_1,e_3\rrbracket\rrbracket$,  
\item $\pounds_{\mathtt{a}(e_1)}\langle e_2,e_3 \rangle =\langle \llbracket e_1,e_2\rrbracket,e_3 \rangle + \langle e_2, \llbracket e_1,e_3\rrbracket \rangle$, 
\item $\mathtt{a}(\llbracket e_1, e_2\rrbracket)= [\mathtt{a}(e_1), \mathtt{a}(e_2)]$,
\item $\llbracket e_1,e_2\rrbracket+\llbracket e_2,e_1\rrbracket=\cD \langle e_1,e_2 \rangle$,
\end{itemize}
	for all $f\in C^\infty(M)$, $e_i \in \Gamma(E)$, where $\cD:C^\infty(M)\rightarrow \Gamma(E)$ is the map defined by $\langle\cD f, e\rangle=\pounds_{\mathtt{a}(e)}f$.
It follows from these properties that
\begin{equation}\label{eq:complex}
0 \to T^*M \stackrel{\mathtt{a}^*}{\to} E^*\cong E \stackrel{\mathtt{a}}{\to} TM \to 0 
\end{equation}
is a complex. A Courant algebroid is called {\em exact} when the previous sequence is exact.

We denote by $\overline{E} $ the Courant algebroid $E$ with the opposite symmetric pairing.

A {\em Dirac structure} in a Courant algebroid $E$ is a subbundle ${\bf L}\subset E$ that is
\begin{itemize}
\item  lagrangian: $\mathbf{L}^\perp=\mathbf{L}$ (equivalently, $\langle{\bf L},{\bf L}\rangle=0$, and $\rank {\bf L}=\frac{1}{2}\rank E$), and 
\item involutive: $\llbracket \Gamma({\bf L}),\Gamma({\bf L})\rrbracket\subset \Gamma({\bf L})$.
 \end{itemize}
 We refer to $(E,{\bf L})$ as a \emph{Manin pair}.

 \begin{exa}[Quadratic Lie algebras] \label{ex:qla}
  Courant algebroids over a point are the same as quadratic Lie algebras $(\fk, \langle \cdot,\cdot\rangle)$, i.e., Lie algebras endowed with a non-degenerate and ad-invariant symmetric pairing. Dirac structures are lagrangian Lie subalgebras.
 \hfill $\diamond$
 \end{exa}

 \begin{exa}[Exact Courant algebroids]\label{ex:twistedCourant}
 On a manifold $M$, any closed 3-form $\eta\in\Omega^3(M)$ gives rise to a Courant algebroid $\bT_\eta M=(TM\oplus T^*M, \langle\cdot,\cdot\rangle, \pr_{TM},\llbracket \cdot,\cdot\rrbracket_\eta )$, where the pairing is defined by the canonical pairing between vectors and covectors, and the bracket is the so-called {\em $\eta$-twisted Courant-Dorfman bracket}, 
 $$
 \llbracket X+\alpha,Y+\beta\rrbracket_\eta=[X,Y]+\pounds_X\beta-i_Yd\alpha + i_Yi_X\eta,
 $$
for $X,Y\in\mathfrak{X}(M)$, $\alpha,\beta\in\Omega^1(M)$. When $\eta=0$, this is known as the {\em standard} Courant algebroid. A Dirac structure in $\bT_\eta M$ is referred to as an {\em $\eta$-twisted Dirac structure} on $M$ \cite{sevwei}.

Any 2-form $B\in \Omega^2(M)$ defines a {\em gauge transformation} 
$$
\tau_B: TM\oplus T^*M \to TM\oplus T^*M, \quad X+\alpha \mapsto X + i_XB +\alpha
$$
that is an isomorphism of Courant algebroids 
\begin{equation}\label{eq:Bfield}
\tau_B: \bT_\eta M \stackrel{\sim}{\longrightarrow} \bT_{\eta-dB}M
\end{equation}
In particular, since $TN$ is a Dirac structure in $\bT N$, $\tau_B(TN)=\gra(B)$ is a Dirac structure in $\bT_{-dB} M$.

The Courant algebroids $\bT_\eta M$ are exact. More generally, any exact Courant algebroid $E$ is isomorphic to one of this type: one can always choose an isotropic splitting $s: TM \to E$ of \eqref{eq:complex} (i.e., a splitting of the exact sequence with isotropic image), and such a choice defines a closed 3-form $\eta$ on $M$ by
\begin{equation}\label{eq:exacteta}
\eta(X,Y,Z)= \langle \llbracket s(X), s(Y) \rrbracket, s(Z) \rangle,
\end{equation}
for $X,Y, Z \in\mathfrak{X}(M)$, in such a way that
$$
(\mathtt{a},s^*): E\to TM\oplus T^*M
$$
is an isomorphism of Courant algebroids $E\cong \bT_\eta M$. Hence a Courant algebroid $E$ is exact if and only if it is isomorphic to $\bT_\eta M$ for some closed $\eta\in \Omega^3(M)$, and any isomorphism is of the form $(\mathtt{a},\lambda): E\to TM\oplus T^*M$, where $s= \lambda^*$ is an isotropic splitting and $\eta$ is given by \eqref{eq:exacteta}. 
Two exact Courant algebroids are isomorphic if and only if they define the same class  $[\eta]\in H^3(M,\bR)$ \cite{sevlet}.
\hfill $\diamond$
  \end{exa}

There is a natural way to define the direct product of Courant algebroids. In this paper we will be mostly interested in direct products when one of the factors is a quadratic Lie algebra.

\begin{exa}\label{ex:prod}
Consider a Courant algebroid $(E\to M, \langle\cdot,\cdot\rangle_E, \mathtt{a}, \llbracket\cdot,\cdot \rrbracket_E)$, and let $(\fk,[\cdot,\cdot]_\fk,\langle \cdot,\cdot\rangle_\fk)$ be a quadratic Lie algebra. The product Courant algebroid $E \times \fk$ is defined by the vector bundle $E\oplus \fk_M$ with anchor $\mathtt{a}$ (extended trivially to $\fk_M$), pairing 
$$
\langle (e_1, w_1), (e_2,w_2)\rangle := \langle e_1, e_2\rangle_E + \langle w_1,w_2\rangle_\fk,
$$
and bracket
$$
\llbracket e_1 + w_1, e_2 + w_2\rrbracket:= \llbracket e_1, e_2 \rrbracket_E + \mathtt{a}^*\langle dw_1, w_2\rangle_\fk + (\pounds_{\mathtt{a}(e_1)}w_2-\pounds_{\mathtt{a}(e_2)}w_1+[w_1, w_2]_\fk),
$$
for $e_i\in \Gamma(E)$ and $w_i\in C^\infty(M, \mathfrak{k})$, $i=1,2$.
\hfill $\diamond$
\end{exa}

Any Dirac structure $\mathbf{L}\subset E$ carries a Lie algebroid structure by the restriction of the bracket and anchor of $E$.
Given an isotropic complement $\mathbf{C}$ for $\mathbf{L}$ (i.e., $E=\mathbf{L}\oplus \mathbf{C}$ and $\langle \mathbf{C}, \mathbf{C}\rangle =0$), one refers to $(E, \mathbf{L}, \mathbf{C})$ as a {\em quasi-Manin triple}.

Quasi-Manin triples are closely related to Lie quasi-bialgebroids \cite{kos:quasi, roycou}. 
Recall that a {\em Lie quasi-bialgebroid} is given by $(A, d_*,\chi)$, where $A\to M$ is a Lie algebroid, $$\chi\in \Gamma(\wedge^3A)\quad \text{and}\quad
d_*: \Gamma(\wedge^\bullet A) \to \Gamma(\wedge^{\bullet+1} A)$$ is a derivation of the Gerstenhaber algebra of $A$  such that $d_*^2=[\chi,\cdot]$ and $d_*\chi=0$; we note that $d_*$ is equivalent to a vector bundle map $\rho_*: A^*\to TM$ and cobracket $F:\Gamma(A)\to \Gamma(A)\wedge\Gamma(A)$ via
$$
\pounds_{\rho_*(\xi)}f = d_*f(\xi),\qquad
F^*(\xi_1,\xi_2)(u)= \pounds_{\rho_*(\xi_1)}(\xi_2(u))- \pounds_{\rho_*(\xi_2)}(\xi_1(u)) - d_*u(\xi_1,\xi_2), 
$$
for $f\in C^\infty(M)$, $\xi, \xi_1, \xi_2 \in \Gamma(A^*)$, $u\in \Gamma(A)$. A Lie quasi-bialgebroid induces a bivector field $\pi$ on $M$ by the condition $\pi^\sharp = \rho \circ (\rho_*)^*$ that satisfies
\begin{equation}\label{eq:qpoiunits}
\frac{1}{2}[\pi,\pi] = \rho(\chi), \qquad \pounds_{\rho(u)}\pi = \rho(d_*u),
\end{equation}
for all $u\in \Gamma(A)$.

If $(A, d_*, \chi)$ is a Lie quasi-bialgebroid, the direct sum $E:=A\oplus A^*$ has a natural Courant algebroid structure, referred to as its {\em Drinfeld double} \cite{liuweixu,roycou}, for which $A$ is a Dirac structure and $A^*$ is an isotropic complement, so we have a quasi-Manin triple. Conversely, for a quasi-Manin triple $(E, \mathbf{L}, \mathbf{C})$, using the pairing on $E$ to identify $\mathbf{C}\cong \mathbf{L}^*$, $\mathbf{L}$ inherits the structure of a Lie quasi-bialgebroid with dual anchor $\mathtt{a}|_{\mathbf{L}^*}: \mathbf{L}^*\to TM$, and cobracket $F: \Gamma(\mathbf{L})\to \Gamma(\mathbf{L})\wedge \Gamma(\mathbf{L})$ and $\chi \in \Gamma(\wedge^3\mathbf{L})$ given by
$$
i_{\xi_2}i_{\xi_1}F(e)= \langle \llbracket \xi_1, \xi_2 \rrbracket, e  \rangle, \qquad \chi(\xi_1,\xi_2,\xi_3)=  \langle \llbracket \xi_1, \xi_2 \rrbracket, \xi_3  \rangle,
$$
for $e\in \Gamma(\mathbf{L})$, $\xi_1, \xi_2, \xi_3 \in \Gamma(\mathbf{L}^*)$. Note that $\chi=0$ if and only if $\mathbf{C}$ is a Dirac structure, in which case $(E, \mathbf{L}, \mathbf{C})$ is called a {\em Manin triple} and corresponds to a {\em Lie bialgebroid}.

When $M$ is a point, Manin pairs $(\fd,\g)$ consist of a quadratic Lie algebra $\fd$ with a lagrangian Lie subalgebra $\g$. In this case we have an equivalence between (quasi-)Manin triples $(\fd, \g, \mathfrak{c})$ and Lie (quasi-)bialgebras $(\g, F, \chi)$, see \cite{kos:quasilb}.
The quadratic Lie algebra corresponding to a Lie quasi-bialgebra, i.e., its Drinfeld double, is given by $\fd=\g\oplus \g^*$, with its canonical symmetric pairing, and bracket
$$
[u+\xi, u'+\xi']_\fd :=[u,u']+i_{\xi}F(u')-i_{\xi'}F(u)+i_{\xi'}i_{\xi}\chi+F^*(\xi, \xi')+\ad_u^*(\xi')-\ad_{u'}^*(\xi).
$$

\subsection{Courant morphisms}\label{subsec:cmorph}
Given a Courant algebroid $(E\to M, \langle\cdot,\cdot\rangle, \mathtt{a}, \llbracket\cdot,\cdot \rrbracket)$ and a submanifold $N\hookrightarrow M$, a {\em Dirac structure with support on $N$}   is a lagrangian subbundle $R \subset E|_N$ over $N$ such that
$\mathtt{a}(R)\subseteq TN$ and, for any sections $e_1, e_2 \in \Gamma(E)$ such that $e_1|_N, e_2|_N \in \Gamma(R)$, then $\llbracket e_1, e_2 \rrbracket|_N \in \Gamma(R)$.

 A {\em Courant morphism} from $E_1\to M_1$ to $E_2\to M_2$ over a map $\phi:M_1\to M_2$,
denoted by 
$$ 
R: E_1\dashrightarrow E_2,
$$
is a Dirac structure in $E_1 \times \overline{E}_2$ with support on the graph of $\phi$,
 $$
 R \subset (E_1 \times \overline{E}_2)|_{\gra(\phi)}.
 $$  
A Courant morphism may or may not be given by (the graph of) an actual map $E_1\to E_2$. Courant morphisms can be composed through the pointwise composition of relations.

A {\em morphism of Manin pairs} \cite[$\S$ 2.3]{buriglsev} from $(E_1,{\bf L}_1)$ to $(E_2,{\bf L}_2)$ 
is a Courant morphism $R: E_1 \dashrightarrow E_2$ over a map $\phi: M_1\to M_2$
 such that the projection $R\cap (\mathbf{L}_1 \times E_2) \to E_2$ is a fiberwise isomorphism onto $\mathbf{L}_2\subset E_2$; equivalently, the projection
$R\cap (\mathbf{L}_1 \times \mathbf{L}_2) \to \mathbf{L}_2$ is a fiberwise isomorphism.

\begin{rema}\label{rem:LAinters}
Note that if $R\cap ({\bf L}_1 \times {\bf L}_2)$ has constant rank, then it has a natural Lie-algebroid structure, as a Lie subalgebroid of ${\bf L}_1 \times {\bf L}_2$ with base $\gra(\phi)$.
\end{rema}

\begin{exa}[Dirac maps]\label{ex:dirmaps}
Any smooth map $\phi: N\to M$ gives rise to a Courant morphism $R_\phi$ from $\bT N$ to $\bT M$ over $\phi$ 
by
$$
R_\phi:= \{((Y,(T\phi)^*(\alpha)),(T\phi(Y),\alpha))\,|\, \, Y\in T_xN, \, \alpha \in T_{\phi(x)}^*M, \, x\in N\}.
$$
More generally, for a closed $\eta \in \Omega^3(M)$, $R_\phi$ is a Courant morphism from
$\bT_{\phi^*\eta}N$ to $\bT_\eta M$.
It follows from \eqref{eq:Bfield} that, given $\sigma \in \Omega^2(N)$ such that $d\sigma = \phi^*\eta$,
$$
R_{\phi,\sigma}:=\tau_{(\sigma,0)}(R_\phi)=\{((Y,\beta),(X,\alpha))\,|\, ((Y,\beta - i_Y\sigma),(X,\alpha)) \in R_\phi \}
$$
is a Courant morphism from $\bT N$ to $\bT_\eta M$.

Consider Dirac structures $\mathbf{L}_1\subset \bT_{\phi^*\eta}N$ and $\mathbf{L}_2 \subset \bT_\eta M$, and the projection $p: R_\phi\cap (\mathbf{L}_1 \times \mathbf{L}_2) \to \mathbf{L}_2$. Then
\begin{itemize}
\item $p$ is fiberwise surjective if and only if, for all $x\in N$, 
$$
\mathbf{L_2}|_{\phi(x)} = \{ (T\phi(Y),\alpha) \,|\, (Y, (T\phi)^*\alpha) \in \mathbf{L}_1 |_x   \},
$$
i.e., $\phi$ is a {\em (forward) Dirac map}; 
\item $p$ is a fiberwise isomorphism (i.e., $R_\phi$ is a morphism of Manin pairs) if and only if $\phi$ is a {\em strong Dirac map}, that is,  a Dirac map satisfying $\ker(T\phi)\cap \mathbf{L}_1 = \{0\}$.
\end{itemize}

In particular, given a Dirac structure $\mathbf{L}$ in $\bT_{\eta}M$ and $\sigma \in \Omega^2(N)$ satisfying $d\sigma=\phi^*\eta$, the fact that $\phi: (N,-\sigma) \to (M,\mathbf{L})$ is a strong Dirac map (called {\em presymplectic realization} in \cite{burint}) is equivalent to  $R_\phi$ being a morphism of Manin pairs from $(\bT_{d\sigma}N, \gra(-\sigma))$ to $(\bT_\eta M, \mathbf{L})$.

Noticing that $R_{\phi,\sigma}\cap (TN \times {\bf L}) = R_\phi \cap (\tau_{-\sigma}(TN)\times {\bf L})$, we can alternatively express the fact that $\phi: (N,-\sigma) \to (M,\mathbf{L})$ is a strong Dirac map by the condition that $R_{\phi,\sigma}$
is a morphism of Manin pairs from $(\bT N,TN)$ to $(\bT_\eta M, \mathbf{L})$.
\hfill $\diamond$
\end{exa}

\begin{exa}[Quasi-Poisson actions as morphisms of Manin pairs]\label{ex:quasipact} 
A {\em quasi-Poisson action} of a Lie quasi-bialgebroid $(A\to N, d_*, \chi)$ is a Lie algebroid action $\rho_M$ of $A$ on $\mu: M\to N$ and a bivector field $\pi$ on $M$ such that
$$\frac{1}{2}[\pi,\pi]  = \rho_M(\chi),\quad \pounds_{\rho_M(u)}\pi  = \rho_M(d_* u) \;\ \forall u\in \Gamma(A)\quad\text{and}\quad \pi^\sharp \circ \mu^* = \rho_M \circ (\rho_*)^*.$$
Note that any Lie quasi-bialgebroid defines a quasi-Poisson action on its base, see \eqref{eq:qpoiunits}.

Following \cite[$\S$ 3.2]{buriglsev},
let $E:= A\oplus A^*$ be the Courant algebroid defined by $(A, d_*, \chi)$. Then 
\begin{equation}\label{eq:qpoisdirac2}
R:=\{ ((\rho_M(u)+ \pi^\sharp(\alpha), \alpha), (u,-\rho_M^*(\alpha)))\,|\, u\in \mu^*A, \alpha \in T^*M \} \subset \bT N \times E
\end{equation}
is a Dirac structure with support on $\gra(\mu)$ which defines a morphism of Manin pairs from $(\bT M, TM)$ to $(\overline{E}, A)$.

Conversely, consider a Manin pair $(E, \mathbf{L})$ over $N$ and a morphism of Manin pairs 
$$
R: (\bT M, TM) \dashrightarrow (\overline{E}, \mathbf{L})
$$ 
over a map $\mu: M\to N$. This induces a Lie-algebroid action of $\mathbf{L}$ on $M$ given by the map $\mu^*\mathbf{L}\to TM$ that associates to each $e\in \mathbf{L}|_{\mu(x)}$ the unique element $X\in T_xM$ such that 
\begin{equation}\label{eq:MPaction}
((X,0),e)\in R|_{(x,\mu(x))}.
\end{equation}
With respect to this action, any choice of isotropic complement $\mathbf{C}$ of $\mathbf{L}$ induces the structure of a quasi-Poisson space on $M$ for the Lie quasi-bialgebroid defined by the quasi-Manin triple $(E, \mathbf{L}, \mathbf{C})$, with bivector field $\pi$ on $M$  determined by the condition 
\begin{equation}\label{eq:biv}
\gra(\pi^\sharp) = R\circ \mathbf{C} := \{ (X,\alpha) \in \bT M\, |\, \exists \, e\in \mathbf{C} \, \mbox{ with }\, ((X,\alpha),e)\in R \}.
\end{equation}
Here ``$\circ$'' denotes composition of relations. 
\hfill $\diamond$
\end{exa}

\subsection{Action Courant algebroids}\label{subsec:appaction}

If a Lie algebra $\fk$ acts on a manifold $M$, with action $\mathtt{a}:\fk_M\to TM$, the trivial bundle $\fk_M=\fk \times M$ carries the structure of an action Lie algebroid, with anchor $\mathtt{a}$ and bracket $[\cdot,\cdot]$ on $\Gamma(\fk_M)=C^\infty(M;\fk)$ given by
$$
[u,v](x):=[u(x),v(x)]+ \pounds_{u(x)}(v)(x) - \pounds_{v(x)}(u)(x).
$$
If the Lie algebra $\fk$ has a quadratic structure (i.e., an ad-invariant, symmetric and nondegenerate pairing) and its action on $M$ has coisotropic stabilizers, then there is also a canonical Courant algebroid structure on $\fk_M=\fk \times M$ with the same anchor $\mathtt{a}$ and bracket 
$$
\llbracket u,v\rrbracket :=[u,v]+\mathtt{a}^*\langle du, v\rangle,
$$ 
for $u,v\in C^\infty(M;\fk)$, see \cite{coupoi}.
With respect to this Courant bracket, if $\fl\subset\fk$ is a Dirac structure, then so is $\fl_M\subset\fk_M$.  We will be interested in cases where the action map $\mathtt{a}$ is fibrewise surjective and the stabilizers are lagrangian, so that the resulting action Courant algebroid is exact.  Then, by Example~\ref{ex:twistedCourant}, the choice of an isotropic splitting  $s:TM\to \fk_M$ gives rise to a  closed 3-form $\eta\in \Omega^3(M)$ by \eqref{eq:exacteta} and a Courant algebroid isomorphism 
$$
\mathtt{e}=(\mathtt{a}, s^*):\fk_M\stackrel{\sim}{\to} \bT_{\eta}M.
$$

\begin{exa}\label{ex:actioncartan} Let $K$ be a Lie group with quadratic Lie algebra. Then $\fd=\fk\oplus\bar{\fk}$ acts on $K$ by $\fd\to \mathfrak{X}(K)$, $u+v\mapsto u^r-v^l$. In this case the action Courant algebroid is exact and  admits a canonical isotropic splitting
$s: TK\to \fd_K$, $X\mapsto \frac{1}{2}(\theta^r(X),-\theta^l(X))$, defining the Courant algebroid isomorphism
\begin{align}
\mathtt{e}:\fd_K\to \bT_\Theta K, \quad \mathtt{e}(u+v, k)=\left(u^r-v^l,\frac{1}{2}\langle \theta^r, u\rangle+\frac{1}{2}\langle \theta^l, v\rangle \right)\Big{|}_k,\label{eq:carcousplitting}
\end{align}
where $\Theta=-\frac{1}{12}\langle[\theta^l, \theta^l], \theta^l\rangle\in\Omega^3(K)$ is the Cartan $3$-form. 
This example is a special case of the action Courant algebroids over homogeneous spaces in Example~\ref{ex:homspa} with $D=K^2$ and $G=K$  the diagonal subgroup. 
\hfill $\diamond$
\end{exa}

\begin{exa}
Let $(G, \pi)$ be a Poisson Lie group, and let $\fd = \g\oplus \g^*$ be the Drinfeld double of its Lie bialgebra. Then the dressing action of $\fd$ on $G$, given by
$$
u+\xi \mapsto -(u^l+\pi^\sharp(\langle\theta^l, \xi\rangle)),
$$
admits a natural isotropic splitting $s: TG\to \fd_G$, $s(X)=-\theta^l(X)$, so
$$
\mathtt{e}:\mathfrak{d}_G\to \bT G, \quad \; \mathtt{e}(u+\xi, g)=-(u^l+\pi^\sharp(\langle\theta^l, \xi\rangle), \xi^l)|_g,
$$
is a Courant algebroid isomorphism \cite[$\S$ 6.1.2]{meipoi}.

\hfill $\diamond$
\end{exa}

\subsection{CA-groupoids and multiplicative Manin pairs}\label{app:mmmp}

A {\em CA-groupoid} \cite{liedir, muldir} is a VB-groupoid $\cE  \rightrightarrows E$ over a Lie groupoid $\gro{G}{M}$ such that $\cE $ carries a Courant algebroid structure which is multiplicative, in the sense that 
$\gra(\gm_{\cE})\to \gra(\gm_{\cG})$ is a Dirac structure with support inside $\cE \times \cE  \times \overline{\cE }$. A {\em morphism of CA-groupoids} is a Courant morphism,
with support on a Lie groupoid morphism, which is also a VB-subgroupoid inside the product of the CA-groupoids.
A {\em multiplicative Manin pair} $(\cE ,\mathbf{L} )$ is a Manin pair such that   $\cE \rightrightarrows E$ is a CA-groupoid and $\mathbf{L} \subset \cE $ is a  VB-subgroupoid.
A {\em morphism of multiplicative Manin pairs} is a morphism of CA-groupoids that is a 
Manin pair morphism.

\begin{exa}\label{ex:pairCA}
Given any Courant algebroid $E$, the product Courant algebroid $E\times \overline{E}$ is a CA-groupoid when regarded as a pair groupoid $E\times E \rightrightarrows E$.\hfill $\diamond$ 
\end{exa}

\begin{exa}[Exact CA-groupoids and gauge transformations]\label{ex:stacagpd} Let $\gro{G}{M}$ be a Lie groupoid. Then the canonical Courant algebroid $\bT \cG$ becomes a CA-groupoid $\bT \cG \rightrightarrows TM\oplus\al_\cG^* $ with the multiplication defined by the Dirac structure 
 \[ \mathbf{m}=\left\{((X_1,\alpha_1),(X_2,\alpha_2),(X_3,\alpha_3))\,|\,  T \mathtt{m}(X_1,X_2)=X_3, \;
\mathtt{m}^* \alpha_3 =\text{pr}_1^*\alpha_1 + \text{pr}_2^*\alpha_2     \right\}\subset \bT \cG \times \bT \cG\times  \overline{\bT \cG}, \]
supported on $\text{Graph}(\gm)$. Notice that we can twist the Courant bracket by any closed multiplicative 3-form on $\cG$  and still obtain a CA-groupoid with the same groupoid structure. 

More generally, one can consider a 2-form $\omega_2$ on $\cG_{(2)}$ and 3-form $\omega_1$ on $\cG$, normalized, such that $\omega_1+\omega_2$ is a closed 2-shifted 2-form on $\cG$. Then the condition $d\omega_2 + \delta \omega_1=0$ ensures that 
\[ \mathbf{m}_{\omega_2}=\left\{((X_1,\alpha_1),(X_2,\alpha_2),(X_3,\alpha_3))\,|\,  T \mathtt{m}(X_1,X_2)=X_3, \;
i_{(X_1,X_2)}\omega_2 = \text{pr}_1^*\alpha_1 + \text{pr}_2^*\alpha_2 -\mathtt{m}^* \alpha_3    \right\}
\]
defines a ``$\omega_2$-twisted'' VB-groupoid structure on $\bT \cG$ over $TM\oplus\al_\cG^*$ that   
is a Dirac structure in $\bT_{\omega_1} \cG \times \bT_{\omega_1} \cG\times  \overline{\bT_{\omega_1} \cG}$; see \cite[Prop. 2.3, Appendix]{tracou}. We denote this CA-groupoid,  with both the Courant bracket and VB-groupoid structure twisted, by $\CAg{\cG}{\omega_2}{\omega_1}$. Any exact CA-groupoid is isomorphic to one of this type.

Given a 2-form $B\in \widehat{\Omega}^2(\cG)$, the corresponding gauge transformation  is an isomorphism of CA-groupoids
\begin{equation}\label{eq:gaugeCAgrp}
\tau_B: \CAg{\cG}{\omega_2}{\omega_1} \stackrel{\sim}{\longrightarrow} \CAg{\cG}{\omega_2+\delta B}{\omega_1-d B}, 
\end{equation}
generalizing \eqref{eq:Bfield}.

If $\psi: \cH \to \cG$ is a morphism of Lie groupoids, then the Courant morphism $R_\psi$ (see Example~\ref{ex:dirmaps}) is a morphism of CA-groupoids from $\CAg{\cH}{\psi^*\omega_2}{\psi^*\omega_1}$ to $\CAg{\cG}{\omega_2}{\omega_1}$. Given $\sigma \in \widehat{\Omega}^2(\cH)$ such that $\partial \sigma = \psi^*(\omega_1+\omega_2)$ (i.e., $-d\sigma=\psi^*\omega_1$ and $\delta \sigma= \psi^* \omega_2$), it follows from \eqref{eq:gaugeCAgrp} that the relation 
$$
R_{\psi,-\sigma}= \tau_{(-\sigma,0)}(R_\psi)
$$
is a morphism of CA-groupoids from $\bT \cH$ to $\CAg{\cG}{\omega_2}{\omega_1}$. 
\hfill $\diamond$
\end{exa}

\begin{exa}[2-shifted symplectic groups]\label{ex:carcouca}  
Given a Lie group $K$ with quadratic Lie algebra, we saw in Example~\ref{ex:actioncartan} that the action of $\fk\oplus{\fk}$ on $K$ given by $(u,v)\mapsto u^r-v^l$ defines an exact action Courant algebroid $({\fk}\oplus\overline{\fk})_K$ over $K$. Considering the groupoid structure given by the product of the pair groupoid $\fk\times\fk\rightrightarrows \fk$ and the Lie group $K$, we have a CA-groupoid 
$$
({\fk}\oplus\overline{\fk})_K \rightrightarrows \fk,
$$
with the property that the natural projection $({\fk}\oplus\overline{\fk})_K \to {\fk}\times\overline{\fk}$ is a morphism of CA-groupoids.

Let $(K,\Omega+\Theta)$ be a 2-shifted symplectic group. Recall that $\fk$ acquires a quadratic structure (see $\S$ \ref{subsec:2sg}). 
By Example~\ref{ex:stacagpd}, $\Omega+\Theta$ defines a CA-groupoid $\CAg{K}{\Omega}{\Theta} \rightrightarrows \fk^*$, where $\Theta\in \Omega^3(K)$ twists the standard Courant bracket on $K$
and $\Omega \in \Omega^2(K\times K)$ deforms the usual groupoid structure on $\bT K$.
The source and target maps of this CA-groupoid are given by
\begin{equation}
         \begin{aligned}
             &\gs_\Omega (X+\alpha)(u)=\alpha(u^l)+\Omega_{|(k,1)}\big((X,0),(0,u)\big),\\
             &\gt_\Omega (X+\alpha)(u)=\alpha(u^r)+\Omega_{|(1,k)}\big((0,X),(u,0)\big),
         \end{aligned}\label{eq:stCA}
     \end{equation}
for $u\in \fk$ and $X+\alpha \in \bT_k K$. With the identification $\fk\cong \fk^*$ given by the pairing, one can verify that the groupoid morphism
$$
(\gt_\Omega,\gs_\Omega, \mathrm{pr}_K): \CAg{K}{\Omega}{\Theta} \to (\fk^*\times\fk^*)\times K \cong  ({\fk}\oplus\overline{\fk})_K
$$
is an isomorphism of CA-groupoids. Its inverse is given by $\mathbf{e}_\Omega:({\fk}\oplus\overline{\fk})_K\to \CAg{K}{\Omega}{\Theta}$,
\begin{equation}\label{eq:splitting exact CA}
    \begin{aligned}
     \mathbf{e}_\Omega(u+v,k)=\left(u^r_k-v^l_k,\Omega|_{(1,k)}((0,\cdot),(u,0))-\Omega|_{(k,1)}((\cdot,0),(0,v))\right)\Big .
    \end{aligned}
\end{equation}
Note that this map agrees with \eqref{eq:carcousplitting} for $\Omega+\Theta$ given by  \eqref{eq:polwie}.
 \hfill $\diamond$ 
\end{exa}

\end{document}